\documentclass[11pt]{amsart}

\usepackage{amsmath,amssymb,amscd, amsthm,mathrsfs,enumitem,url,tikz-cd}
\usepackage[hmargin=2cm,vmargin=3cm]{geometry}
\usepackage[all]{xypic}

\setlist[enumerate]{leftmargin=*}
\setlist[itemize]{labelindent=\parindent, leftmargin=*}

\numberwithin{equation}{section}

\theoremstyle{plain}
\newtheorem{thm}{Theorem}[section]
\newtheorem{lem}[thm]{Lemma}
\newtheorem{prop}[thm]{Proposition}
\newtheorem{cor}[thm]{Corollary}
\theoremstyle{definition}
\newtheorem{defn}[thm]{Definition}
\theoremstyle{remark}

\newcommand\Hom{\operatorname{Hom}}

\newcommand\disc{\mathrm{disc}}
\newcommand\GL{\mathrm{GL}}
\newcommand\Mp{\mathrm{Mp}}
\newcommand\OO{\mathrm{O}}
\newcommand\PGL{\mathrm{PGL}}
\newcommand\SL{\mathrm{SL}}
\newcommand\SO{\mathrm{SO}}

\newcommand\Spin{\mathrm{Spin}}
\newcommand\Sp{\mathrm{Sp}}

\newcommand\A{\mathbb{A}}
\newcommand\C{\mathbb{C}}

\newcommand\Q{\mathbb{Q}}
\newcommand\R{\mathbb{R}}

\newcommand\Z{\mathbb{Z}}

\newcommand{\BIGOP}[1]{\mathop{\mathchoice%
{\raise-0.22em\hbox{\huge $#1$}}%
{\raise-0.05em\hbox{\Large $#1$}}{\hbox{\large $#1$}}{#1}}}
\newcommand{\BIGboxplus}{\mathop{\mathchoice%
{\raise-0.35em\hbox{\huge $\boxplus$}}%
{\raise-0.15em\hbox{\Large $\boxplus$}}{\hbox{\large $\boxtimes$}}{\boxtimes}}}

\title[Relative Character Identities and Theta Correspondence]{Relative Character Identities \\and Theta Correspondence}
\author{Wee Teck Gan}
 
\author{Xiaolei Wan}
 \address{Department of Mathematics, National University of Singapore, 10 Lower Kent Ridge Road, Singapore 119076}
\email{matgwt@nus.edu.sg}
\email{xiaoleiwan@u.nus.edu}

\begin{document}

\maketitle

\section{\bf Introduction}
 This paper is inspired by the talk of Yiannis Sakellaridis in the Simons Symposium  held at the Schloss Elmau in April 2018.  Let us begin by describing the relevant context for his talk. 
 \vskip 5pt
 
 The study of periods of automorphic forms has been an important theme in the Langlands program, beginning with the early work of Harder-Langlands-Rapoport and Jacquet. 
 In particular, the nonvanishing of certain periods is known to characterize  the  image of certain Langlands functorial lifting and to be related to the analytic properties of certain automorphic L-functions.  An effective approach for
proving such results is the technique of relative trace formulae developed by Jacquet. Typically, such an approach involves the comparison of the geometric sides of two relative trace formulae, which results in a global spectral identity and an accompanying family of local relative character identities.  
\vskip 5pt

 In \cite{SV}, Sakellaridis and Venkatesh initiated a general framework for treating such period problems in the context of spherical varieties. In particular, to a spherical variety $X = H \backslash G$ over a local field $F$ or a global field $k$, they associated 
 \vskip 5pt
 
 \begin{itemize}
 \item a Langlands dual group $X^{\vee}$ (at least when $G$ is split), together with a canonical (up to conjugacy) map
 \[  \iota:   X^{\vee}  \times \SL_2(\C)  \longrightarrow G^{\vee}. \]
 \item a $\frac{1}{2} \mathbb{Z}$-graded finite-dimensional algebraic representation $V_X = \oplus_d  V_{X}^d$ of $X^{\vee}$, which gives rise to an L-function  
 \[ L_X(s ,\rho) =  \prod_d L(s + d, \rho, V_{X}^d) \]
for each L-parameter $\rho$ valued in $X^{\vee}$
 \end{itemize}
  They then conjectured, among other things,  that representations of $G$ (in the automorphic dual)  which have nonzero $H$-periods are those belonging to A-packets whose associated A-parameters factor through $\iota$. This means roughly that the $H$-distinguished representations of $G$ are those which are Langlands functorial lift via $\iota$ from a (split) group $G_X$ whose dual group $G_X^{\vee}$ is $X^{\vee}$. Experience shows that  it is sometimes more pertinent to regard $H$-distinguished representations of $G$ as lifted from  the {\em Whittaker variety} $(N_X, \psi) \backslash G_X$, as opposed to the group variety $G_X$ itself.
 \vskip 5pt
 
 The conjecture of Sakellaridis$-$Venkatesh can be made on several fronts. We give a brief description of the various incarnations of their conjecture (at least a first approximation), under some simplifying hypotheses and without using the language of A-parameters. Our description is adapted to the needs of this paper. For the conjecture in its most general form (taking into account Vogan L-packets for example),  the reader should consult \cite{SV}.
\vskip 5pt

\begin{itemize}
\item[(a)]   In the context of smooth representation theory of $G(F)$ over a local field $F$, one is interested in determining $\Hom_H(\pi, \C)$ for any $\pi \in {\rm Irr}(G(F))$. 
One expects (in the context of this paper) a map  
\[  \iota_*:  {\rm Irr}(G_X(F))  \longrightarrow {\rm Irr}(G(F)),  \]
 such that for any $\pi \in {\rm Irr}(G(F))$,  there is an isomorphism 
\[  f:  \bigoplus_{\sigma: \iota_*(\sigma) = \pi} \Hom_{N_X}(\sigma, \psi) \cong \Hom_H( \pi, \C). \]
In the smooth setting, the Sakellaridis-Venkatesh conjecture thus gives a precise quantitative formulation of the expectation that $H$-distinguished representations of $G$ are lifted from $G_X$. 
\vskip 5pt

If further $\iota_*$ is injective, there is at most one term on the left hand side,  and all these Hom spaces are at most one-dimensional (by the uniqueness of Whittaker models).
This will be the favourable situation encountered in this paper. 
In such instances, if $\ell \in \Hom_{N_X}(\sigma, \psi)$, with corresponding $f(\ell) \in \Hom_H( \iota_*(\sigma), \C)$, one can define relative characters $\mathcal{B}_{\sigma, \ell}$ and $\mathcal{B}_{\iota_*(\sigma), f(\ell)}$ which are certain equivariant distributions on $(N_X, \psi)\backslash G_X$ and $X$ respectively. In this case, one might expect a relative character identity relating $\mathcal{B}_{\sigma, \ell}$ and $\mathcal{B}_{\iota_*(\sigma), f(\ell)}$.  
 \vskip 10pt
 
 \item[(b)]  In the context of $L^2$-representation theory, one is interested in obtaining the spectral decomposition of the unitary representation $L^2(X)$ of $G$ (relative to a fixed $G$-invariant measure on $X$). 
 By abstract results of functional analysis, one has  a direct integral decomposition
 \[   L^2(X)  \cong \int_{\Omega}  \pi_{\omega}^{\oplus m(\omega)}   \, d\mu_X(\omega)  \]
 where
 \begin{itemize}
 \item $(\Omega, d\mu_X)$ is some measure space; 
 \item $\pi: \omega \mapsto \pi_{\omega}$ is a measurable field of  irreducible unitary representations of $G$ defined on $\Omega$, giving rise to a
 measurable map from  $\Omega$ to the unitary dual  $\widehat{G}$ of $G$;
 \item $m: \Omega \rightarrow \mathbb{N} \cup \{ \infty \}$ is a measurable multiplicity function.
 \end{itemize}
There is some fluidity in this direct integral decomposition; for example, given $\Omega$, only the measure class of $d\mu_X$ is well-defined (without explicating the isomorphism). 
 \vskip 5pt
 
 In this $L^2$-setting, the crux of the Sakellaridis-Venkatesh conjecture is to provide a canonical candidate for $(\Omega,d\mu_X, \pi)$.
 Namely, one expects a map
 \[  \iota_* : \widehat{G}_X \longrightarrow \widehat{G}  \]
 associated to $\iota$ from the unitary dual of $G_X$ to that of $G$, so that one has a (unitary)  isomorphism
 \[
 L^2(X)  \cong \int_{\widehat{G}_X}    \iota_*(\sigma)^{\oplus m(\sigma)} \, d\mu_{G_X}(\sigma), \]
 where $d\mu_{G_X}$ denotes the Plancherel measure of $G_X$ and $m(\sigma)$ is a multiplicity space which is typically isomorphic to  the dual space of 
 $\Hom_{N_X}(\sigma, \psi)$.  In other words, one may take $(\Omega, d\mu_X, \pi)$ to be $(\widehat{G}_X,  d\mu_{G_X}, \iota_*)$.
 One can think of this as saying that the spectral decomposition of $L^2(X)$ is obtained from the Whittaker-Plancherel theorem
 \[  L^2(N_X ,\psi \backslash G_X) \cong \int_{\widehat{G}_X}   \sigma^{\oplus m(\sigma)} \, d\mu_{G_X}(\sigma) \]
 by applying $\iota_*$.   One consequence of such a spectral decomposition is that it provides a canonical element $\ell_{\iota_*(\sigma)} \in \Hom_H(\iota_*(\sigma), \C)$, as we explained in \S \ref{S:bern}, for $m(-) \cdot d\mu_{G_X}$-almost all $\sigma$. Because of the presence of the Plancherel measure $d\mu_{G_X}$ of $G_X$, only tempered representations $\sigma$ in $\widehat{G}_X$ need to be considered (though $\iota_*(\sigma)$ may be nontempered). 
\vskip 5pt

\item[(c)]  Globally,  when $k$ is a global field with ring of adeles $\A$, one considers the global period integral along $H$:
\[  P_H: \mathcal{A}_{cusp}(G)  \longrightarrow \C \]
defined by
\[  P_H(\phi)  = \int_{H(k) \backslash H(\A)} \phi(h) \, dh \]
on the space of cusp forms on $G$. The restriction of $P_H$ to  a cuspidal representation $\Pi = \otimes_v \Pi_v$ of $G$ then defines an element 
$P_{H, \Pi} \in \Hom_{H(\A)} (\Pi, \C)$.  One is interested in two problems in the global setting:
\begin{itemize}
\item[(i)]  characterising those $\Pi$ for which $P_{H,\Pi}$ is nonzero as functorial lifts from $G_X$ via the map $\iota$;
\item[(ii)]  seeing if $P_{H,\Pi}$  can be decomposed as the tensor product of local functionals. 
\end{itemize}
\vskip 5pt

Such a factorization certainly exists in the instances discussed in this paper  since the local Hom spaces $\Hom_{H(F_v)}(\Pi_v, \C)$ are at most 1-dimensional for all places $v$. 
In (a), we have seen that these Hom spaces are nonzero precisely when $\Pi_v = \iota_*(\sigma_v)$ for some $\sigma_v \in {\rm Irr}(G_X(k_v))$.  
  Thus, 
in the context of the first global problem,  one would like to show that, if $P_{H, \Pi} \ne 0$,  there exists  a cuspidal representation $\Sigma$ of $G_X$ such that $\Sigma_v  \cong \sigma_v$ for all $v$, so that $\Pi \cong \iota_*(\Sigma)$. 
 
 \vskip 5pt

On the other hand,  in (b), we have remarked that the spectral decomposition of $L^2(X)$ in the local setting gives rise to a canonical basis element  $\ell_{\Pi_v} \in \Hom_{H(F_v)}(\Pi_v, \C)$.   For the second global problem, it is natural to compare the two elements $P_{H,\Pi}$ and $\prod_v^* \ell_{\Pi_v}$. Here the  asterisk in the product  indicates that  there may be a need to normalize the local functionals $\ell_{\Pi_v}$ appropriately to ensure that the Euler product $\prod_v \ell_{\Pi_v}(\phi_v)$ converges.  More precisely, to see if the Euler product converges, one would need to evaluate 
 $\ell_{\Pi_v}( \phi^0_v)$ where $\phi^0_v$ is a spherical unit vector in $\Pi_v$ for almost all $v$. This evaluation has been carried by Sakellaridis in \cite{S1, S2} and this is where the L-factor $L_X(s, - )$ associated to the $\frac{1}{2} \cdot \Z$-graded representation $V_X$ enters the picture. Namely,  it turns out that with $\Pi_v \cong  \iota_*(\Sigma_v)$ for tempered $\Sigma_v \in \widehat{G}_{X,v}$, one has:
  \[  |\ell_{\Pi_v}(\phi^0_v)|^2 = L^{\#}_{X,v}(1/2,  \Sigma_v)  := \Delta_v(0)  \cdot\frac{ L_{X,v}(1/2,  \Sigma_v) }{L(1, \Sigma_v, Ad)}> 0, \]
  where $\Delta_v(s)$ is itself a product of local L-factors which depends only on  $X$ and not on the representation $\Sigma_v$ and $L(s, \Sigma_v, Ad)$ denotes the adjoint L-factor. 
 This necessitates that one defines a normalization of $\ell_{\Pi_v}$ by:
 \[    \ell_{\Pi_v}^{\flat}  = \frac{1}{|L^{\#}_{X,v}(1/2, \Sigma_v)|^{1/2}} \cdot \ell_{\Pi_v}  \]
 Then  the main issue with the second global problem is to determine the constant $c(\Pi)$ such that
\[  |P_{H,\Pi}(\phi)|^2 =  c(\Pi) \cdot   L^{\#}_X(1/2, \Pi) \cdot \prod_v |\ell^{\flat}_{\Pi_v}(\phi_v)|^2 \quad \quad \text{for $\phi = \otimes_v \phi_v \in \otimes_v \Pi_v$.} \]
 Here the global L-function $L^{\#}_X(s, \Sigma)$ is defined by the Euler product $\prod_v L^{\#}_{X,v}(s, \Sigma_v)$ for ${\rm Re}(s) \gg 0$ and needs to be meromorphically continued so that one can evaluate it at $s = 1/2$. 
\end{itemize}
\vskip 10pt

This concludes our brief and simplified description of the Sakellaridis-Venkatesh conjecture. It is instructive to observe the crucial unifying role played by the typically ignored $L^2$-theory, which supplies the canonical basis elements in the relevant local Hom spaces for use in the factorization of the global periods. 
\vskip 10pt

We can now describe the content of Sakellaridis' lecture at the Simons Symposium.
In a series of recent papers \cite{S4,S5, S6}, Sakellaridis examined aspects of the above program in the context of rank 1 spherical varieties $X$. There is a classification of such rank 1 $X$'s, but a standard example is $X \cong \SO_{n-1} \backslash \SO_n$, i.e. a hyperboloid (or a sphere) in an n-dimensional quadratic space, and a more exotic example is $\Spin_9 \backslash F_4$.  In this rank 1 setting, 
the group $G_X$ is $\SL_2$ or its variants (such as $\PGL_2$ or ${\rm Mp}_2$). For example, for $X = \SO_{n-1} \backslash \SO_{n}$ with $n$ even, $X^{\vee} \cong \PGL_2(\C)$, so that $G_X \cong \SL_2$  and the map $\iota$ is given by:
\[ \begin{CD}
  \iota:  \PGL_2(\C) \times \SL_2(\C) @>{{\rm Sym}^2\times {\rm Sym}^{n-4}}>> \SO_3(\C) \times \SO_{n-3}(\C)  @>>> \SO_n(\C).\end{CD}  \]
On the other hand, if $n$ is odd, then $X^{\vee} \cong \SL_2(\C)$ and we take $G_X \cong \Mp_2$, with the map $\iota$ given by
\[ \begin{CD}
  \iota:  \SL_2(\C) \times \SL_2(\C) @>{{\rm Sym}^1 \times {\rm Sym}^{n-4}}>> \Sp_2(\C) \times \Sp_{n-3}(\C)  @>>> \Sp_{n-1}(\C).\end{CD}  \]
In such rank $1$ setting, Sakellaridis developed a theory of transfer of test functions from $X$ to $(N,\psi)\backslash G_X$ as a first step towards establishing local relative character identities and effecting a global comparison of the relative trace formula of $X$ and the Kuznetsov trace formula for $G_X$.  The formula for the transfer map he discovered was motivated by considering an analogous transfer for the boundary degenerations of $X$ and $(N_X, \psi) \backslash G_X$. For the hyperboloid $\SO_{n-1} \backslash \SO_n$, the boundary degeneration is simply the cone of nonzero null vectors in the underlying quadratic space.   In any case, the transfer map he wrote down differs from the typical  transfer map in the theory of endoscopy in two aspects:
\vskip 5pt

\begin{itemize}
\item the spaces of test functions may be larger than the space of compactly supported smooth functions;
\vskip 5pt

\item the transfer map in endoscopy is carried out via  an orbit-by-orbit comparison,
whereas the transfer map in this relative setting is more global in nature, involving an integral kernel  transformation reminiscent of the Fourier transform.
\end{itemize}
An ongoing work \cite{JK} of D. Johnstone and R. Krishna establishes the fundamental lemma for the basic functions in the space of test functions; this is necessary for the comparison of relative trace formulae.  As an example, in the special case when $n = 4$, one has:
\[  X = \SO_3 \backslash \SO_4 \cong \PGL_2 \backslash (\SL_2 \times \SL_2)/\mu_2^{\Delta}. \]
The relative trace formula for this $X$ is essentially the stable trace formula for $\SL_2$.  Thus, the expected comparison of relative trace formulae is between the stable trace formula for $\SL_2$ and the Kuznetsov trace formula for $\SL_2$.  The local transfer in this case was first investigated in the thesis work of Z. Rudnick. 
The discussion of these results was the content of Sakellaridis's lecture in the Simons Symposium.  
\vskip 5pt

On the other hand, the spectral analysis of $L^2(X)$ when $X = \SO_{n-1} \backslash \SO_n$ or the analysis of the $\SO_{n-1}$-period for representations of $\SO_n$  (both locally and globally) is familiar from   the theory of theta correspondence. The $L^2$-theory was studied in the early work of  Strichartz \cite{St} and Howe \cite{H}.   In a  paper \cite{GG} by the first author and R. Gomez, the $L^2$-theory was treated using theta correspondence for essentially general rank 1 spherical varieties from the viewpoint of the Sakellaridis-Venkatesh conjecture.  For the smooth theory, one can see the recent expository paper \cite{G}. In the case of $X = \SO_{n-1} \backslash \SO_n$, it was known that $\SO_{n-1}$-distinguished representations of $\SO_n$ are theta lifts (of $\psi$-generic representations) from $\SL_2$ or $\Mp_2$ according to whether $n$ is even or odd. Indeed, the theta lifting from 
$\SL_2$ or $\Mp_2$ to $\SO_n$  realises the functorial lifting  (at least at the level of unramified representations)  predicted by the map $\iota: X^{\vee} \times \SL_2(\C) \longrightarrow \SO_n(\C)$.
As such, it is very natural to ask if the results discussed in Sakellaridis' talk can be approached from the viewpoint of the  theta correspondence. 
\vskip 5pt

This paper is the result of this investigation. In short, its main conclusion is that the theory of transfer developed by Sakellaridis can be very efficiently developed using the  theta correspondence. More precisely,
\vskip 5pt

\begin{itemize}
\item  one can give a conceptual   definition of the transfer and the relevant spaces of test functions (Definition \ref{D:trans}), from which the fundamental lemma (for  the basic function and its translate by the spherical Hecke algebra) follows readily (see Lemmas \ref{L:fund} and \ref{L:fund2});
\vskip 5pt

\item one can establish the desired relative character identities highlighted in (a) above, without doing a geometric comparison; (see Theorem \ref{T:main})
\vskip 5pt

\item one can express this conceptually defined transfer in geometric terms, from which one sees that it agrees with Sakellaridis' formula (see Proposition \ref{P:geom});

\vskip 5pt

\item one can address the two global problems highlighted in (c) above (see Theorem \ref{T:global}). 
\end{itemize}
\vskip 5pt
We leave the precise formulation of the results to the main body of the paper. We would like to remark that, as far as we are aware, the paper
\cite{MR} of Mao$-$Rallis is the first instance where one finds a  derivation of relative character identities using the theta correspondence; this approach was followed up by the paper \cite{BLM} of Baruch$-$Lapid$-$Mao. The situation treated in this paper is in fact simpler than those in \cite{MR} and \cite{BLM}. 
 In addition, it has been known to practitioners that the theory of theta correspondence is useful for  addressing  period problems in the smooth local context, the global context, as well as in the local $L^2$-context \cite{G, GG}, with similar computations and parallel treatment in the various settings. One goal of  this paper is to demonstrate how  the treatment of the 3 different threads can be synthesised into a  rather coherent story.
 
\vskip 5pt

Here  is a short summary of the contents of this paper. In \S 2, we recall some foundational results of Bernstein \cite{B} on spectral decomposition of $L^2(X)$. These results provide the mechanism for us to navigate between the $L^2$-setting and the smooth setting. We illustrate Bernstein's general theory in the setting of the Harish-Chandra-Plancherel formula and the Whittaker-Plancherel formula in \S 3 and  further specialize to the group $\SL_2$ in \S \ref{S:SL2}, where we set up some standard conventions and establish some basic results.
 In \S 5, we recall the setup of theta correspondence, especially a recent result of Sakellaridis \cite{S3} on the spectral decomposition of the Weil representation when restricted to a dual pair. Using the theory of theta correspondence, we address in \S 6 the local problems (a) and (b), except for the part involving relative character identities.   After recalling the notion of relative characters in \S 7, we come to the heart of the paper (\S8-9),  where we develop the theory of transfer and establish some of its key properties, culminating in the relative character identity in \S 9.  In \S 10, we place ourselves in the unramified setting and explicitly determine the local L-factor $L_X(s,-)$ using theta correspondence.  We verify that our transfer map is the same as that of Sakellaridis' in \S 11, where we describe the transfer in geometric terms, as an explicit integral transform. 
 The final \S 12 discusses and resolves the global problems.
\vskip 5pt

\noindent{\bf Acknowledgments:} 
The first author thanks Sug Woo Shin, Nicholas Templier and Werner Mueller for their kind invitation to participate in the Simons Symposium and the 
Simons Foundation for providing travel support.  He also thanks Yiannis Sakellaridis for helpful conversations on the various topics discussed in this paper. We thank the referees for their careful work and helpful suggestions, especially the suggestion to be absolutely precise about normalization of measures.
The first author is partially supported  by a Singapore government MOE Tier 2 grant R146-000-233-112, whereas the second author is supported by an MOE Graduate Research Scholarship.

\vskip 15pt

\section{\bf Spectral Decomposition \`{a} la Bernstein}  \label{S:bern}

Let $F$ be a local field and $G$ a reductive group over $F$ acting transitively on a variety $X$.  
 We fix a base point $x_0 \in X(F)$, with stabilizer $H \subset G$, so that $g \mapsto g^{-1}\cdot x_0$ gives an identification $ H \backslash G \cong X$. 
 For simplicity, we  shall write $X = G(F) \cdot x_0  \cong H(F) \backslash G(F)$.   
 
 \vskip 5pt
 
 \subsection{\bf Direct integral decompositions} 
Suppose that there is a $G$-invariant measure $dx$ on $X$, in which case we may consider the unitary representation $L^2(X, dx)$ of $G$, with $G$-invariant inner product
\[  \langle \phi_1, \phi_2 \rangle_X  = \int_X \phi_1(x) \cdot \overline{\phi_2(x)} \, \, dx. \]
Such a unitary representation admits a direct integral decomposition
\begin{equation} \label{E:direct} 
\iota:  L^2(X, dx)  \cong \int_{\Omega}   \sigma(\omega) \, \,  d \mu(\omega).  
 \end{equation}
Here,
\vskip 5pt

\begin{itemize}
\item  $\Omega$ is a measurable space, equipped with a measure $d\mu(\omega)$; 
\vskip 5pt

\item $\sigma: \omega \mapsto \sigma(\omega)$  is a measurable field of irreducible unitary representations of $G$ over $\Omega$, which we may regard 
as a measurable map from $\Omega$ to the unitary dual $\widehat{G}$ of $G$ (equipped with the Fell topology and the corresponding Borel structure).
\end{itemize}
 
\vskip 5pt

In this section, we give an exposition of some results of Bernstein  \cite{B} which provide some useful ways of understanding the above direct integral decomposition.
This viewpoint of Bernstein underpins the results of this paper.
\vskip 5pt

\subsection{\bf Pointwise-defined and fine morphisms.}
Let $S \subset L^2(X)$ be a subspace which is $G$-stable. Following Bernstein \cite[\S 1.3]{B}, one says that the inclusion $S \hookrightarrow L^2(X)$ is {\em pointwise-defined} (relative to $\iota$) if  there exists a family of $G$-equivariant (continuous) morphisms $\alpha_{\sigma(\omega)}:  S \longrightarrow  \sigma(\omega)$ for $\omega\in \Omega$ such that for each $\phi \in S$,  the element $\iota(\phi)$ in the direct integral decomposition in (\ref{E:direct}) is the measurable section
\[  \iota(\phi) (\omega)  = \alpha_{\sigma(\omega)}(\phi). \]
In particular, these sections determine the measurable field structure on the right hand side of  (\ref{E:direct}).
 The family $\{ \alpha_{\sigma(\omega)}: \omega \in {\rm supp}(d\mu)\}$ is essentially unique, in the sense that any two such families differ only on  a subset of $\Omega$ with  measure zero with respect to $d\mu$.   Bernstein calls the  embedding $S \hookrightarrow L^2(X)$ {\em fine} if it is pointwise-defined relative to any such isomorphism $\iota$ to a direct integral decomposition.
\vskip 5pt

\subsection{\bf The maps $\alpha_{\sigma(\omega)}$ and $\beta_{\sigma(\omega)}$}  \label{SS:alpha}
A basic result of Bernstein \cite[Prop. 2.3]{B}, obtained as an application of the Gelfand-Kostyuchenko method \cite[Thm. 1.5]{B}, is that the natural inclusion $C^{\infty}_c(X) \hookrightarrow L^2(X)$ is fine. 
For any isomoprhism $\iota$ as in (\ref{E:direct}), we let $\{\alpha_{\sigma(\omega)}: \omega \in {\rm supp}(d\mu) \}$ be the associated family of $G$-equivariant morphisms as above. 

\vskip 5pt
The elements in $C^{\infty}_c(X)$ are $G$-smooth vectors and so the image of each $\alpha_{\sigma(\omega)}$ is contained in the space $\sigma(\omega)^{\infty}$ of smooth vectors in $\sigma(\omega)$. As the map $\alpha_{\sigma(\omega)}$ is nonzero for $d\mu$-almost all $\omega$, its image is dense in $\sigma(\omega)^{\infty}$, and is in fact equal to $\sigma(\omega)^{\infty}$  when $F$ is $p$-adic (where there is no topology considered on $\sigma(\omega)^{\infty}$). To simplify notation,  we shall sometimes write $\sigma(\omega)$ in place of $\sigma(\omega)^{\infty}$, trusting that the context will make it clear whether one is working with a unitary representation on a Hilbert space or a smooth representation. In particular, $\alpha_{\sigma(\omega)}  \in \Hom_G(C^{\infty}_c(X), \sigma(\omega))$.  
\vskip 5pt

If $\alpha_{\sigma(\omega)}$ is nonzero, then by duality, one obtains a $G$-equivariant embedding
\[  \overline{\beta}_{\sigma(\omega)}: \sigma(\omega)^{\vee} \cong \overline{\sigma(\omega)}  \longrightarrow C^{\infty}(X). \]
 Here,  the isomorphism $\sigma(\omega)^{\vee} \cong \overline{\sigma(\omega)}$ is induced by the fixed inner product $\langle -, - \rangle_{\sigma}$ and the duality between $C^{\infty}_c(X)$ and $C^{\infty}(X)$ is given by the natural pairing induced by integration with respect to the $G$-invariant measure $dx$. Taking complex conjugate on $C^{\infty}(X)$, we obtain a $G$-equivariant linear map
   \[ \beta_{\sigma(\omega)}: \sigma(\omega)^{\infty} \longrightarrow  C^{\infty}(X). \]
The maps $\alpha_{\sigma(\omega)}$ and $\beta_{\sigma(\omega)}$ are thus related by the adjunction formula:
 \begin{equation}  \label{E:alpha}
  \langle \alpha_{\sigma(\omega)}(\phi),  v \rangle_{\sigma(\omega)} =  \langle \phi,  \beta_{\sigma(\omega)}(v) \rangle_X, \quad \text{  for $\phi \in C^{\infty}_c(X)$ and $v \in \sigma(\omega)^{\infty}$.}  \end{equation}
If we compose $\beta_{\sigma(\omega)}$ with the evaluation-at-$x_0$ map $ev_{x_0}$, we obtain
\[  \ell_{\sigma(\omega)} := ev_{x_0} \circ \beta_{\sigma(\omega)}  \in  \Hom_H(\sigma(\omega)^{\infty}, \C). \] 
Thus the direct integral  decomposition gives rise to a family of canonical elements $\ell_{\sigma(\omega)} \in  \Hom_H(\sigma(\omega)^{\infty}, \C)$ for $\omega \in {\rm supp}(d\mu)$. This family  depends on the isomorphism $\iota$ in (\ref{E:direct}); changing $\iota$  will result in another family  which  differs from the original one by  a  measurable function 
$f: {\rm supp}(d\mu) \longrightarrow S^1$.  Thus, the family 
\[  \{ \alpha_{\sigma(\omega)} \otimes \overline{\alpha_{\sigma(\omega)}}: \omega \in {\rm supp}(d\mu) \} \]
 is independent of the choice of the isomorphism $\iota$ in (\ref{E:direct}). Likewise, the family 
 \[  \{ \beta_{\sigma(\omega)} \otimes \overline{\beta_{\sigma(\omega)}}: \omega \in {\rm supp}(d\mu) \} \]
  is independent of  $\iota$.
\vskip 10pt

\subsection{\bf Harish-Chandra$-$Schwartz space of $X$.}
In \cite[Pg. 689]{B}, Bernstein showed that the space $X$ has a naturally associated Harish-Chandra Schwartz space $\mathcal{C}(X)$ which is $G$-stable and which 
contains $C^{\infty}_c(X)$.   Moreover, $\mathcal{C}(X)$  has a natural (complete) topology, such that $C^{\infty}_c(X)$ is a dense subspace.  
Indeed, $\mathcal{C}(X)$ is a Frech\'et space in the archimedean case and is a strict LF space in the non-archimedean case.
More importantly, he showed in \cite[Thm. 3.2]{B}  that the inclusion $\mathcal{C}(X) \hookrightarrow L^2(X)$ is fine. Hence, the maps $\alpha_{\sigma(\omega)}: C^{\infty}_c(X) \rightarrow \sigma(\omega)$ defined above extend continuously to  the larger space $\mathcal{C}(X)$:
\[  \alpha_{\sigma(\omega)}:  \mathcal{C}(X) \longrightarrow \sigma(\omega)^{\infty}. \]
The dual map $\beta_{\sigma(\omega)}$ then takes value in the weak Harish-Chandra Schwartz space $\mathcal{C}^w(X) \subset C^{\infty}(X)$ (see \cite{B} or \cite[\S 2.4]{BP2} for the group case).
The elements $\ell_{\sigma(\omega)}  \in \Hom_H(\sigma(\omega), \C)$ are called $X$-tempered forms and the support of $d\mu$ consists  precisely of those representations with nonzero $X$-tempered forms \cite[Pg. 689]{B}. 
\vskip 5pt

\vskip 10pt

\subsection{\bf Inner Product.}
 The direct integral decomposition (\ref{E:direct}) leads to a spectral decomposition of the inner product $\langle-,-\rangle_X$ of $X$:
 \begin{equation} \label{E:inner}
  \langle \phi_1, \phi_2 \rangle _X= \int_{\Omega} J_{\sigma(\omega)}(\phi_1, \phi_2 ) \, \, d\mu(\omega), \end{equation}
 where $J_{\sigma(\omega)}$ is a $G$-invariant positive-semidefinite Hermitian form on $C^{\infty}_c(X)$ given by:
  \begin{align} \label{E:product}
    J_{\sigma(\omega)} (\phi_1, \phi_2)  &=   \langle \iota(\phi_1)(\omega), \iota(\phi_2)(\omega) \rangle_{\sigma(\omega)} \notag \\
    &=  \langle \alpha_{\sigma(\omega)}(\phi_1), \alpha_{\sigma(\omega)}(\phi_2) \rangle_{\sigma(\omega)} \notag \\
    &=   \langle \beta_{\sigma(\omega)} \alpha_{\sigma(\omega)} (\phi_1), \phi_2 \rangle_X  
    \end{align}
  for $d\mu$-almost all $\omega$. In particular, $J_{\sigma(\omega)}$ factors as:
  \[  \begin{CD}
 J_{\sigma(\omega)}:  C^{\infty}_c(X) \times \overline{C^{\infty}_c(X)} @>\alpha_{\sigma(\omega)} \otimes \overline{\alpha_{\sigma(\omega)}}>> \sigma(\omega) \otimes  \overline{\sigma(\omega)} @>\langle- , -\rangle_{\sigma(\omega)}>> \C. \end{CD} \]
 \vskip 10pt

\subsection{\bf Pointwise spectral decomposition.}
The fact that the morphism $\mathcal{C}(X) \hookrightarrow L^2(X)$ is fine leads to a pointwise spectral decomposition for elements of $\mathcal{C}(X)$. More precisely, for $\phi \in \mathcal{C}(X)$, one has
\begin{equation} \label{E:point}
 \phi(x) = \int_{\Omega} \beta_{\sigma(\omega)} \alpha_{\sigma(\omega)}(\phi) (x) \, d\mu(\omega) \end{equation}
for any $x \in X$. We give a sketch of the derivation of this when $F$ is non-archimedean.  In that case, $\phi \in \mathcal{C}(X)$ is fixed by some open compact subgroup $J \subset G$. The group $J$  also fixes $ \beta_{\sigma(\omega)} \alpha_{\sigma(\omega)}(\phi)$ for any $\sigma$, since $\alpha_{\sigma(\omega)}$ and $\beta_{\sigma(\omega)}$ are $G$-equivariant. If $1_{xJ}$ denotes the characteristic function of the open compact subset $xJ \subset X$, then it follows that 
\[ \langle \phi, 1_{xJ} \rangle_X    = {\rm Vol}(xJ;dx) \cdot \phi(x) \quad \text{and} \quad  \langle   \beta_{\sigma(\omega)} \alpha_{\sigma(\omega)}(\phi),  1_{xJ} \rangle_X = 
 {\rm Vol}(xJ;dx) \cdot   \beta_{\sigma(\omega)} \alpha_{\sigma(\omega)}(\phi) (x) \]
 where ${\rm Vol}(xJ;dx)$ is the volume of  $xJ \subset X$ with respect to the measure $dx$. 
 Now it follows  that
  \begin{align}
   \phi(x)   &= \frac{1}{{\rm Vol}(xJ;dx)} \cdot \langle \phi, 1_{xJ} \rangle_X  \notag \\
   &= \frac{1}{{\rm Vol}(xJ; dx)}   \int_{\Omega} \langle   \beta_{\sigma(\omega)} \alpha_{\sigma(\omega)}(\phi),  1_{xJ} \rangle_X \, \,    d\mu(\omega)  \notag \\
   &=  \int_{\Omega} \beta_{\sigma(\omega)} \alpha_{\sigma(\omega)}(\phi) (x) \, d\mu(\omega), \notag
   \end{align}
 where the second equality is a consequence of (\ref{E:inner}) and  (\ref{E:product}). 
 \vskip 10pt
 
The crux of Bernstein's viewpoint  in \cite{B} is  that to give the isomorphism $\iota$ in the direct integral decomposition (\ref{E:direct}) is equivalent to giving the family $\{ \alpha_{\sigma(\omega)}:  \omega \in \Omega \}$ (satisfying appropriate properties), together with  the measure $d\mu$ on $\Omega$.  In the next section, we shall illustrate this in two basic examples.
\vskip 10pt

\section{\bf Basic Plancherel Theorems}
In this section, we describe two basic Plancherel theorems as an illustration of the abstract theory of Bernstein discussed in the previous section. 
These are the Harish-Chandra-Plancherel theorem and the Whittaker-Plancherel theorem. The latter will play a crucial role in this paper.
\vskip 5pt

We shall continue to work over a local field $F$. However, we will implicitly be assuming that $F$ is non-archimedean. 
In fact, the results of this paper will hold for archimedean local fields as well, but greater care is needed in introducing the various objects (such as various spaces of functions and the topologies they carry) and  in formulating the results. Thus, there are analytic and topological considerations that need to be addressed in the archimedean case. We refer the reader to the papers \cite{BP1, BP2} and the thesis of the second author \cite{Wan}  where these issues are dealt with carefully and content ourselves with treating the nonarchimedean case in the interest of efficiency.  
\vskip 5pt

\subsection{\bf Harish-Chandra-Plancherel Theorem.}  \label{SS:HCP}
The most basic example is the regular representation $L^2(G)$ of a semisimple group $G \times G$ (acting by right and left translation):
\[  (g_1, g_2) f (g) =  f(g_2^{-1} \cdot g \cdot g_1). \] 
Here, we have fixed a Haar measure $dg$ on $G$ which defines the inner product on $L^2(G)$.   
\vskip 5pt

 Now  Harish-Chandra's Plancherel theorem  \cite{W, Wa1} asserts that there is an explicitly constructed $G \times G$-equivariant isomorphism
\begin{equation} \label{E:HCP}     L^2(G)  \cong  \int_{\widehat{G} } \overline{\sigma} \widehat{\boxtimes} \sigma  \, d\mu_G(\sigma) \end{equation}
for a specific measure  $d\mu_G$  on $\widehat{G}$ known as  the Plancherel measure of $G$ (which depends on the Haar measure $dg$). The support of this measure is precisely the subset $\widehat{G}_{temp}$  of irreducible tempered representations of $G$.   Thus, in this case, one may take the measurable space $\Omega$ to be the unitary dual $\widehat{G}$ and the map $\widehat{G} \rightarrow \widehat{G \times G}$ is given by $\sigma \mapsto \overline{\sigma} \hat{\boxtimes}  \sigma$. 
 Implicit in the theorem is the data of a measurable field of unitary representations  over $\widehat{G}$ whose fiber at $\sigma \in \widehat{G}$ is the representation $\overline{\sigma} \widehat{\boxtimes } \sigma$.

\vskip 5pt

One may describe the above direct integral decomposition (including the isomorphism) from Bernstein's viewpoint.
The Hilbert space $\overline{\sigma} \widehat{\boxtimes} \sigma$ is naturally identified with the space ${\rm End}_{HS}(\sigma)$ of Hilbert-Schmidt operators on $\sigma$, equipped with the Hilbert-Schmidt norm, and its space of $G \times G$-smooth
vectors is the space  $\overline{\sigma}^{\infty} \otimes \sigma^{\infty} = {\rm End}_{\rm fin}(\sigma^{\infty})$ of finite rank operators on $\sigma^{\infty}$. 
To describe the direct integral decomposition,  one needs to give  the family of maps:
\[  \alpha_{\overline{\sigma} \boxtimes \sigma} :  C^{\infty}_c(G)  \longrightarrow    {\rm End}_{\rm fin}(\sigma^{\infty}),    \]
 The map   $\alpha_{\overline{\sigma} \boxtimes \sigma}$  is given by
\[   \alpha_{\overline{\sigma} \boxtimes\sigma}(f)  = \sigma(f)  :=  \int_G  f(g)  \cdot \sigma(g)  \, dg \]
and the (conjugate) dual map 
\[  \beta_{\overline{\sigma} \boxtimes\sigma}: \overline{\sigma}^{\infty}  \otimes \sigma^{\infty}  \longrightarrow C^{\infty}(G)  \]
is given by the formation of matrix coefficients.  This data characterizes and explicates the measurable field of unitary representations implicit in the Plancherel theorem: the sections $\sigma \mapsto \sigma(f)$ for $f \in C^{\infty}_c(G)$  generate the family of measurable sections of the Hilbert space bundle  $\{ {\rm End}_{HS}(\sigma) =  \overline{\sigma} \widehat{\boxtimes} \sigma \}_{\sigma \in \widehat{G}}$  over $\widehat{G}$.
\vskip 5pt

The associated inner product $J_{\overline{\sigma} \boxtimes\sigma}$ is given by:
\[   J_{\overline{\sigma} \boxtimes\sigma}(f_1, f_2)    = {\rm Tr}(\sigma(f_1)  \sigma(f_2^{\vee})), \]
where 
\[  f_2^{\vee}(g)  = \overline{f_2(g^{-1})}. \]
 The $\overline{\sigma} \otimes \sigma$-component of $f \in C^{\infty}_c(G)$  in the  pointwise spectral decomposition is the function given by
 \[  (\beta_{\overline{\sigma} \boxtimes\sigma} \circ   \alpha_{\overline{\sigma} \boxtimes\sigma})(f)(g) = {\rm Tr}(\sigma(g)^{-1} \circ \sigma(f)). \] 
 In particular, 
 \[  f  \mapsto  \ell_{\overline{\sigma} \boxtimes\sigma} (\alpha_{\overline{\sigma} \boxtimes\sigma}(f)) = Tr(\sigma(f))\]
 is the Harish-Chandra character distribution of $\sigma$. For tempered $\sigma$, it extends to the (original) Harish-Chandra Schwartz space $\mathcal{C}(G)$. 
 
  \vskip 5pt

  \vskip 5pt
  It is instructive to take note of how the Plancherel measure $d\mu_G$ depends on the Haar measure $dg$. If we replace $dg$ by $\lambda \cdot dg$ for some $\lambda \in \R^{\times}_{>0}$, then we observe that
  \[  J_{\overline{\sigma} \boxtimes\sigma} \mapsto \lambda^2 \cdot J_{\overline{\sigma} \boxtimes\sigma}, \quad  \alpha_{\overline{\sigma} \boxtimes\sigma} \mapsto \lambda \cdot  \alpha_{\overline{\sigma} \boxtimes\sigma} \quad \text{and} \quad   \beta_{\overline{\sigma} \boxtimes\sigma} \mapsto  \beta_{\overline{\sigma} \boxtimes\sigma}. \]
   Hence, we have
  \[ d \mu_G \mapsto \lambda^{-1} \cdot d\mu_G. \]
\vskip 5pt

We have restricted ourselves to semisimple groups in this subsection for simplicity. When $G$ is a reductive algebraic group and $Z \subset G$ the maximal $F$-split torus in the center of $G$, then one may fix a unitary character $\chi$ of $Z$ and consider the unitary representation $L^2_{\chi}(G)$ consisting of $L^2$-functions $f$ which satisfies $f(zg) = \chi(z) \cdot f(g)$ for $z \in Z$ and $g \in G$ and  equipped with the unitary structure determined by a Haar measure $dg/dz$ of $G/Z$. Moreover, one may also consider nonlinear finite central extensions of $G(F)$ by finite cyclic groups. In all these cases, the Harish-Chandra Plancherel theorem continue to hold. 
 \vskip 10pt

\subsection{\bf Whittaker-Plancherel Theorem.}  \label{SS:WP}
Our second example is the Whittaker-Plancherel theorem (see \cite{BP2, D, SV, W}), which is a variant of the setting discussed above. Let $G$ be a quasi-split semisimple group with $N$ the unipotent radical of a Borel subgroup. Fix a nondegenerate unitary character $\psi_N$ of $N$. We consider the {\em Whittaker variety} $(N ,\psi_N) \backslash G$ and its associated unitary representation $L^2(N, \psi_N \backslash G)$ (which depends on fixed Haar measures $dg$ on $G$ and $dn$ on $N$). This extends the setting we discussed above, as one is considering $L^2$-sections of a line bundle on the spherical variety $N \backslash G$ instead of  $L^2$-functions, but it is also covered in \cite{B}.  
\vskip 5pt
It has been shown (see \cite{D, SV, Wa2}) that one has a direct integral decomposition
\begin{equation} \label{E:WPlan}
  L^2(N ,\psi_N \backslash G)  \cong \int_{\widehat{G}}   \dim \Hom_N(\sigma, \psi_N)  \cdot  \sigma \,  \, d\mu_G(\sigma), \end{equation}
where we recall that $d\mu_G$ is the Plancherel measure of $G$ (associated to the fixed Haar measure $dg$). Thus, in this case, we are taking $\Omega$ to be $\widehat{G}$ and   the map $\Omega \rightarrow \widehat{G}$ is the identity map.  The  spectral measure $d\mu_{G,\psi_N}$ is equal to $ \dim \Hom_N(\sigma, \psi_N)  \cdot d\mu_G$, whose
support is the subset $\widehat{G}_{temp,\psi_N}$ of $\psi_N$-generic irreducible tempered representations.
\vskip 5pt

Associated  to this direct integral decomposition is the family of morphisms 
\[  \alpha_{\sigma}:  C^{\infty}_c(N,\psi_N \backslash G)  \longrightarrow \sigma  \]
for all $\sigma \in \widehat{G}_{temp,\psi_N}$.  Moreover, the map $\alpha_{\sigma}$ extends to the Harish-Chandra-Schwarz space $\mathcal{C}(N, \psi_N \backslash G)$.   
 We describe instead the  (conjugate) dual map 
 \[  \beta_{\sigma} \otimes \overline{\beta_{\sigma}}:  \sigma \otimes \overline{\sigma} \longrightarrow C^{\infty}(N \times N, \psi_N  \otimes \overline{\psi_N} \backslash G  \times G) \]
 as follows. Given $v_1, v_2 \in \sigma$, one has
 \begin{equation} \label{E:Whit}
   \beta_{\sigma} \otimes \overline{\beta_{\sigma}} (v_1 \otimes v_2)(g_1, g_2) = \int^*_N  \overline{\psi_N(n)}  \cdot \langle \sigma(n\cdot g_1) (v_1),  \sigma(g_2)(v_2) \rangle_{\sigma} \, dn  \end{equation} 
 where the integral is a regularized one (see \cite[Prop. 2.3]{LM} , \cite[\S 6.3]{SV} and \cite{BP1}).  Here, note that $\beta_{\sigma} \otimes \overline{\beta_{\sigma}}$ depends on $dn$, as it should. The composite of this with the evaluation-at-$1$ map is thus the Whittaker functional
 \begin{equation} \label{E:Whit2}
  \ell_{\sigma} \otimes \overline{\ell_{\sigma}}:  v_1 \otimes v_2 \mapsto  \int^*_N  \overline{\psi_N(n)}  \cdot \langle \sigma(n) (v_1),  v_2 \rangle _{\sigma}\, dn.  \end{equation}
 The associated (positive semidefinite) inner product $J_{\sigma}$ on $C^{\infty}_c(N ,\psi_N \backslash G)$ is then given by
 \begin{align}
   &J_{\sigma} (f_1, f_2)   \notag \\
 = &\sum_{v \in {\rm ONB}(\sigma)}  \langle f_1, \beta_{\sigma}(v) \rangle_{N \backslash G} \cdot \langle \beta_{\sigma}(v),  f_2 \rangle_{N \backslash G} \notag \\
 = &\sum_{v \in {\rm ONB}(\sigma)}  \int_{N \times N \backslash G \times G} f_1(g_1) \cdot \overline{f_2(g_2)} \cdot  \ell_{\sigma}(g_2 v) \cdot \overline{\ell_{\sigma}(g_1v)} \, \, \frac{dg_1}{dn} \, \, \frac{dg_2}{dn}   \notag \\
 =  &\sum_{v \in {\rm ONB}(\sigma)}   \int_{N \times N \backslash G \times G} f_1(g_1) \cdot \overline{f_2(g_2)} \cdot  \left( \int_N^*  \overline{\psi_N(n)}  \cdot \langle ng_2 v , g_1 v \rangle \, dn  \right)\, \, \frac{dg_1}{dn} \, \, \frac{dg_2}{dn}.  \notag
 \end{align}
\vskip 5pt

 The above maps specify on the right hand side of (\ref{E:WPlan})  the structure of a measurable field of unitary representations on $\widehat{G}_{temp,\psi_N}$ whose fiber at $\sigma \in \widehat{G}_{temp,\psi_N}$ is the representation $\sigma$.  We can think of this measurable field of unitary representations as a ``tautological" or ``universal bundle of unitary representations" over the ``moduli space" $\widehat{G}_{temp, \psi_N}$ of irreducible $\psi_N$-generic tempered representations.

 \vskip 5pt
It is again useful to take note of how the various quantities change when one replaces the Haar measure $dg$ of $G$ by $\lambda \cdot dg$ for some 
$\lambda \in \R^{\times}_{>0}$. In the Whittaker-Plancherel case, one sees from the above formula that $\beta_{\sigma}$ and $\ell_{\sigma}$ are unchanged whereas
\[  J_{\sigma} \mapsto \lambda^2 \cdot J_{\sigma} \quad \text{and} \quad \alpha_{\sigma} \mapsto \lambda \cdot \alpha_{\sigma}, \] 
keeping in mind that the Plancherel measure $d\mu_G$ gets replaced by $\lambda^{-1} \cdot d\mu_G$. 
 \vskip 5pt
 
 As in the case of the Harish-Chandra Plancherel theorem, we could have worked with a reductive algebraic group (in which case we fix a central character $\chi$ as before and consider $L^2_{\chi}(N, \psi_N \backslash G)$) or a nonlinear finite cover thereof. One has the Whittaker-Plancherel theorem in these settings as well, though we take note that uniqueness of Whittaker models fails for nonlinear covering groups in general.

\vskip 10pt
\subsection{\bf Continuity properties}
We now consider the issue of continuity  (in $\sigma$) for some of the quantities discussed above. We first need to say a few words about the Fell topology on $\widehat{G}_{temp}$.
\vskip 5pt

The unitary dual $\widehat{G}$ is typically non-Hausdorff even though it is still a T1 space for the groups considered here.  The tempered dual $\widehat{G}_{temp}$ is still not necessarily Hausdorff, but can often be replaced by a substitute which is Hausdorff. Namely, one can work with the space of equivalence classes of induced representations  $\tau = {\rm Ind}_P^G \pi$ where $P$ is a parabolic subgroup of $G$ and $\pi$ a discrete series representation of its Levi factor $M$.  This space was variously denoted by $\mathcal{T}$ in \cite{S3}, ${\rm Temp}_{\rm Ind}(G)$ in \cite{BP2} and $\mathcal{X}_{temp}(G)$ in \cite{BP1, X}, so we are spoilt for choices! To add to this galore, we shall denote this space by $\widehat{G}_{temp}^{ind}$.
Then $\widehat{G}_{temp}^{ind}$ has the structure of an orbifold (given by  twisting $\pi$ by unramified unitary characters of $M$).  There is a natural continuous finite-to-one surjective map
\[  \widehat{G}_{temp} \longrightarrow   \widehat{G}_{temp}^{ind}  \]
sending a tempered irreducible representation $\sigma$ to the unique induced representation ${\rm Ind}_P^G \pi$ containing $\sigma$. 
This map is injective outside a subset of $\widehat{G}_{temp}$ which has measure zero with respect to the Plancherel measure $d\mu_G$. 
\vskip 5pt

In the setting of the Harish-Chandra-Plancherel theorem of \S\ref{SS:HCP}, one could safely replace the integral over $\widehat{G}_{temp}$ in (\ref{E:HCP}) by an integral over $\widehat{G}_{temp}^{ind}$. Moreover, we have the Hermitian form $J_{\overline{\sigma} \boxtimes \sigma}(\phi_1, \phi_2)$ for $\phi_i \in C^{\infty}_c(G)$  (or more generally $\mathcal{C}(G)$) and $\sigma \in \widehat{G}_{temp}$. We can similarly define $J_{\overline{\tau} \boxtimes \tau}(\phi_1, \phi_2)$ for $\tau \in \widehat{G}_{temp}^{ind}$. Then we have \cite[\S2.13]{BP2}:
\vskip 5pt

\begin{lem}   \label{L:cont1}
For fixed $\phi_1$ and $\phi_2$ in $\mathcal{C}(G)$, the maps
\[  \tau  \mapsto J_{\overline{\tau} \boxtimes \tau}(\phi_1, \phi_2)   \]
  is continuous as a $\C$-valued function on $\widehat{G}_{temp}^{ind}$.
In particular, the map $\sigma \mapsto  J_{\overline{\sigma} \boxtimes \sigma}(\phi_1, \phi_2)$ is continuous on the subset of $\widehat{G}_{temp}$ which maps injectively into $\widehat{G}_{temp}^{ind}$. 
\end{lem}
\vskip 5pt

In the context of the Whittaker-Plancherel theorem, we are working with the subset $\widehat{G}_{temp, \psi_N}$.  When uniqueness of Whitaker models holds (such as for reductive algebraic groups or the metaplectic groups which are two fold covers of symplectic groups), each $\tau = {\rm Ind}_P^G \pi$ in $\widehat{G}_{temp}^{ind}$ can have at most one irreducible constituent which is $\psi_N$-generic. Hence, we see that the composite map
\[  \widehat{G}_{temp,\psi_N} \longrightarrow \widehat{G}_{temp} \longrightarrow \widehat{G}_{temp}^{ind} \]
is injective (and continuous). As a consequence of \cite[\S2.14]{BP2}, we have:

\begin{lem} \label{L:cont2}
In the context of the Whittaker-Plancherel theorem,  for fixed $f_1$ and $f_2$ in $\mathcal{C}(N, \psi_N \backslash G)$, the map
 \[  \sigma \mapsto J_{\sigma}(f_1, f_2) \]
 is a continuous $\C$-valued function on $\widehat{G}_{temp, \psi_N}$.  
  Likewise, for fixed $f \in \mathcal{C}(N, \psi_N \backslash G)$,  the map
  \[ \sigma \mapsto \beta_{\sigma} \alpha_{\sigma} (f)(1) \]
  is continuous. 
 \end{lem}
\vskip 10pt

\section{\bf $\GL_2$ and $\SL_2$}   \label{S:SL2}
In this section, we specialize the discussion of the previous section to the case of $\GL_2$ and $\SL_2$. Since the group $\SL_2$ will feature prominently in the rest of the paper, 
we also take the opportunity to set up some precise conventions which will be used for the rest of the paper.  
\vskip 5pt

\subsection{\bf Measures on $F$ and $F^{\times}$} \label{SS:measures}
Let us first fix a nontrivial additive character 
\[ \psi : F \longrightarrow S^1. \]
Then $\psi$ determines an additive  Haar measure $d_{\psi}x$ on $F$, characterized by the requirement that
$d_{\psi}x$ is self-dual with respect to the Fourier transform relative to the pairing $(x,y) \mapsto \psi(xy)$ on $F$. If $F$ is nonarchimedean with ring of integers $\mathcal{O}_F$ and $\psi$ has conductor $\mathcal{O}_F$, then the Haar measure $d_{\psi} x$ gives $\mathcal{O}_F$ volume $1$. 
 One also obtains a multiplicative Haar measure on $F^{\times}$ given by 
 \[  d^{\times}_{\psi}x = d_{\psi}x / |x|. \] 
 \vskip 5pt

 More generally, for any algebraic group $G$ over $F$, a nonzero element $\omega$ of $\wedge^{\rm top} {\rm Lie}(G)^*$ and the additive Haar measure $d_{\psi}x$ together give rise  to a right-invariant (or left-invariant) Haar measure $|\omega|_{\psi}$. Since $\psi$ is fixed throughout, we will often suppress it from the notation and simply write $|\omega|$.  
\vskip 5pt

\subsection{\bf The group $\GL_2$}   \label{SS:GL2}
We now consider the group $\GL_2$ over $F$. Let $\tilde{B}$ be the upper triangular Borel subgroup  with unipotent radical $N $ (the group of upper triangular unipotent matrices) and consider the diagonal maximal torus. We can write the diagonal maximal torus  as  $Z \cdot S$ where $Z$ is the center of $\GL_2$ (the scalar matrices) and 
\[ S = \{ s(y) = \left( \begin{array}{cc}
y & 0 \\
0 & 1 \end{array} \right) : y \in F^{\times} \}. \]
By \S \ref{SS:measures}, we have Haar measures on $N(F) = F$, $Z(F) = F^{\times}$  and $S(F) = F^{\times}$ and hence a right-invariant measure on $\tilde{B}$.
\vskip 5pt

\vskip 5pt
Regard the fixed additive character $\psi$ of $F$ as a character of $N(F)$. 
For a fixed a unitary character $\chi$ of $Z$, the Whittaker-Plancherel theorem for $L^2_{\chi}( N, \psi \backslash \GL_2)$ gives a family of $\GL_2$-equivariant embeddings
\[  \tilde{\beta}_{\tilde{\sigma}, \psi} : \tilde{\sigma} \longrightarrow C^{\infty}_{\chi}( N, \psi \backslash \GL_2) = {\rm Ind}_{Z \cdot N}^{\GL_2} \chi \otimes \psi \]
for irreducible tempered representations $\tilde{\sigma}$ of $\GL_2$ with central character $\chi$. As we have noted,  $\tilde{\beta}_{\tilde{\sigma},\psi}$ only depends on the Haar measure $d_{\psi}n$  on $N$ (which we have fixed), and does not depend on the choice of the Haar measure on $\GL_2$ which enters into the formulation of the Whittaker-Plancherel theorem.

\vskip 5pt

We may consider the $\tilde{B}$-equivariant map 
\[  {\rm rest}:  C^{\infty}_{\chi}( N,  \psi \backslash \GL_2) \longrightarrow  C^{\infty}_{\chi} (N, \psi \backslash \tilde{B} ) = {\rm Ind}_{Z \cdot N}^{\tilde{B}} \chi \otimes \psi \]
given by the restriction of functions. The Haar measures we have fixed endow the latter space with a unitary structure, whose inner product is given by
\[  \langle f_1, f_2 \rangle_{Z \cdot N \backslash \tilde{B}} = \int_{F^{\times}} f_1  \left( \begin{array}{cc}
y & 0 \\
0 & 1 \end{array} \right) \cdot \overline{ f_2 \left( \begin{array}{cc}
y & 0 \\
0 & 1 \end{array} \right)} \, \, d^{\times}_{\psi} y. \]
\vskip 5pt

We now note the following basic result, which is a reformulation of   \cite[Lemma 4.4]{LM} and \cite[Prop. 2.14.3]{BP2}. In these references, this result was shown in the setting of  
$\GL_n$ (with appropriate formulation). In any case, this result is the reason why we consider the case of $\GL_2$ in this section.  
\vskip 5pt

\begin{prop} \label{P:GL2}
The composite map ${\rm rest} \circ \tilde{\beta}_{\tilde{\sigma},\psi}$ gives an isometric embedding
\[ {\rm rest} \circ  \tilde{\beta}_{\tilde{\sigma},\psi} :  \tilde{\sigma}  \longrightarrow L^2_{\chi}(N, \psi\backslash \tilde{B} ). \]
 \end{prop}
\vskip 5pt

 \subsection{\bf The group $\SL_2$} \label{SS:SL2}
Now we turn to the group $\SL_2$.  The goal is to deduce from Proposition \ref{P:GL2} its analog in the setting of $\SL_2$. 
 We first take the opportunity to introduce some conventions for $\SL_2$ which will be enforced throughout the paper. 
\vskip 5pt

We first fix the upper triangular Borel subgroup $B = T \cdot N$, where $T$ is the diagonal torus and $N$ the unipotent radical of $B$, and a maximal compact subgroup $K$ in good relative position with respect to $B$. 
For example, when $F$ is non-archimedean with ring of integers $\mathcal{O}_F$, we can simply take $K = \SL_2(\mathcal{O}_F)$. 
We have natural identifications  $N(F)  \cong F$ and $T(F) \cong F^{\times}$ such that the modulus character of $B$ is given by $\delta_B(t)  = |t|^2$. 
For $a \in F^{\times}$ and $b \in N$, we write 
\[ n(b)  = \left( \begin{array}{cc}
1 & b \\
0 & 1 \end{array} \right) \in N(F) \quad \text{and}  \quad t(a)  = \left( \begin{array}{cc}
a & 0 \\
0 & a^{-1} \end{array} \right) \in T(F).\]  
Further, the groups $N(F) = F^{\times}$ and $T(F) = F^{\times}$ carry the fixed Haar measure $dn = d_{\psi}x$ and $dt = d_{\psi}x/ |x|$. 
 We also fix a Haar measure $dg$ on $\SL_2(F)$, which in turn determines a Plancherel measure $d\mu_{\SL_2}$ on the unitary dual $\widehat{\SL}_2$.

 \vskip 10pt

Now we come to the analog of Proposition \ref{P:GL2} in the $\SL_2$-setting. The main difference between $\GL_2$ and $\SL_2$ is that, whereas  there is a unique equivalence class of Whittaker datum in the case of $\GL_2$, there are $F^{\times}/ F^{\times 2}$-worth of them in the case of $SL_2$.   For any $a \in F^{\times}$, set $\psi_a(x) = \psi(ax)$ so that $\psi_a$ is a nontrivial additive character of $F$. Then the two Whittaker data $(N, \psi_a)$ and $(N, \psi_b)$ of $\SL_2$ are equivalent if and only if $a /b \in F^{\times 2}$.  In formulating an analog of Proposition \ref{P:GL2} in the $\SL_2$-setting, it will be necessary to take all the various  inequivalent Whittaker data into account. 
\vskip 5pt

Henceforth, let us fix a set of representatives $[a]$ for $F^{\times}/F^{\times 2}$, so that
\begin{equation} \label{E:rep}
 F^{\times} = \bigsqcup_{[a] \in F^{\times}/F^{\times 2}} a F^{\times 2}. \end{equation}
We shall assume $a = 1$ is one of the representatives.
\vskip 5pt

For each $a \in F^{\times}$, the Whittaker-Plancherel theorem for $L^2(N,\psi_a \backslash \SL_2)$ (relative to the fixed Haar measures $dg$ on $\SL_2$ and $d_{\psi}x$ on $N$)  then furnishes the maps $\alpha_{\sigma, \psi_a}$, $\beta_{\sigma, \psi_a}$ and $\ell_{\sigma, \psi_a}$ for any irreducible $\psi_a$-generic tempered representation $\sigma$ of $\SL_2$.  In particular, if  $\sigma$ is $\psi_a$-generic, then
\[ \beta_{\sigma, \psi_a}: \sigma  \longrightarrow  C^{\infty}(N, \psi_a \backslash \SL_2) \]
is an $\SL_2$-equivariant embedding.
As in the $\GL_2$ case, we may consider the restriction of functions 
\[  C^{\infty}(N, \psi_a \backslash \SL_2) \longrightarrow  C^{\infty}(N, \psi_a \backslash B). \]
Let us scale this restriction map a little, by setting
\[ {\rm rest}_a: C^{\infty}(N, \psi_a \backslash \SL_2) \longrightarrow  C^{\infty}(N, \psi_a \backslash B) \]
to be
\[ {\rm rest}_a(f) = |a|^{1/2} \cdot  f|_B.  \]
Combining all these maps together gives us an $B$-equivariant map
\begin{equation} \label{E:rest}
j_{\sigma} := \bigoplus_{[a] \in F^{\times}/ F^{\times 2}} {\rm rest}_a \circ \beta_{\sigma, \psi_a}  : \sigma \longrightarrow  \bigoplus_{[a] \in F^{\times}/ F^{\times 2}}  C^{\infty}(N, \psi_a \backslash B) \end{equation}
where here, the map $\beta_{\sigma, \psi_a}$ is interpreted to be $0$ if $\sigma$ is not $\psi_a$-generic. 
Now we have the following analog of Proposition \ref{P:GL2}, which will play a crucial role later on (in the proof of Proposition \ref{P:L2p})..
\vskip 5pt

\begin{prop} \label{P:SL2}
Equip $C^{\infty}(N, \psi_a \backslash B) \cong C^{\infty}(T)$ with the unitary structure
\[  \langle f_1, f_2 \rangle_{N \backslash B} =\frac{1}{2} \cdot |2|_F \cdot  \int_{F^{\times}} f_1(t(b)) \cdot \overline{f_2(t(b))} \, d^{\times}_{\psi}b, \]
which is the natural inner product associated to the Haar measures fixed on $T$ scaled by the factor $|2|_F /2$. Then, for any irreducible tempered  representation $\sigma$ of $\SL_2$, the $B$-equivariant map defined by (\ref{E:rest}) is an isometry 
\[ j_{\sigma}: \sigma \longrightarrow  \bigoplus_{[a] \in F^{\times}/ F^{\times 2}}  L^2(N, \psi_a \backslash B). \]
\end{prop}

 \vskip 5pt
\begin{proof}
To deduce this proposition from  Proposition \ref{P:GL2},  we naturally regard $\SL_2$ as a subgroup of $\GL_2$.
Given an irreducible tempered representation $\sigma$ of $\SL_2$, we pick a unitary representation $\tilde{\sigma}$ of $\GL_2$ with unitary central character $\chi$ such that $\sigma \subset \tilde{\sigma}$. We may assume that the inner product $\langle -, - \rangle_{\tilde{\sigma}}$  restricts to the inner product $\langle-, -\rangle_{\sigma}$ on $\sigma$.
\vskip 5pt

Consider the $\psi$-Whittaker functional $\tilde{\ell}_{\tilde{\sigma}, \psi}$ and the associated map $\tilde{\beta}_{\tilde{\sigma},\psi}$ for $\tilde{\sigma}$.
 Observe that for any $a \in F^{\times}$, $\tilde{\ell}_{\tilde{\sigma},\psi} \circ s(a) \in \Hom_N(\tilde{\sigma}, \psi_a)$ and the following statements are equivalent:
 \vskip 5pt
 \begin{itemize}
 \item[-]  the restriction of $\tilde{\ell}_{\tilde{\sigma},\psi} \circ s(a) $ to $\sigma$ is nonzero;
 \item[-]  $\sigma$ is $\psi_a$-generic;
 \item[-]  $s(a)(\sigma) = \sigma$ as a subspace of $\tilde{\sigma}$.
 \end{itemize}
Further, if $\sigma$ is $\psi$-generic, then from the formula (\ref{E:Whit2}), one sees that  the restriction of  $\tilde{\ell}_{\tilde{\sigma},\psi} $ to $\sigma$ is equal to $\ell_{\sigma, \psi}$.
\vskip 5pt

Now take $w_1, w_2 \in \sigma \subset \tilde{\sigma}$. By Proposition \ref{P:GL2}, one has
\[
  \langle w_1, w_2 \rangle_{\sigma} = \langle w_1, w_2 \rangle_{\tilde{\sigma}} = \int_{F^{\times}} \tilde{\beta}_{\tilde{\sigma}} (w_1)(s(x)) \cdot \overline{\tilde{\beta}_{\tilde{\sigma}} (w_2)(s(x))} \, d^{\times}_{\psi}(x), \]
where we have written $\tilde{\beta}_{\tilde{\sigma}}$ in place of $\tilde{\beta}_{\tilde{\sigma},\psi}$ to simplify notation.
To evaluate the latter integral, we decompose the domain of integration into square classes as in (\ref{E:rep}) and uniformize each square class $aF^{\times 2}$ by $F^{\times}$, using the map 
$b \mapsto a b^2$ (which is a $2$-to-$1$ map). Performing the corresponding change of variables in  the integral (i.e. replacing  $x$ by $ab^2$ on $aF^{\times 2}$), we obtain:
\begin{align}
  \langle w_1, w_2 \rangle_{\sigma}  &= \sum_{[a] \in F^{\times}/F^{\times 2}} 
\frac{1}{2} \cdot |2|_F \cdot  \int_{F^{\times}} \tilde{\beta}_{\tilde{\sigma}}(w_1) (s(ab^2)) \cdot \overline{\tilde{\beta}_{\tilde{\sigma}}( w_2) (s(ab^2))}
 \, d^{\times}_{\psi}b \notag \\
 &= \sum_{[a] \in F^{\times}/F^{\times 2}} 
\frac{1}{2} \cdot |2|_F \cdot  \int_{F^{\times}} \tilde{\beta}_{\tilde{\sigma}}(w_1) (s(a) \cdot t(b)) \cdot \overline{\tilde{\beta}_{\tilde{\sigma}}( w_2) (s(a) \cdot t(b))}
 \, d^{\times}_{\psi}b \notag 
 \end{align}
 This computation is the source of the factor $|2|_F/2$ appearing in the proposition. 
Let us define a map
\[  j_{\sigma, a} : \sigma \mapsto   C^{\infty}(N, \psi_a \backslash B) \]
by
\[    j_{\sigma, a} (w) (t(b)) =   \tilde{\beta}_{\tilde{\sigma}, \psi} (w)(s(a) \cdot t(b)), \]
noting that $j_{\sigma, a}$ is nonzero if and only if $\sigma$ is $\psi_a$-generic. 
Then  we have shown that for $w_1, w_2 \in \sigma$,  
\[  \langle w_1, w_2 \rangle_{\sigma}  = \sum_{[a] \in F^{\times}/F^{\times 2}}  \langle  j_{\sigma, a}(w_1), j_{\sigma, a}(w_2) \rangle_{N \backslash B}. \]
In other words, one has an isometry 
\[  \bigoplus_{[a] \in F^{\times}/F^{\times 2}} j_{\sigma, a} :  \sigma \longrightarrow \bigoplus_{[a] \in F^{\times}/F^{\times 2}}   L^2(N,\psi_a \backslash B) \]
where the unitary structure of the latter spaces are as given in the proposition. 
\vskip 5pt

It remains then to show that 
\[ j_{\sigma, a}  ={\rm rest}_a \circ  \beta_{\sigma, \psi_a}  \quad \text{for each $a \in F^{\times}$,}  \]
or equivalently
\[ \tilde{\ell}_{\tilde{\sigma}}  \circ s(a)  = |a|^{1/2} \cdot \ell_{\sigma, \psi_a} \quad \text{ on $\sigma \subset \tilde{\sigma}$.} \]
To see this, for $w_1, w_2 \in \sigma \subset \tilde{\sigma}$, we apply (\ref{E:Whit2}) to get
\begin{align}
  &\tilde{\ell}_{\tilde{\sigma}} (s(a) w_1) \cdot  \overline{ \tilde{\ell}_{\tilde{\sigma}} (s(a) w_2) } \notag \\
  = &\int^*_F   \langle  n(x) \cdot s(a) w_1,  s(a) w_2 \rangle_{\tilde{\sigma}} \cdot \overline{\psi(x)} \, d_{\psi}x \notag \\
  =&\int^*_N \langle s(a) \cdot n(a^{-1}x) w_1, s(a) w_2 \rangle{\tilde{\sigma}} \cdot \overline{\psi(x)} \, d_{\psi}x \notag \\
  =&\int^*_N \langle  n(x) w_1, w_2 \rangle_{\tilde{\sigma}} \cdot \overline{\psi(ax)} \cdot |a|  \, d_{\psi}x  \notag\\
 = &|a| \cdot \ell_{\sigma, \psi_a}(w_1) \cdot \overline{\ell_{\sigma, \psi_a}}(w_2). \notag
  \end{align}
  Here in the penultimate equality, we have applied a change of variables, replacing $x$ by $ax$ and used the unitarity of $\tilde{\sigma}$.
  We have thus shown that
  \[  \tilde{\ell}_{\tilde{\sigma}}  \circ s(a) =   |a|^{1/2} \cdot \ell_{\sigma, \psi_a} \]
  at least up to a root of unity (which we may ignore, by absorbing it into $\tilde{\ell}_{\tilde{\sigma}}$).  This completes the proof of Proposition \ref{P:SL2}.
 \end{proof}
 \vskip 5pt
 
 We can describe the image of the isometry $j_{\sigma}$ precisely. Observe that the center $Z_{\SL_2} = \mu_2$ of $\SL_2$ induces a decomposition
 \[  L^2(N, \psi_a, \backslash B) =  L^2(N, \psi_a, \backslash B)^+  \oplus   L^2(N, \psi_a, \backslash B)^- \]
 into two irreducible  $B$-subrepresentations which are the $\pm$-eigenspaces of $\mu_2$. Then one has:
 \vskip 5pt
 
 \begin{cor} \label{C:SL2}
 Let $z_{\sigma} = \pm$ denote the central character of the irreducible tempered representation $\sigma$ of $\SL_2$. 
 The map $j_{\sigma}$ in (\ref{E:rest}) defines a $B$-equivariant  isometric isomorphism
 \[  j_{\sigma} : \sigma = \bigoplus_{[a] \in F^{\times}/ F^{\times 2}}  \dim \Hom_N(\sigma, \psi_a) \cdot L^2(N, \psi_a \backslash B)^{z_{\sigma}} \]
 where the unitary structure of the right hand side is as defined in Proposition \ref{P:SL2}.
 \end{cor}
 
 \vskip 10pt
 
 \subsection{\bf Harish-Chandra-Schwartz space}
Finally, we explicate when a function $f \in C^{\infty}(N,\psi \backslash \SL_2)$ lies in the Harish-Chandra Schwartz space.
 The measures $dg$ and $dn$  determine an $\SL_2(F)$-invariant measure
on  $N(F) \backslash \SL_2(F)$, which can be described as follows. An element $f \in C^{\infty}(N,\psi \backslash \SL_2)$ is determined by its restriction to $TK$, by the Iwasawa decomposition. Then the integral of $f$ with respect to $dg/dn$ is  given by 
\[  f \mapsto       \int_T \int_K  f(tk)  \cdot \delta_B(t)^{-1}  \, dt \, dk \] 
for some Haar measure $dk$ of $K$. 
\vskip 5pt

Given a function $f \in C^{\infty}(N,\psi \backslash \SL_2)$,  
 the smoothness of $f$ implies that the function $t \to f(tk)$ on $T \cong F^{\times}$ is necessarily rapidly decreasing at $|t| \to \infty$ (indeed, it vanishes on some domain $|t| > C$ in the p-adic case).  
Thus the analytic properties of $f$ depend on its asymptotics as $|t| \to 0$. 
We have the following lemma:
\vskip 5pt

\begin{lem}  \label{L:SL2}
Let $f \in C^{\infty}(N,\psi \backslash \SL_2)$ and suppose that there exists $C>0$ and $d > 0$ such that
\[  \sup_{k \in K} |f(tk)|  \leq  C  \cdot |t|^d      \quad \text{as $|t| \to 0$}.  \]

\vskip 5pt

\begin{itemize}
\item[(i)] If $d > 1$, then $f \in \mathcal{C}(N, \psi\backslash \SL_2)$.
\vskip 5pt
\item[(ii)] If $d > 2$, then $f  \in L^1(N ,\psi \backslash \SL_2)$.
\end{itemize}
\end{lem}
\vskip 5pt
 \vskip 15pt

\section{\bf Theta Correspondence}
In this section, we recall the setup of the theta correspondence and recall some results of Sakellaridis \cite{S3} on the spectral decomposition of the Weil representation for a dual pair. 

\subsection{\bf Weil representation.}
If $W$ is a symplectic vector space and $(V,q)$ a quadratic space over a local field $F$, then one has a dual reductive pair 
\[   \Sp(W) \times \OO(V) \longrightarrow  \Sp(V \otimes W). \]
In this paper, we shall only consider the case where $W = F \cdot e  \oplus F \cdot f$ is 2-dimensional with $\langle e,f \rangle_W =1$.  With the Witt basis $\{e, f\}$, we may identify $\Sp(W)$ with  $\SL_2(F)$, and we let $B =T \cdot N$ be the Borel subgroup which stabilises the line $F \cdot e$ (so that $B$ is upper triangular). In particular, the conventions we have set up in \S \ref{SS:SL2} for $\SL_2$ apply to $\Sp(W)$.
\vskip 5pt

Attached to  a fixed nontrivial additive character $\psi$ of $F$ and other auxiliary data, this dual pair has a distinguished representation $\Omega_{\psi}$ known as the Weil representation. To be precise, if $\dim V$ is odd, we need to work with the metaplectic double cover $\Mp_2(F)$ of $\SL_2(F)$. To simplify notation, we shall ignore this issue; the reader may assume $\dim V$ is even. We refer the reader to \cite{GQT, GS} for the metaplectic cases. 
\vskip 5pt

To describe the Weil representation $\Omega_{\psi}$, we first need to endow the vector space $V$ with a Haar measure. Let $\langle-, -\rangle$  be the symmetric bilinear form associated to the quadratic form $q$ on $V$, so that $\langle v_1, v_2 \rangle = q(v_1+ v_2) -q(v_1)- q(v_2)$. Then one has an $S^1$-valued nondegenerate pairing $\psi( \langle- , - \rangle)$ on $V$. We then equip $V$ with the Haar measure $d_{\psi}v$ which is self-dual with respect to the Fourier transform defined by this pairing and observe that $d_{\psi}v$ is
  $\OO(V)$-invariant.  The unitary representation $\Omega_{\psi}$ can be realised on $L^2(f \otimes V) = L^2(V)$, where the inner product is defined using the Haar measure $d_{\psi}v$.  The action of various elements of $\SL_2(F) \times \OO(V)$ via $\Omega_{\psi}$ is given as follows:

\[  \begin{cases}
h \cdot \Phi (v)  = \Phi (h^{-1} \cdot v), \text{  for $h \in \OO(V)$;} \\
n(b) \cdot \Phi (v)  = \psi(b\cdot  q(v)) \cdot \Phi(v), \text{  for $n(b) = \left( \begin{array}{cc}
1 & b \\
0 & 1 \end{array}  \right) \in N$;} \\
t(a) \cdot \Phi(v)  = |a|^{\frac{1}{2}  \dim V} \chi_{\disc(V)}(a) \cdot \Phi( a^2 v) , \text{  for $t(a) =  \left( \begin{array}{cc}
a & 0 \\
0 & a^{-1} \end{array}  \right) \in T$.}  
\end{cases} \]
Here $\disc(V) \in F^{\times}/ F^{\times 2}$ is the discriminant of $(V,q)$ and $\chi_{\disc(V)}$ is the associated quadratic character of $F^{\times}$.
This describes $\Omega_{\psi}$ as a representation of $B \times \OO(V)$. To describe the full action of $\SL_2(F)$, one needs to give the action of a nontrivial Weyl group element 
\[ w = \left( \begin{array}{cc}
0 & 1 \\
 -1 & 0 \end{array}  \right). \]
Its action is given  by a normalized Fourier transform $\mathcal{F}$:
\begin{equation} \label{E:weyl}
 w \cdot \Phi (v) = \mathcal{F} (\Phi)(v) :=  \gamma_{\psi, q} \cdot \int  \Phi(v') \cdot \psi( \langle v, v' \rangle) \, d_{\psi}v' \end{equation}
where $\gamma_{\psi, q}$ is a root of unity (a Weil index) whose precise value need not concern us here. 
\vskip 5pt

One may consider the underlying smooth representation $\Omega_{\psi}^{\infty}$ which is realized on the subspace $S(V)$ of Schwartz-Bruhat functions on $V$. 
Following our convention, we shall use $\Omega_{\psi}$ to denote the Weil representation in both the smooth and $L^2$-setting when there is no cause for confusion. 

\vskip 10pt

\subsection{\bf Smooth Theta correspondence.}. \label{SS:smooth}
The theory of theta correspondence concerns the understanding of the representation $\Omega_{\psi}$ of $\SL_2(F) \times \OO(V)$. One can consider this question on the level of smooth representation theory or $L^2$-representation theory. In this subsection, we recall the setup of the smooth theory. Henceforth, we shall assume that $\dim V \geq 3$ (and sometimes $\dim V \geq 4$). 
\vskip 5pt

For $\sigma \in {\rm Irr}(\SL_2)$, the (smooth) big theta lift of $\sigma$ to $\OO(V)$ is:
\[  \Theta_{\psi}(\sigma) := (\Omega_{\psi}^{\infty} \otimes \sigma^{\vee})_{\SL_2} \]
where we are considering the space of $\SL_2$-coinvariants.
With this definition, we have the natural  $\SL_2$-invariant and $\OO(V)$-equivariant projection map
\[
 A_{\sigma} : \Omega_{\psi}^{\infty}\otimes \sigma^{\vee} \longrightarrow \Theta_{\psi}(\sigma) \]
which gives by duality a canonical $\SL_2 \times \OO(V)$-equivariant map
\[
  \theta_{\sigma} : \Omega_{\psi}^{\infty} \longrightarrow \sigma \boxtimes \Theta_{\psi}(\sigma). \]
Likewise, for $\pi \in {\rm Irr}(\OO(V))$, the (smooth) big theta lift of $\pi$ to $\SL_2$ is:
\[  \Theta_{\psi}(\pi) :=  (\Omega_{\psi}^{\infty} \otimes \pi^{\vee})_{\OO(V)} \]
where we are considering the space of $\OO(V)$-coinvariants.   

\vskip 5pt

By the Howe duality principle \cite{GT}, one knows that:
\vskip 5pt

\begin{itemize}
\item[-] the representations $\Theta_{\psi}(\sigma)$ and $\Theta_{\psi}(\pi)$  are finite length representations which (if nonzero) have unique irreducible quotients $\theta_{\psi}(\sigma)$ and $\theta_{\psi}(\pi)$ respectively (known as the small theta lifts);
\item[-] for any $\sigma_!, \sigma_2 \in {\rm Irr}(\SL_2)$, 
\[  \theta_{\psi}(\sigma_1) \cong \theta_{\psi}(\sigma_2) \ne 0 \Longrightarrow  \sigma_1 \cong \sigma_2. \]
\end{itemize}
As a consequence, we see that if $\pi := \theta_{\psi}(\sigma)$, then $\sigma \cong \theta_{\psi}(\pi)$.
\vskip 5pt

Composing  $A_{\sigma}$ and $\theta_{\sigma}$ with the natural projection $\Theta_{\psi}(\sigma) \twoheadrightarrow \theta_{\psi}(\sigma)$, we have canonical equivariant maps (still denoted by the same symbols)
\begin{equation} \label{E:A-can}
 A_{\sigma}: \Omega \otimes \sigma^{\vee} \longrightarrow \theta_{\psi}(\sigma) \end{equation}
and
\begin{equation} \label{E:theta}
  \theta_{\sigma} : \Omega_{\psi}^{\infty} \longrightarrow \sigma \boxtimes \theta_{\psi}(\sigma). \end{equation}
\vskip 5pt

The theta correspondence  for $\SL_2 \times \OO(V)$ (when $\dim V$ is even) and $\Mp_2 \times \OO(V)$ (when $\dim V$ is odd) was studied in great detail by Rallis \cite{R3}. 
His results were supplemented by later results of J.S. Li \cite{Li}. We may summarize their results by:
\vskip 5pt

\begin{prop}\label{P:the}
(i) Assume that $\dim V \geq 4$ is even and the Witt index ${\rm Witt}(V)$ of $V$ is $\geq 2$ (so that one is in the stable range). If $\sigma \in {\rm Irr}(\SL_2)$ is unitary, then $\theta(\sigma)$ is nonzero unitary, so that one has an injective map
\[  \theta_{\psi}: \widehat{\SL_2} \longrightarrow \widehat{\OO(V)}. \]
In general, the theta correspondence gives a map
\[  \theta_{\psi}: \widehat{\SL_2}_{temp} \longrightarrow \widehat{\OO(V)}  \cup \{0\}  \]
which is injective on that part of the domain outside the preimage of $0$. 

\vskip 5pt

(ii) Assume that $\dim V \geq  3$ is odd and ${\rm Witt}(V)$ is $\geq 2$. If $\sigma \in {\rm Irr}({\rm Mp}_2)$ is a unitary genuine representation, then
$\theta_{\psi}(\sigma)$ is nonzero unitary, so that one has an injective map
\[ \theta_{\psi}:  \widehat{{\rm Mp}_2}\longrightarrow \widehat{\OO(V)} \]
where $ \widehat{{\rm Mp}_2}$ denotes the $\psi$-generic genuine unitary dual of $\Mp_2(F)$.  In general, the theta correspondence gives a map
\[  \theta_{\psi} : \widehat{\Mp_2}_{temp} \longrightarrow \widehat{\OO(V)} \cup \{0\}  \]
which is  injective on that part of the domain outside the preimage of $0$. 
 \end{prop}
\vskip 5pt

One can in fact describe the map  $\theta_{\psi}$ very explicitly but we will not need this description here. 

\vskip 5pt
\subsection{\bf Doubling zeta integral}  \label{SS:doubling}
Using the unitary structures of $\Omega_{\psi}$ and $\sigma$ (and the Haar measure $dg$ on $\SL_2$), the representation $\theta_{\psi}(\sigma)$ of $\OO(V)$ can be given a unitary structure by the local doubling zeta integral. More precisely, for $\Phi_1, \Phi_2 \in S(V)$ and $v_1, v_2 \in \sigma$, the local doubling zeta integral is
given by:
\begin{equation} \label{E:double}
   Z_{\sigma}( \Phi_1, \Phi_2, v_1,v_2) 
=  \int_{\SL_2}  \langle g \cdot  \Phi_1, \Phi_2 \rangle_{\Omega}  \cdot\overline{ \langle \sigma(g) \cdot v_1, v_2 \rangle_{\sigma} }\, \, dg, \end{equation}
which converges for tempered $\sigma$ when $\dim V \geq 3$. It defines a $(\SL_2 \times \SL_2)$-invariant and $\OO(V)^{\Delta}$-invariant map
\[ Z_{\sigma}: \Omega_{\psi} \otimes \overline{\Omega_{\psi}} \otimes \overline{\sigma} \otimes \sigma  \longrightarrow  \C. \]
The inner product $\langle-, -\rangle_{\sigma}$ on $\sigma$ gives an isomorphism $\overline{\sigma} \cong \sigma^{\vee}$. 
Hence, $Z_{\sigma}$ factors through the canonical projection map 
\[ A_{\sigma} \otimes \overline{A_{\sigma}} : \Omega_{\psi} \otimes \overline{\Omega_{\psi}} \otimes \sigma^{\vee} \otimes \overline{\sigma^{\vee}}  \longrightarrow  \Theta_{\psi}(\sigma) \otimes \overline{\Theta_{\psi}(\sigma)} \]
so that 
\begin{equation}  \label{E:zeta}
  Z_{\sigma}(\Phi_1, \Phi_2, v_1, v_2) =  \langle A_{\sigma}(\Phi_1,v_1) , A_{\sigma}(\Phi_2, v_2) \rangle_{\theta(\sigma)}. \end{equation}
for some  Hermitian form $\langle-, - \rangle_{\theta(\sigma)}$ on $\Theta_{\psi}(\sigma)$. We have:
\vskip 5pt

\begin{prop} \label{P:zeta}
Suppose that $\sigma$ is an irreducible tempered representation of $\SL_2$ (or ${\rm Mp}_2$) such that $\theta_{\psi}(\sigma) \ne 0$.  Then the Hermitian form $\langle-, - \rangle_{\theta(\sigma)}$ on $\Theta_{\psi}(\sigma)$
descends to a nonzero inner product on $\theta_{\psi}(\sigma)$. 
\end{prop}
\begin{proof}
We note:
\vskip 5pt
\begin{itemize}
\item[-] If the Witt index of $V$ is $\geq 2$ (so that one is in the stable range), this is due to \cite{Li}.
\item[-] In general, it was shown by Rallis \cite[Prop. 6.1]{R3} that 
\[  \theta_{\psi}(\sigma) \ne 0 \Longleftrightarrow \langle-, - \rangle_{\theta(\sigma)} \ne 0.\]
\item[-]   In the archimedean case, it was shown in  \cite{He} that $\langle-, - \rangle_{\theta(\sigma)}$ descends to $\theta_{\psi}(\sigma)$. 
\item[-]  Consider the nonarchimedean  case with ${\rm Witt}(V) \leq 1$.    There are only a few cases of such, all in low rank.  In these small number of low rank cases,  one can verify that $\Theta_{\psi}(\sigma)$ is irreducible  when $\sigma$ is tempered. 
\end{itemize}
Taken together, the proposition is proved.
\end{proof}

Henceforth, we shall equip $\theta_{\psi}(\sigma)$ with this unitary structure; it depends on $dg$, $\langle-, - \rangle_{\sigma}$ and $d_{\psi}v$. 
By completion, we may regard $\theta_{\psi}(\sigma)$ as an irreducible  unitary representation of $\OO(V)$. Observe that 
the  identity (\ref{E:zeta}) may be considered as the local analog of the Rallis inner product formula \cite{GQT}. A reformulation, using the map $\theta_{\sigma}$ instead of $A_{\sigma}$ is:
\begin{equation} \label{E:yia}
 \langle \theta_{\sigma}(\Phi_1) , \theta_{\sigma}(\Phi_2) \rangle_{\sigma \boxtimes \theta(\sigma)} 
 = \sum_{v \in {\rm ONB}(\sigma)}  Z_{\sigma}(  \Phi_1, \Phi_2, v,v)  \end{equation} 
 where ${\rm ONB}(\sigma)$ denotes an orthonormal basis of $\sigma$.
 \vskip 10pt

\subsection{\bf $L^2$-theta correspondence.}

Now we consider the theta correspondence in the $L^2$-setting. 
Though we are not exactly in the setting discussed in \S \ref{S:bern}, Bernstein's theory continues to apply here (see \cite{S3}). 
When $\dim V \geq 3$, it was shown in \cite{GG, S3} that one has  a direct integral decomposition of $\SL_2(F) \times \OO(V)$-representations:
\begin{equation} \label{E:ome}
  \Omega_{\psi} = L^2(V)  \cong \int_{\widehat{\SL_2}}   \sigma \boxtimes \theta_{\psi}(\sigma)  \, d\mu_{\SL_2}(\sigma),  \end{equation}
where $d\mu_{\SL_2}$ is the Plancherel measure of $\SL_2(F)$ (associated to the fixed Haar measure of $\SL_2$). Hence the spectral measure of $\Omega_{\psi}$ as an $\SL_2$-module is absolutely continuous with respect to the Plancherel measure. Indeed, by Propositions \ref{P:the} and \ref{P:zeta},  when ${\rm Witt}(V) \geq 2$,  the support of $\Omega_{\psi}$ as an $\SL_2$-module is precisely $\widehat{\SL_2}_{temp}$. 
\vskip 5pt

We need to explicate the spectral decomposition (\ref{E:ome}) here. 
By the theory of spectral decomposition \`{a} la Bernstein, one way to do this is to give a spectral decomposition of the inner product $\langle-, - \rangle_{\Omega}$. In \cite{S3}, Sakellaridis showed that for $\Phi_1, \Phi_2\ \in S(V)$, 
\[ \langle \Phi_1, \Phi_2 \rangle_{\Omega} = \int_{\widehat{\SL_2}}  J^{\theta}_{\sigma} (\Phi_1, \Phi_2)  \, d\mu_{\SL_2}(\sigma)  \]
where
\begin{equation} \label{E:yia2}
 J^{\theta}_{\sigma} (\Phi_1, \Phi_2) = \langle \theta_{\sigma}(\Phi_1), \theta_{\sigma}(\Phi_2) \rangle_{\sigma \boxtimes \theta(\sigma)} 
 =   \sum_{v \in {\rm ONB}(\sigma)}  Z_{\sigma}(  \Phi_1, \Phi_2, v,v)  \end{equation}
where the unitary structure on $\Theta(\sigma)$ is that defined in the previous subsection.
 What this says is that  the family of canonical maps $\theta_{\sigma}$ defined in (\ref{E:theta}) is precisely the family of maps associated to the direct integral decomposition (\ref{E:ome})  for $\sigma \in \widehat{\SL_2}_{temp}$.

 \vskip 5pt

Now for $\tau \in \widehat{\SL_2}_{temp}^{ind}$, one can define $Z_{\tau}(\Phi_1, \Phi_2, v_1, v_2)$ by the same formula as in (\ref{E:double}) and then define $J^{\theta}_{\tau}$ by the formula (\ref{E:yia2}). Then it is useful to note \cite[Lemma 3.3]{X}: 
\vskip 5pt

\begin{lem}  \label{L:cont3}
For fixed $\Phi_i$   the $\C$-valued function  
 \[  \tau \mapsto  J^{\theta}_{\tau} (\Phi_1, \Phi_2) \]
  is continuous in $\tau \in  \widehat{\SL_2}_{temp}^{ind}$.
\end{lem}
\vskip 5pt

\subsection{\bf  The maps $A_{\sigma}$ and $B_{\theta(\sigma)}$}  \label{SS:AB}
We have seen the canonical maps $A_{\sigma}$ and $\theta_{\sigma}$ in (\ref{E:A-can}) and (\ref{E:theta}) which intervene in the spectral decomposition (\ref{E:ome}).
Identifying $\sigma^{\vee}$ with $\overline{\sigma}$ using $\langle-,-\rangle_{\sigma}$, we may regard $A_{\sigma}$ as a map $\Omega_{\psi} \otimes \overline{\sigma} \longrightarrow \theta_{\psi}(\sigma)$. Then $A_{\sigma}$ and $\theta_{\sigma}$ are related by:
 \[    A_{\sigma}(\Phi, v)  = \langle \theta_{\sigma}(\Phi), v \rangle_{\sigma}. \]
  Likewise, we have a $\OO(V)$-invariant and $\SL_2$-equivariant map
\[  B_{\theta(\sigma)} : \Omega_{\psi} \otimes \overline{\theta_{\psi}(\sigma)} \longrightarrow \sigma \]
characterized by 
\[  B_{\theta(\sigma)} (\Phi, w)  = \langle \theta_{\sigma}(\Phi),  w \rangle_{\theta(\sigma)}. \]
The two maps are related by:
\begin{equation} \label{E:AB}
 \langle A_{\sigma}(\Phi, v) , w \rangle_{\theta(\sigma)} = \langle \theta_{\sigma}(\Phi), v \otimes w \rangle_{\sigma \otimes \theta(\sigma)} = \langle B_{\theta(\sigma)}(\phi, w), v \rangle_{\sigma}\end{equation}
for $\Phi \in \Omega_{\psi}$, $v \in \sigma$ and $w \in \theta(\sigma)$. Moreover, the inner product $J^{\theta}_{\sigma}$ can be expressed in terms of $A_{\sigma}$ and 
$B_{\theta(\sigma)}$ as follows:
\[  J^{\theta}_{\sigma}(\Phi_1, \Phi_2) = \langle \theta_{\sigma}(\Phi_1), \theta_{\sigma}(\Phi_2) \rangle_{\sigma \otimes \theta(\sigma)}
 = \sum_{v \in {\rm ONB}(\sigma)}  \langle A_{\sigma}(\Phi_1, v), A_{\sigma}(\Phi_2, v) \rangle_{\theta(\sigma)}.   \]
and
\[   J^{\theta}_{\sigma}(\Phi_1, \Phi_2) =  \sum_{w \in {\rm ONB}(\theta(\sigma))}  \langle B_{\theta(\sigma)}(\Phi_1, w) ,  B_{\theta(\sigma)}(\Phi_2, w) \rangle_{\sigma}. \] 
 
The maps $A_{\sigma}$ and $B_{\theta(\sigma)}$ are local versions of global theta lifting  considered in \S \ref{S:global}.
\vskip 10pt

To summarize, this section discusses the smooth theta correspondence and the $L^2$-theta correspondence and the relation between them. In particular,  through the theory of the doubling zeta integral, we equip $\theta_{\psi}(\sigma)$ with a unitary structure so that the there is a strong synergy between the smooth theory and the $L^2$-theory.   
 \vskip 15pt

\section{\bf Periods}
It is a basic principle that theta correspondence frequently allows one to transfer periods on one member of a dual pair to the other member.  For an exposition of this in the setting of smooth theta correspondence, the reader can consult \cite{G}.   On the other hand, in the setting of $L^2$-theta correspondence, this principle has been exploited in \cite{GG} to establish low rank cases of the local conjecture of Sakellaridis-Venkatesh on the unitary spectrum of spherical varieties. 
\vskip 5pt

In this section, we shall consider the dual pair $\SL_2 \times \OO(V)$ and show how the spectral decomposition \`{a} la Bernstein allows one to
refine the results of \cite{G} and \cite{GG}.

\vskip 10pt

\subsection{\bf Transfer of periods.}
We first consider periods in smooth representation theory.  For $a \in F^{\times}$,  fix a vector $v_a \in V$ with $q(v_a)  =a$ (if it exists), so that $V = F \cdot v_a \oplus v_a^{\perp}$. Set 
\[  X_a =  \{  v \in V:  q(v)  =a \} \subset V, \] 
which is a Zariski closed subset of $V$. By Witt's theorem, $\OO(V)$ acts transitively on $X_a$  and the stabilizer of $v_a$ in $\OO(V)$ is $\OO(v_a^{\perp})$. Hence
\[  X_a  \cong  \OO(v_a^{\perp}) \backslash \OO(V) \]
via $h \mapsto h^{-1} \cdot v_a$.  If $v_a$ does not exist, we understand $X_a$ to be empty (i.e. the algebraic variety has no $F$-points). 
To fix ideas, we shall assume that $v_1$ exists; this is not a serious hypothesis.  We also set $\psi_a(x)  = \psi(ax)$, so that $\psi_a$ is a nontrivial additive character of $F$, and write $S(X_a)$ for the space of Schwartz-Bruhat functions on $X_a$, so that $S(X_a) = C^{\infty}_c(X_a)$ if $F$ is nonarchimedean.
\vskip 5pt

The following proposition essentially resolves the local problem (a) in the smooth setting  for the Sakellaridis-Venkatesh conjecture highlighted in the introduction, except for the part about relative character identities.
It is essentially a folklore result and a proof has been written down in \cite{G} in a more general setting. We recount the proof here to explicate the isomorphism $f_a$ in the  proposition. 
 \vskip 5pt

\begin{prop}  \label{P:smoothp}
Let $\pi$ be an irreducible smooth representation of $\OO(V,q)$ and let $\Theta_{\psi}(\pi) =(\Omega_{\psi} \otimes \pi^{\vee})_{\SL_2}$ be its big theta lift to $\SL_2$ (or $\Mp_2$ if $\dim V$ is odd).  
For $a \in F^{\times}$,  there is an explicit  isomorphism (to be described in the proof)
\[  f_a:  \Hom(\Theta_{\psi}(\pi)_{N, \psi_a},\C)   \cong \Hom_{\OO(V)}(S(X_a), \pi) \cong   \Hom_{\OO(v_a^{\perp})}(\pi^{\vee}, \C), \]
where the second isomorphism is by Frobenius reciprocity. Here, the right hand side is understood to  be $0$ if $X_a$ is empty.
  In particular,   when $\pi$ is such that  $\sigma:= \Theta_{\psi}(\pi)$ is irreducible, we see that $\sigma$ is $\psi_a$-generic if and only if $\pi$ is $\OO(v_a^{\perp})$-distinguished, in which case
  \[ \dim  \Hom_{\OO(v_a^{\perp})}(\pi^{\vee}, \C)  = 1. \]
  
\end{prop}
\vskip 5pt

\begin{proof}
We describe the proof when $F$ is nonarchimedean. The archimedean case is based on the same ideas, and the reader can consult \cite{GZ, Z} for a careful treatment.
\vskip 5pt

We prove the proposition by computing the space
\[  \Hom_{N \times \OO(V)} ( \Omega_{\psi}, \psi_a \boxtimes \pi) \]
in two different ways.  
\vskip 5pt

On one hand,  let us fix any equivariant surjective map 
\[  \theta:  \Omega_{\psi} \longrightarrow   \Theta_{\psi}(\pi) \boxtimes  \pi. \]
Then $\theta$ induces an isomorphism
\[ 
\theta^*:     \Hom_N(\Theta_{\psi}(\pi), \psi_a)  \cong  \Hom_{N \times \OO(V)} ( \Omega_{\psi} , \psi_a \otimes \pi).  \] 
On the other hand,  for $a \in F^{\times}$,  consider the surjective restriction map 
\[  {\rm rest}:  \Omega_{\psi}= S(V)  \longrightarrow   S(X_a).  \]
This map induces an equivariant isomorphism
\[  {\rm rest}: \Omega_{N, \psi_a}  \cong  S(X_a). \]
Hence, we have an induced isomorphism
 \[  {\rm rest}^*:   \Hom_{\OO(V)}(C^{\infty}_c(X_a), \pi)  \cong  \Hom_{\OO(V)}( \Omega_{N, \psi_a}, \pi)  \cong  \Hom_{N \times \OO(V)} ( \Omega_{\psi}, \psi_a \otimes \pi).     \]
 Since
 \[  S(X_a) \cong {\rm ind}_{\OO(v_a^{\perp})}^{\OO(V)} \C, \]
it follows by Frobenius reciprocity that one has the desired isomorphism:
\[ \begin{CD}
f_a^{-1} :   \Hom_{\OO(v_a^{\perp})}(\pi^{\vee}, \C) @>{\rm Frob}>>  \Hom_{\OO(V)}(S(X_a), \pi) @> (\theta^*)^{-1} \circ {\rm rest}^*>>  \Hom_N(\Theta_{\psi}(\pi), \psi_a). \end{CD} \]
This proves the proposition.
\end{proof}

The purpose of recounting the proof of the proposition is to bring forth the point that the isomorphism
\[  f_a: \Hom(\Theta_{\psi}(\pi)_{N, \psi_a},\C)   \cong   \Hom_{\OO(v_a^{\perp})}(\pi^{\vee}, \C) \]
essentially depends only on the choice of the equivariant projection map
\[  \theta:  \Omega_{\psi} \longrightarrow   \Theta_{\psi}(\pi) \boxtimes  \pi. \]
On the other hand, when $\sigma$ is an irreducible tempered representation of $\SL_2$ with $\theta_{\psi}(\sigma) \ne 0$, we have seen in (\ref{E:ome})  that there is a canonical map
\[  \theta_{\sigma}:  \Omega_{\psi}  \longrightarrow   \sigma \boxtimes  \theta_{\psi}(\sigma).  \]
 Repeating the proof of the proposition using this map $\theta_{\sigma}$, we obtain an injective map
 \begin{equation} \label{E:fa}
  f_a: Hom_N(\sigma, \psi_a) \longrightarrow \Hom_{\OO(V)}(S(X_a), \theta_{\psi}(\sigma)). \end{equation} 
  It is an isomorphism if $\Theta_{\psi}(\theta_{\psi}(\sigma)) \cong \sigma$ (by Proposition \ref{P:smoothp}) or if  $\sigma$ is $\psi_a$-generic (since the target space has dimension at most $1$ by \cite{AGRS})
\vskip 10pt

\vskip 10pt

\subsection{\bf Unitary structure on $L^2(X_a)$.}
We may begin our investigation of the local problem (b) in the Sakellaridis-Venkatesh conjecture, i.e. in the $L^2$-setting.  
With $a \in F^{\times}$,  we have seen that
\[  X_a = \{ v \in V: q(v)  =a\}  \cong \OO(v_a^{\perp}) \backslash \OO(V), \]
under $h^{-1} v_a \longleftrightarrow h$ (assuming $X_a(F)$ has a point $v_a$). 
We may equip $X_a$ with an $\OO(V)$-invariant measure and consider the space $L^2(X_a)$; of course the space $L^2(X_a)$ does not depend on the choice of the $\OO(V)$-invariant measure but the unitary structure does.
In \cite{GG}, using the spectral decomposition (\ref{E:ome}) of $\Omega_{\psi}$, one obtains a spectral decomposition of $L^2(X_a)$. However, we wish to refine the results of \cite{GG} by being more precise about  the unitary structures and invariant measures used here.
\vskip 5pt

The hyperboloids $X_a$ are precisely the fibers of the $\OO(V)$-invariant map given by the quadratic form $q: V \longrightarrow F$. This map is submersive at all points of $V$ outside the zero vector.  In particular, if we ignore the null cone $X_0$ and consider the map $q$ over the Zariski open subset $F^{\times} \subset F$, the Haar measures $d_{\psi}v$ and $d_{\psi}x$ we have already fixed for $V$ and $F$ induces an $\OO(V)$-invariant measure $|\omega_a|$ for each fiber $X_a$ (with $a \in F^{\times}$), characterized by:
 for any compactly-supported smooth functions $f$ on $V \setminus X_0$, 
\[ \int_V f(v) \, d_{\psi}(v)  = \int_{F^{\times}} \left(   \int_{X_a}  f \cdot |\omega_a| \right) \, d_{\psi}a,\]
where the function of $a \in F^{\times}$ defined by the inner integral on the right-hand-side is smooth and compactly supported.
It is these measures $|\omega_a|$ that we shall use on $X_a$. Hence, we shall be considering $L^2(X_a, |\omega_a|)$.
\vskip 5pt

The map $q: V \longrightarrow F$ is $F^{\times}$-equivariant where $t \in F^{\times}$ acts by scaling on $V$ and via $x  \mapsto t^2x$ on $F$.  
The measure $d_{\psi}v$ on $V$ is $\OO(V)$-invariant and satisfies: for $b \in F^{\times}$, 
\[ d_{\psi}(bv) = |b|^{dim V} \cdot d_{\psi}v. \]
 This homogeneity property implies the following property of the family of measures $|\omega_a|$.  For $b \in F^{\times}$, scalar multiplication-by-$b$ gives an isomorphism of varieties  $\lambda_b: X_a 
 \longrightarrow X_{ab^2}$ and one may consider the pushforward measure $ (\lambda_b)_*(|\omega_a|) $ on $X_{ab^2}$. 
 \vskip 5pt
 
 \begin{lem} \label{L:homo}
 In the above context, one has:
 \[   (\lambda_b)_*(|\omega_a|)  = |b|^{2-\dim V} \cdot |\omega_{ab^2}| \]
 for any $a,b \in F^{\times}$.
 \end{lem}
 \vskip 5pt
 
 \begin{proof}
For $f \in C^{\infty}_c(V \setminus X_0)$ and fixed $b \in F^{\times}$, we have
\begin{equation} \label{E:homo}
 \int_V f(bv) \, dv = |b|^{-\dim V} \cdot \int_V f(v) \, dv. \end{equation}
The left-hand-side of (\ref{E:homo}) is given by
\begin{align}
 \int_V f(bv) \, dv &= \int_{F^{\times}}  \int_{X_a} f(bx)  \cdot |\omega_a(x)|  \,  d_{\psi}a  \notag \\
 &= \int_{a \in F^{\times}} \int_{X_a} \lambda_b^*(f)(x) \cdot  \cdot |\omega_a(x)|  \,  d_{\psi}a  \notag \\
 &= \int_{a \in F^{\times}}  \int_{X_{ab^2}} f(x) \cdot |(\lambda_b)_* (\omega_a|)(x)|  \, d_{\psi}a \notag \\
 &= \int_{c \in F^{\times} }\int_{X_c}   f(x) \cdot |(\lambda_b)_*(\omega_{cb^{-2}})|  \cdot |b|^{-2} \, d_{\psi}c \notag 
 \end{align}
 where in the last equality, we have made a change of variables by setting $c = ab^2$, so that $d_{\psi}c = |b|^2 \cdot d_{\psi}a$. 
 \vskip 5pt
 
 On the other hand, the right hand side of (\ref{E:homo}) is given by
 \[
  |b|^{-\dim V} \cdot \int_V f(v) \, dv = |b|^{-\dim V} \cdot \int_{c \in F^{\times}} \int_{X_c}  f(x)  \cdot |\omega_c(x)| \, d_{\psi}c. \]
  Comparing the two sides, one obtains:
  \[   |(\lambda_b)_*(\omega_{cb^{-2}})|  = |b|^{2- \dim V} \cdot |\omega_c|  \]
  which is the desired assertion.
\end{proof} 
 \vskip 5pt
 Now  observe that (on the level of $F$-valued points),
\[  V \setminus X_0 = \bigcup_{[a] \in F^{\times 2} \backslash F^{\times}} F^{\times} \cdot X_a  \subset V  \]
is open dense with complement of measure $0$. Hence the measure $d_{\psi}v$ induces $\OO(V)$-invariant measures on each of the open sets $Y_a:= F^{\times} \cdot X_a$ and we have
 \[  L^2(V)  =  \bigoplus_{[a] \in F^{\times 2} \backslash F^{\times}} L^2(Y_a)  \]
We would like a more direct description of the unitary structures on the Hilbert spaces $L^2(Y_a)$ in terms of    appropriate invariant measures on $F^{\times} \times X_a$.
\vskip 5pt

Consider the natural surjective map
\[  m :  F^{\times} \times X_a \longrightarrow    Y_a= F^{\times} \cdot X_a, \]
defined by $m(t, x). \mapsto t \cdot x$. This map  $m$ induces an isomorphism $\mu_2 \backslash (F^{\times} \times X_a) \cong Y_a$ where $\mu_2 = \{ \pm 1 \}$ acts diagonally on $F^{\times} \times X_a$ by scaling on each factor.  In terms of the identification $X_A = \OO(v_a^{\perp}) \backslash \OO(V)$, this action of $\mu_2$ on $X_a$ is 
 the left-translation action of $\OO(v_a) = \mu_2$ (where $\OO(v_a)$  is the orthogonal group of the 1-dimensional quadratic space $F \cdot v_a$) which commutes with the right-translation action of  $\OO(V)$. In any case, 
 via $f \mapsto m^*(f)$,  we may identify functions on $Y_a$ with functions on $F^{\times} \times X_a$ invariant under the action of $\mu_2$. 
Now we note:
\vskip 5pt

\begin{lem} \label{L:measure}
For a smooth compactly supported function $f$ on $Y_a$, one has: 
\[  \int_{Y_a} f(v) \, d_{\psi}v  = \frac{1}{2} \cdot |2a|_F \cdot  \int_{b \in F^{\times}}   \left( \int_{x \in X_a}  f(b \cdot x) \cdot |\omega_a(x)|  \right)   \cdot |b|^{\dim V}  \, \frac{d_{\psi}b}{|b|}.\]
Hence, if we define 
\[  \phi_f : F^{\times} \times X_a \longrightarrow \C \]
by 
\[  \phi_f(b,x)=  \sqrt{|a|} \cdot |b|^{\frac{\dim V}{2}}  \cdot f(b \cdot x), \]
then one has
\[   \langle  f , f \rangle_{\Omega}   = \frac{1}{2} \cdot |2|_F\cdot  \int_{F^{\times}} \int_{X_a}   |\phi_f (b,x)|^2 \, |\omega_a(x)|  \, d_{\psi}^{\times} b. \]
In particular, the map $f \mapsto \phi_f$ defines an isometric isomorphism
\[   L^2(Y_a, d_{\psi}v) \longrightarrow L^2( F^{\times} \times X_a)^{\Delta \mu_2} \]
where the unitary structure on the left is defined by $d_{\psi}v$ and that on the right is defined by the $\OO(V)$-invariant measure $|\omega_a|$ on $X_a$ and the measure
$ |2|_F /2 \cdot d^{\times}_{\psi} t$ of $F^{\times}$ (defined in  \S \ref{SS:SL2}). Further, for $\epsilon = \pm$, if we let $L^2(Y_a)^{\epsilon}$ and $L^2(F^{\times})^{\epsilon}$ denote the $\epsilon$-eigenspace of the $\mu_2$-action, then 
\[ L^2(Y_a) \cong  \bigoplus_{\epsilon = \pm} L^2(F^{\times})^{\epsilon} \widehat{\otimes}  L^2(X_a)^{\epsilon}. \]   
\end{lem}
\vskip 5pt

\begin{proof}
We consider $f \in C^{\infty}_c(Y_a)$. Then
\begin{align}
  \int_{Y_a} f(y) \, dy &= \int_{t \in aF^{\times 2}} \int_{X_t} f(x) \cdot |\omega_t(x)| \, d_{\psi}t  \notag \\
&= \frac{1}{2} \int_{b \in F^{\times}}   \int_{X_{ab^2}}  f(x) \cdot |\omega_{ab^2}(x)| \cdot |2ab|_F \,  d_{\psi}b \notag \\
&= \frac{1}{2} \cdot |2a|_F \cdot  \int_{b \in F^{\times}} \int_{X_{ab^2}} f(x) \cdot | (\lambda_b)_*(|\omega_a|)  \cdot  |b|^{\dim V} \,  \frac{d_{\psi}b}{|b|} \notag \\
&= \frac{1}{2} \cdot |2a|_F \cdot  \int_{b \in F^{\times}} \int_{X_a}    \lambda_b^*(f)(x)  \cdot |\omega_a(x)|  \cdot |b|^{\dim V} \, d^{\times}_{\psi}b
 \notag \\
 &= \frac{1}{2} \cdot |2a|_F \cdot  \int_{b \in F^{\times}} \int_{X_a}   f(b \cdot x)  \cdot |\omega_a(x)|  \cdot |b|^{\dim V} \, d^{\times}_{\psi}b. 
 \end{align}
 Here, in the second equality, we have made a change of variables, replacing $t$ by $ab^2$, so that $d_{\psi}t = |2ab|_F \cdot d_{\psi}b$, whereas in the third equality, we have applied Lemma \ref{L:homo}.  This establishes the lemma.
\end{proof}

\subsection{\bf Spectral decomposition of $L^2(X_a)$.}
 We are now ready to show the direct integral decomposition of the unitary representation $L^2(X_a, |\omega_a|)$ of $\OO(V)$, where the unitary structure is determined by the $\OO(V)$-invariant measures $|\omega_a|$. Observe that 
  $L^2(X_a, |\omega_a|)$ is in fact  a representation of $\OO(v_a) \times \OO(V)$, where $\OO(v_a) \cong \mu_2$ acts by left translation on $\OO(v_a^{\perp}) \backslash \OO(V)$ (this is the action of scaling by $-1$ on $X_a$). This gives a decomposition
\[  L^2(X_a)  = L^2(X_a)^+  \oplus L^2(X_a)^-  \]
into $\OO(V)$-submodules which are  the $\pm$-eigenspaces of the $\OO(v_a)$-action.   
\vskip 5pt

As mentioned before, the spectral decomposition of $L^2(X_a)$ as an $\OO(V)$-module has been  
 obtained in \cite{GG}. The following proposition is a special case of the results in \cite{GG}; we recount the proof here to explicate certain isomorphisms used in the course of the proof.
\vskip 5pt

\begin{prop}  \label{P:L2p}
 We have an explicit isomorphism (to be described in the proof):
\[  L^2(X_a , |\omega_a|)  \cong \int_{\widehat{\SL}_2}  \dim \Hom_N(\sigma, \psi_a) \cdot \theta_{\psi}(\sigma) \, d\mu_{\SL_2}(\sigma). \]

\end{prop}

\begin{proof}  
We shall exploit the spectral decomposition of the unitary Weil representation $\Omega_{\psi}$ of $\SL_2(F) \times \OO(V)$  on $L^2(V)$ given in (\ref{E:ome}). More precisely, we shall consider its restriction to $B \times \OO(V)$.   We have seen  that 
\[   \bigcup_{[a] \in F^{\times 2} \backslash F^{\times}} Y_a = \bigcup_{[a] \in F^{\times 2} \backslash F^{\times}} F^{\times} \cdot X_a \subset V \]
is open dense (with complement of measure $0$), so that
\[   L^2(V)  =  \bigoplus_{a \in F^{\times 2} \backslash F^{\times}} L^2(Y_a)  \]
From the formulae for the action of $B \times \OO(V)$ on $\Omega_{\psi}$, one sees that the subspace $L^2(Y_a)$ is stable under the action of $B \times \OO(V)$ and thus  this is a decomposition of  $B \times \OO(V)$-modules.
Moreover, with $T \cong F^{\times}$ acting on $V$ by scaling, we see that
$T \times \OO(V)$ acts transitively on $Y_a$ and the stabilizer of $v_a \in X_a$  is the subgroup
\[   \mu_2^{\Delta} \times \OO(v_a^{\perp}) \subset Z \times \OO(v_a)  \times \OO(v_a^{\perp}), \]
where $Z = \mu_2$ is the center of $\SL_2(F)$ and $\OO(v_a) = \mu_2$. 
Thus, the isometry $f \mapsto \phi_f$ described in Lemma \ref{L:measure} gives an $B \times \OO(V)$-equivariant isometric map
\[  L^2(Y_a, d_{\psi}v)  \cong \bigoplus_{\epsilon = \pm}  \left( {\rm ind}_{ZN}^B  \epsilon \otimes \psi_a\right)  \widehat{\otimes} L^2(X_a)^{\epsilon} =  \bigoplus_{\epsilon = \pm} L^2(N, \psi_a \backslash B)^{\epsilon} \widehat{\otimes} L^2(X_a)^{\epsilon} = L^2(F^{\times} \times X_a)^{\mu_2} ,\] 
where the unitary structure on  $L^2(N,\psi_a \backslash B)^{\epsilon}$ is given by that defined in Lemma \ref{L:measure} or equivalently in Proposition \ref{P:SL2}.
 Hence, we conclude that  
\begin{equation} \label{E:star}
\Omega_{\psi} \cong  \bigoplus_{a \in F^{\times 2} \backslash F^{\times}}  L^2(F^{\times} \times X_a)^{\Delta \mu_2} \cong \bigoplus_{a \in F^{\times 2} \backslash F^{\times}} \bigoplus_{\epsilon = \pm}  L^2(N, \psi_a \backslash B)^{\epsilon} \widehat{\otimes}  L^2(X_a)^{\epsilon}
\end{equation}
via
\[    f \mapsto  \left(  \phi_{f|_{Y_a}} \right)_{a \in F^{\times 2} \backslash F^{\times}} . \]
 
\vskip 5pt

On the other hand, by (\ref{E:ome}), one has:
\[  \iota:  L^2(V)  \cong \int_{\widehat{\SL}_2}  \sigma \otimes \theta_{\psi}(\sigma) \, d\mu_{\SL_2}(\sigma)  \]
as $\SL_2 \times \OO(V)$-modules. Restricting from $\SL_2$ to $B$, Corollary \ref{C:SL2} gives
an isometric $B$-equivariant isomorphism
\[  
j_{\sigma} = \bigoplus_{a \in F^{\times} \backslash F^{\times 2} } j_{\sigma,a} :  \sigma|_B  \cong \bigoplus_{[a] \in F^{\times 2} \backslash F^{\times}}   \dim \sigma_{N, \psi_a} 
 \cdot L^2(N, \psi_a \backslash B)^{z_{\sigma}} \]
where $z_{\sigma} = \pm$ denotes the central character of $\sigma$ and the unitary structure on the right-hand-side is  as in Lemma \ref{L:measure}. Hence,  via  $j_{\sigma} \circ \iota$ for each $\sigma$, one has a unitary isomorphism:
\begin{equation} \label{E:star2}
  \Omega_{\psi}  \cong  \bigoplus_{a \in F^{\times 2} \backslash F^{\times}} \bigoplus_{\epsilon= \pm}  L^2(N,\psi_a \backslash B)^{\epsilon}  \widehat{\otimes} \int_{\widehat{\SL}_2}    \dim \sigma_{N, \psi_a} \cdot \theta_{\psi}(\sigma) \cdot 1(z_{\sigma} = \epsilon)  \,  d\mu_{\SL_2}(\sigma). \end{equation}

Comparing the two descriptions of $\Omega = L^2(V)$ as a $B \times \OO(V)$-module given in (\ref{E:star}) and (\ref{E:star2}),  one obtains an isomorphism
\[  L^2(X_a)^{\epsilon}   \cong \int_{\widehat{\SL}_2}   \dim \sigma_{N, \psi_a} \cdot \theta_{\psi}(\sigma) \cdot 1(\omega_{\sigma} = \epsilon)  \,  d\mu_{\SL_2}(\sigma). \]
for $\epsilon = \pm$. Summing over $\epsilon$, we obtain  the desired isomorphism in the proposition.
\end{proof}
 \vskip 10pt
 
 \vskip 10pt

\subsection{\bf A commutative diagram.}
Examining the proof of  Proposition \ref{P:L2p}, the unitary isomorphism  there can be explicated as follows. Given $f \in S(X_a)$, we first chooce $\Phi \in S(V)$ such that 
\[   |a|^{1/2} \cdot {\rm rest}(\Phi) = f. \]
Then the image of $f$ under the isomorphism of Proposition \ref{P:L2p} is represented by the measurable section of the direct integral decomposition  given by
\[  \sigma \mapsto    |a|^{1/2}  \cdot \ell_{\sigma, \psi_a}( \theta_{\sigma} ( \Phi) ), \]
where $\theta_{\sigma}$ is as given in (\ref{E:theta}) and $\ell_{\sigma, \psi_a}$ is the $\psi_a$-Whittaker functional arising from the Whittaker-Plancherel theorem for 
$L^2(N, \psi_a \backslash \SL_2)$. In other words, we have:


\vskip 5pt

\begin{prop}  \label{P:com}
For each $\sigma \in \widehat{\SL_2}_{temp,\psi_a}$, there is a commutative diagram: 
    \[
 \xymatrix{ &  S(V)
 \ar[dl]_{\rm rest} \ar[dr]^{\theta_{\sigma}}& \\
 S(X_a)
  \ar[dr]_{\alpha_{\theta(\sigma),a}} & &
 \sigma \boxtimes \theta(\sigma)  
    \ar[dl]^{\ell_{\sigma, \psi_a}}   & \\
 & \C_{\psi_a} \otimes \theta(\sigma) & 
  }
\]
where $\alpha_{\theta(\sigma),a}$ is the morphism associated to the direct integral decomposition of Proposition \ref{P:L2p}.
\end{prop}
 This proposition gives a precise relation between the transfer of periods in the smooth setting and the spectral decomposition of $L^2(N, \psi_a \backslash \SL_2)$, $L^2(X_a, |\omega_a|)$ and 
$\Omega_{\psi}$ in the $L^2$-theory.  Indeed, it is fairly clear that one has a commutative diagram as in the proposition up to scalars. 
The point of the proposition is to explicate the scalar. More precisely, specializing to the case $a=1$, one has:

\vskip 5pt

\begin{cor}  \label{C:com1}
Under the isomorphism 
\[  f_{\sigma} : \Hom_N(\sigma, \psi)  \cong \Hom_{\OO(V)}(C^{\infty}_c(X_1),  \theta_{\psi}(\sigma)) \] 
given in  (\ref{E:fa}) for $\sigma \in \widehat{\SL_2}_{temp, \psi}$ (which is induced by the map  $\theta_{\sigma}$ of (\ref{E:theta})), 
one has:
\begin{equation} \label{E:com1}
 f_{\sigma}( \ell_{\sigma})  =  \alpha_{\theta(\sigma)} \end{equation}
where  the Whittaker functional $\ell_{\sigma}$  is the one in (\ref{E:Whit2}) which intervenes in the Whittaker-Plancherel theorem for $(N, \psi) \backslash \SL_2$ and the morphism  $\alpha_{\theta(\sigma)}$ is the one which intervenes in the spectral decomposition of $L^2(X_1, |\omega_1|)$ obtained in Proposition \ref{P:L2p}.
 \end{cor}
 
 Another way to interpret Corollary \ref{C:com1} is that if one defines the elements $\alpha_{\theta(\sigma)}$ by (\ref{E:com1}) or equivalently by requiring that the diagram in Proposition \ref{P:com} be commutative, then the family $\{ \alpha_{\theta(\sigma)} : \sigma \in \widehat{\SL_2}_{temp,\psi}\}$ induces the spectral decomposition of  $L^2(X_1, |\omega_1|)$ in Proposition \ref{P:L2p}.  
 \vskip 10pt

 We conclude this section with a few formal consequences of Proposition \ref{P:com}.
  The commutative diagram in Proposition \ref{P:com} gives an identity in $\theta(\sigma)$. If we pair both sides of the identity with a vector in $\theta_{\psi}(\sigma)$, using the inner product on $\theta_{\psi}(\sigma)$, we obtain:
 \vskip 5pt
 
\begin{cor}  \label{C:com2}
For any $\Phi \in S(V) = \Omega_{\psi}$ and $w \in \theta(\sigma)$,  one has
\[  \ell_{\sigma} ( B_{\theta(\sigma)}(\Phi, w)) =  \langle \Phi|_X, \beta_{\theta(\sigma)} (w)  \rangle_X,\]
where $B_{\theta(\sigma)}$ was defined in \S \ref{SS:AB}.  
\end{cor}

\begin{proof}
We have
\[ 
 \langle \Phi|_X, \beta_{\theta(\sigma)} (w)  \rangle_X  
   = \langle \alpha_{\theta(\sigma)}(\Phi|_X), w \rangle_{\theta(\sigma)}  
=   \langle \ell_{\sigma}(\theta_{\sigma}(\Phi)), w \rangle_{\theta(\sigma)}  
=  \ell_{\sigma} \left( \langle  \theta_{\sigma}(\Phi)), w \rangle_{\theta(\sigma)}  \right)  
=  \ell_{\sigma}(B_{\theta(\sigma)}(\Phi, w)).  
\] 
\end{proof}
\vskip 5pt

We may also ``double-up" the commutative diagram in Proposition \ref{P:com} and contract the resulting doubled identity using the inner product on $\theta_{\psi}(\sigma)$. This gives:
\vskip 5pt

\begin{cor}  \label{C:com3}
For $\Phi_1, \Phi_2 \in S(V) = \Omega_{\psi}$, one has:
\[ J_{\theta(\sigma)} (\Phi_1|_X, \Phi_2|_X) = \int^*_N  \overline{\psi(n)}  \cdot  J_{\sigma}^{\theta}( n \cdot \Phi_1, \Phi_2) \, \, dn.  
\]
For fixed $\Phi_1|_X$ and $\Phi_2|_X$ in $S(X)$, the $\C$-valued function $\sigma \mapsto J_{\theta(\sigma)} (\Phi_1|_X, \Phi_2|_X)$ is  continuous in $\sigma \in \widehat{\SL_2}_{temp, \psi}$.
\end{cor}

\begin{proof}
We have
 \begin{align}
&J_{\theta(\sigma)} (\Phi_1|_X, \Phi_2|_X) \notag \\
=  &\langle \alpha_{\theta(\sigma)}(\Phi_1|_X),  \alpha_{\theta(\sigma)} (\Phi_2|_X) \rangle_{\theta(\sigma)}  \hskip 55pt  \text{(definition of $J_{\theta(\sigma)}$)}  \notag \\
= & \langle  \ell_{\sigma}( \theta_{\sigma}(\Phi_1)),  \ell_{\sigma}( \theta_{\sigma}(\Phi_2)) \rangle_{\theta(\sigma)}   \hskip 65pt  \text{(by Proposition \ref{P:com})} \notag \\
= &\ell_{\sigma} \otimes \overline{\ell_{\sigma}} \left(  \langle  \theta_{\sigma}(\Phi_1),  \theta_{\sigma}(\Phi_2) \rangle_{\theta(\sigma)}    \right) \hskip 55pt \text{(clear)}  \notag \\
= &   \int^*_N  \overline{\psi(n)} \cdot  \langle n \cdot  \theta_{\sigma}(\Phi_1),  \theta_{\sigma}(\Phi_2) \rangle_{ \sigma \otimes \theta(\sigma)}  \, \, dn  \hskip 10pt \text{(formula for $\ell_{\sigma} \otimes \overline{\ell_{\sigma}}$)} \notag \\
=& \int^*_N  \overline{\psi(n)} \cdot J^{\theta}_{\sigma}(n \cdot \Phi_1, \Phi_2)  \, \, dn  \hskip65pt   \text{(definition of $J^{\theta}_{\sigma}$)} \notag
\end{align}
The continuity of $\sigma \mapsto J_{\theta(\sigma)} (\Phi_1|_X, \Phi_2|_X)$  follows from the above formula, together with Lemma \ref{L:cont2} and Lemma \ref{L:cont3}.
\end{proof}
\vskip 5pt

\noindent The last two corollaries thus give different variants of the identity in Proposition \ref{P:com}.

\vskip 10pt

\section{\bf Relative Characters} \label{S:RC}
 In this section, we briefly recall the notion of the relative character  associated to a period in its various incarnations. 
 \vskip 5pt

\subsection{\bf Relative characters}  \label{SS:RC}
Suppose that, for $i =1$ or $2$, $H_i \subset G$ is a  subgroup of $G$ and $\chi_i: H_i(F) \rightarrow S^1$ a unitary character of $H_i(F)$. We fix also Haar measures $dg$ on $G$ and $dh_i$ on $H_i$.
 For any $\pi \in \widehat{G}$ and  $L_i\in \Hom_{H_i}(\pi, \chi_i)$,  one can associate a distribution on $G$ as follows. 
 Given $(f_1, f_2) \in C^{\infty}_c(G) \times C^{\infty}_c(G)$, one sets:
 \begin{equation}
   \mathcal{B}_{\pi, L_1, L_2}(f_1, f_2)  = \sum_{v \in {\rm ONB}(\pi)}  \overline{L_1(\pi(\overline{f_1})(v))} \cdot L_2(\pi( \overline{f_2})(v)) \end{equation}
 where the sum runs over an orthonormal basis of $\pi$. The sum defining $\mathcal{B}_{\pi, L_1, L_2}(f_1, f_2)$ is independent of the choice of the orthonormal basis.  It gives a  linear  map
 \[  \mathcal{B}_{\pi, L_1, L_2}: C^{\infty}_c(G)  \otimes \overline{C^{\infty}_c(G)} \longrightarrow  \C \]
 which is $G(F)^{\Delta}$-invariant (and which depends on the Haar measure $dg$).
 \vskip 5pt

 The distribution $\mathcal{B}_{\pi, L_1, L_2}$ is called the relative character of $\pi$ with respect to $(L_1, L_2)$.
Note that, in the literature,   it is frequent  to find a different convention in the definition of the relative character, using instead the sum
 \[ \sum_{v \in {\rm ONB}(\pi)}  L_1( \pi(f_1)(v)) \cdot  \overline{L_2(\pi(f_2)(v))}. \]
 The difference between the two conventions is merely one of form rather than substance, and it is easy to convert from one convention to the other using complex conjugation. 
 We choose the normalisation given above so as to avoid the appearance of multiple complex conjugations in later formulae.  
 \vskip 5pt
 
  Now a short computation gives
 \[  \overline{L_i( \pi(\overline{f_i}) v)}  =  \int_{H_i \backslash G}   \overline{L_i(\pi(g_i)(v))}  \cdot (f_i)_{H_i,  \chi_i}(g) \,\,  \frac{dg}{dh_i} \] 
 with
 \[  (f_i)_{H_i, \chi_i}(g)  = \int_{H_i}  f(h_i g)  \cdot  \overline{\chi_i(h_i)} \, dh_i. \]
Hence one deduces that the linear form $\mathcal{B}_{\pi, L_1, L_2}$ factors as
 \[  
 C^{\infty}_c(G)  \otimes \overline{C^{\infty}_c(G)} \twoheadrightarrow C^{\infty}_c(H_1, \chi_1 \backslash G) \otimes \overline{  C^{\infty}_c(H_2, \chi_2 \backslash G) } \longrightarrow  \C. \]
 We may think of $C^{\infty}_c(H_i, \chi_i \backslash G(F))$ as the space of compactly supported smooth sections of the line bundle on $X_i = H_i \backslash G$ determined by 
 $\chi_i$ and denote this space by the alternative notation $C^{\infty}_c(X_i, \chi_i)$. Then we shall think of $\mathcal{B}_{\pi, L_1, L_2}$ as a
 $G^{\Delta}$-invariant  linear form on $C^{\infty}_c(X_1, \chi_1) \otimes \overline{C^{\infty}_c(X_2, \chi_2)}$; this now depends on the Haar measures $dg$ and $dh_i$, or rather on the $G$-invariant measure $dg/dh_i$ on $H_i \backslash G$. 
 \vskip 5pt
 
 Let us write:
 \[  f_{L,v}(g)  = L(\pi(g) v)  \]
for  the matrix coefficient associated to $L\in \pi^*$ and $v \in \pi$.  Then the distribution $\mathcal{B}_{\pi, L_1, L_2}$ is given by the formula
 \begin{equation}  \label{E:RC}
  \mathcal{B}_{\pi, L_1, L_2}( \phi_1, \phi_2)  = \sum_{v \in {\rm ONB}(\pi)}  \langle \phi_1, f_{L_1,v} \rangle_{X_1} \cdot \langle f_{L_2,v},  \phi_2 \rangle_{X_2}, \end{equation}
  where $\langle-, - \rangle_{X_i}$ is the inner product defined by the measure $dg/dh_i$.
  \vskip 10pt
 
  \subsection{\bf Alternative incarnation}  \label{SS:RC2}
 We can also give an alternative formulation of the notion of relative characters. 
 Continuing with the context of \S \ref{SS:RC},  it is not difficult to verify that 
  \[   \mathcal{B}_{\pi, L_1, L_2}(f_1, f_2)  =  \sum_{v \in {\rm ONB}(\pi)}    \overline{L_1( \pi(\overline{f_1 \ast f_2^{\vee}})(v))} \cdot  L_2(v), \]
   where  
 \[    f_2^{\vee}(g)  = \overline{f_2(g^{-1})}, \]
 and 
 \[  (f_1 \ast f_2)(g)  =\int_{G} f_1(g x^{-1}) \cdot f_2(x) \, \, dx  \]
 is the convolution of $f_1$ and $f_2$. 
 \vskip 5pt
 Thus, we may  alternatively define $\mathcal{B}_{\pi, L_1, L_2}$ as a linear form
 \[  \mathcal{B}_{\pi, L_1, L_2}: C^{\infty}_c(G) \longrightarrow   \C  \]
 given by the formula
 \[  \mathcal{B}_{\pi, L_1, L_2}(f)  = \sum_{v \in {\rm ONB}(\pi)} \overline{L_1( \pi(\overline{f})(v)} \cdot L_2(v). \]
 As in \S \ref{SS:RC}, this linear form factors as:
 \[  \mathcal{B}_{\pi, L_1, L_2} : C^{\infty}_c(G) \twoheadrightarrow C^{\infty}_c(X_1, \chi_1) \longrightarrow   \C,  \] 
 so that we may regard it as a linear form on $C^{\infty}_c(X_1, \chi_1)$, given by the formula
 \begin{equation}  \label{E:B}
   \mathcal{B}_{\pi, L_1, L_2} (\phi) =     \sum_{v \in {\rm ONB}(\pi) } \langle \phi, f_{L_1, v} \rangle_{X_1} \cdot f_{L_2, v}(1). \end{equation}
  In fact, it further factors as:
 \[  \mathcal{B}_{\pi, L_1, L_2} :   C^{\infty}_c(X_1, \chi_1) \twoheadrightarrow C^{\infty}_c(X_1, \chi_1)_{H_2, \chi_2}  \longrightarrow \C. \]
From this alternative description of the relative character, we can recover the previous version discussed in the previous subsection by using the fact that any $f \in C^{\infty}_c(G)$ can be expressed as a finite linear combination of $f_1 \ast f_2$. This is clear in the nonarchimedean case and is a result of Dixmier-Malliavin in the archimedean case.

 \vskip 10pt

\subsection{\bf $J_{\sigma}$ as a relative character}
  We shall now relate the notion of relative character with the theory of direct integral decomposition.
\vskip 5pt

  We shall focus on the case when $H_1 = H_2 = H$ and $\chi_1 = \chi_2 =\chi$ and such that $\dim \Hom_H(\pi, \chi) \leq 1$. 
 With $X = H \backslash G$, equipped with a $G$-invariant measure $dx$,  suppose one has a direct integral decomposition: 
 \[  L^2(X, \chi, dx)  := L^2( H, \chi \backslash G) = \int_{\Omega} \sigma(\omega) \, \, d\mu(\omega) \]
 with  associated families of maps $\{ \alpha_{\sigma(\omega)} \}$ and $\{ \beta_{\sigma(\omega)} \}$ (see \S \ref{SS:alpha})  and
   associated decomposition of inner product as in (\ref{E:inner}):
 \[  \langle- , -\rangle_X =  \int_{\Omega} J_{\sigma(\omega)}(-, -)  \, d\mu(\omega). \]
 Observe that the positive semidefinite Hermitian form $J_{\sigma(\omega)}$  is a $G^{\Delta}$-invariant linear form
 \[  J_{\sigma(\omega)}:  C^{\infty}_c(X, \chi) \otimes \overline{C^{\infty}_c(X, \chi)} \longrightarrow \C. \]
 This suggests that $J_{\sigma(\omega)}$ may be regarded as a relative character according to our definition in \S \ref{SS:RC}. Indeed, one has:
 \vskip 5pt
 
 \begin{lem}
 One has $J_{\sigma(\omega)}  = \mathcal{B}_{\sigma(\omega), \ell_{\sigma(\omega)}, \ell_{\sigma(\omega)}}$, 
 where  $\ell_{\sigma(\omega)}  = ev_1 \circ \beta_{\sigma(\omega)}  \in \Hom_H(\sigma(\omega), \chi)$. 
  \end{lem}
  \vskip 5pt
  
  \begin{proof}
  Since $\omega$ is fixed in the proposition, we shall write $\sigma = \sigma(\omega)$ for simplicity. Now
  we have: 

  \begin{align}
  J_{\sigma}(\phi_1, \phi_2) =  \langle \alpha_{\sigma}(\phi_1) , \alpha_{\sigma}(\phi_2) \rangle_{\sigma} 
     &= 
  \sum_{v \in {\rm ONB}(\sigma)} \langle  \alpha_{\sigma}(\phi_1) , v \rangle_{\sigma} \cdot \langle v,  \alpha_{\sigma}(\phi_2) \rangle_{\sigma} \notag \\
  &=  \sum_{v \in {\rm ONB}(\sigma)} \langle  \phi_1, \beta_{\sigma}(v) \rangle_X  \cdot  \langle \beta_{\sigma}(v), \phi_2 \rangle_X. \notag  
  \end{align} 
  Noting that
  \[  \beta_{\sigma}(v) (g)  = ev_1 \circ \beta_{\sigma} (\sigma(g)(v))  = \ell_{\sigma}( \sigma(g) (v)) = f_{\ell_{\sigma}, v}(g),  \]
  we see that the lemma follows by the equation (\ref{E:RC}).
  \end{proof}
   \vskip 10pt

 We can also work with the alternative context of \S \ref{SS:RC2}. In this incarnation, one has:
 \vskip 5pt
 
 \begin{lem}  \label{L:altrel}
 As a linear form on $C^{\infty}_c(X, \chi)$, one has
 \[  \mathcal{B}_{\sigma(\omega), \ell_{\sigma(\omega)}, \ell_{\sigma(\omega)}} (\phi)  =   \ell_{\sigma(\omega)}( \alpha_{\sigma(\omega)}(\phi))  = \beta_{\sigma(\omega)} \alpha_{\sigma(\omega)}(\phi)(1). \]
  \end{lem}
  \begin{proof}
  We write $\sigma = \sigma(\omega)$ for simplicity.
  Then we have
  \[  \alpha_{\sigma}(\phi)  = \sum_{v \in {\rm ONB}(\sigma)}  \langle \alpha_{\sigma}(\phi), v \rangle_{\sigma} \cdot v  = \sum_v  \langle \phi, f_{\ell_{\sigma}, v} \rangle_X \cdot v. \]
  Hence
  \[  \ell_{\sigma}(\alpha_{\sigma}(\phi)) =    \sum_v  \langle \phi, f_{\ell_{\sigma}, v} \rangle_X \cdot \ell_{\sigma}( v), \]
\vskip 5pt
\noindent   so that the lemma follows by equation (\ref{E:B}).
     \end{proof}
     \vskip 5pt
     
 \begin{cor}
 Let $\mathcal{C}(X, \chi)$ be the Harish-Chandra-Schwarz space of $X = (H, \chi) \backslash G$. Then  the relative character $\mathcal{B}_{\sigma, \ell_{\sigma}, \ell_{\sigma}}$ extends to $\mathcal{C}(X, \chi)$. 
\end{cor}
 \vskip 10pt
 
\subsection{\bf Space of orbital integrals}  \label{SS:oi}
Set
 \[  \mathcal{I}(X, \chi) :=  C^{\infty}_c(X, \chi)_{H, \chi}. \]
 We think of this as ``the space of  orbital integrals" on $X$. Indeed,   given a $H$-orbit on $X$, the associated orbital integral factors to $\mathcal{I}(X, \chi)$.  As noted above,  the relative character $\mathcal{B}_{\sigma, \ell_{\sigma}, \ell_{\sigma}}$ factors to give a linear form on $\mathcal{I}(X, \chi)$.
 Henceforth, we will write  $\mathcal{B}_{\sigma, \ell_{\sigma}}$ in place of $\mathcal{B}_{\sigma, \ell_{\sigma}, \ell_{\sigma}}$ to simplify notation.

\vskip 15pt

\section{\bf Transfer of Test Functions}
If two periods on the two members of a dual pair  are related by theta correspondence as in Proposition \ref{P:smoothp}, then one might ask if the associated relative characters are related in a precise way. Such a relation is called a relative character identity.   To compare the two relative characters in question, which are distributions on different spaces, we first need to define a correspondence of the relevant spaces of test functions.
\vskip 10pt

    \subsection{\bf A Correspondence of Test Functions.}  \label{SS:corr}
  The considerations of the previous sections suggest that one considers the following maps. Set
  \[  p:  S(V) = \Omega^{\infty} \longrightarrow  C^{\infty}(N,\psi \backslash \SL_2)  \]
  given by
  \[  p(\Phi) (g)  :=  (g \cdot \Phi)(v_1). \]
 This map is $\OO(v_1^{\perp})$-invariant and $\SL_2$-equivariant.  Let us set
 \begin{equation} \label{E:test}
    \mathcal{S}(N,\psi\backslash \SL_2)   := \text{image of $p$}, \end{equation}
 noting that it is a $\SL_2$-submodule. 
  Likewise, consider the $\OO(V) \times (N, \psi)$-equivariant  {\em surjective} restriction map
  \[  q = {\rm rest}:  S(V) \longrightarrow S(X_1) \]
  so that
  \[  q(\Phi)(h) = (h \cdot \Phi)(v_1) = \Phi(h^{-1} \cdot v_1) \quad \text{ for $h \in  \OO(v_1^{\perp}) \backslash \OO(V) \cong X_1$.} \]   
     We have already seen and used the map $q$ in the setting of smooth theta correspondence, seeing that it induces an $\OO(V)$-equivariant isomorphism
     \[  q:  S(V)_{N,\psi} \cong S(X_1)\]
     
Hence we have the diagram:
     \begin{equation}  \label{E:diagram}
 \xymatrix{ &  S(V)
 \ar[dl]_p \ar[dr]^q & \\
 \mathcal{S}(N, \psi \backslash \SL_2)
  & &
 S(X_1)  
  }
\end{equation}
 We now make a definition: 
\vskip 5pt

\begin{defn}  \label{D:trans}
Say that $f \in \mathcal{S}(N, \psi \backslash \SL_2)$ and $ \phi  \in S(X_1) $ are in correspondence (or are transfers of each other) if there exists $\Phi \in S(V)$ such that $p(\Phi)  =f$ and $q(\Phi) = \phi$. 
\end{defn}
   \vskip 5pt
   
 Our  goal in this section is to  establish some  basic properties of the spaces of test functions and the transfer defined above. We start with the following simple observation.
 \begin{prop}
 Every $f \in \mathcal{S}(N, \psi \backslash \SL_2)$ has a transfer $\phi \in S(X_1)$ and vice versa.  
 \end{prop}

\begin{proof}
This is simply because the maps $p$ and $q$ above are surjective. 
\end{proof}
 \vskip 5pt

 We also note:
 \vskip 5pt
 
 \begin{lem}  \label{L:sch}
 The space $\mathcal{S}(N,\psi\backslash \SL_2)$ is contained in the Harish-Chandra-Schwarz space of the Whittaker variety $(N, \psi) \backslash \SL_2$. In particular, for any $\sigma \in \widehat{\SL_2}_{temp,\psi}$, the associated relative character $\mathcal{B}_{\sigma, \ell_{\sigma}}$ extends to a linear form on $\mathcal{S}(N,\psi\backslash \SL_2)$.
 \end{lem}
   
    \begin{proof}
   From the formula defining  the Weil representation,  we see that for $f = p(\Phi)$,
   \[       | f(t(a) k) | =   |a|^{ \dim V / 2}  \cdot |(k \cdot \Phi)(a \cdot v_1)|. \] 
   It follows by Lemma \ref{L:SL2}(i) that $f \in \mathcal{C}(N, \psi \backslash \SL_2)$ if $\dim V \geq 3$.    
       \end{proof}
     \vskip 10pt
     
     \subsection{\bf Basic function and fundamental lemmas} \label{SS:basicFL}
     We shall now place ourselves in the unramified situation. Namely, let us assume that:
     \vskip 5pt
     \begin{itemize}
     \item $F$ is a nonarchimedean local field of residual characteristic different from $2$;
     \item the conductor of the additive character $\psi: F \rightarrow S^1$ is the ring of integers $\mathcal{O}_F$ of $F$;
       \item the quadratic space $V$ contains a self-dual lattice $L$ and $v_1 \in L$.
     \end{itemize}
Under these hypotheses, we have: 
\vskip 5pt

\begin{itemize}
\item  the measure $d_{\psi}x$ of $F$ is such that the volume of $\mathcal{O}_F$ is $1$;
\item the measure $d_{\psi}v$ on $V$ is such that the volume of $L$ is $1$;
\item the stabilizer $K' = K'_L$ of $L$ in $\OO(V)$ is a hyperspecial maximal compact subgroup.
 \end{itemize} 
 
 We let $\Phi_0  \in S(V)$ be the characteristic function of $L$, so that $\Phi_0$ is a unit vector in $\Omega_{\psi}$. Here is a basic definition: 
      
    \vskip 5pt
    
    \begin{defn}  \label{D:basic}
    Set 
    \[  f_0 = p(\Phi_0) \quad \quad \text{and} \quad \quad \phi_0 = q(\Phi_0). \]
    We call these the basic functions in the relevant space of test functions.
    \end{defn}
    \vskip 5pt
    
    Observe that $\phi_0 = q(\Phi_0)$ is the characteristic function of $X_1(F) \cap L$. On the other hand, $f_0$ is not compactly supported. Indeed: $f_0$ is determined by its value on $T$ and we have
    \[   f_0(t(a))  =( t(a) \cdot \Phi_0)(v_1)  = \begin{cases}
|a|^{\dim V /2}, \text{ if $|a| \leq 1$;} \\
0, \text{ if $|a| > 1$.} \end{cases} \]
    It is also immediate from definition that one has the following ``fundamental lemma":
    \vskip 5pt
     \begin{lem}  \label{L:fund}
     The basic functions $f_0$ and $\phi_0$ correspond.
     \end{lem}
     
   Let $K = \SL_2(\mathcal{O}_F) \subset \SL_2(F)$ and $K' = K'_L \subset \OO(V)$, so that they are hyperspecial maximal compact subgroups which 
   fix the unramified vector $\Phi_0$.  Endow $\SL_2$ and $\OO(V)$ with Haar measures such that the volumes of $K$ and $K'$ are $1$. Then
 we have the corresponding spherical Hecke algebras $\mathcal{H}(\SL_2, K)$ and $\mathcal{H}(G, K')$. These are commutative unital algebras whose units are the characteristic functions $1_K$ and $1_{K'}$ respectively.  
 \vskip 5pt
 
   It was shown by Howe (see \cite[Chap. 5, Thm. I.4, Pg 103]{MVW} and \cite[Pg 107]{MVW}) that
  \begin{equation} \label{E:howe}
   \Omega_{\psi}^K  = C^{\infty}_c(\OO(V)) \cdot \Phi_0 \quad \text{and} \quad \Omega_{\psi}^{K'}  = C^{\infty}_c(\SL_2) \cdot \Phi_0. \end{equation}
 By  applying $1_{K'}$ and $1_K$ respectively to these two equations, it follows  that
  \[   (\Omega^{\infty})^{K \times K'}  = \mathcal{H}(\SL_2, K) \cdot \Phi_0   =  \mathcal{H}(\OO(V),K') \cdot \Phi_0. \]
   It also follows from (\ref{E:howe})  that if  one has a nonzero equivariant map
  \[  \Omega^{\infty} \twoheadrightarrow \sigma \otimes \pi  \in {\rm Irr}(\SL_2 \times \OO(V)), \]
  then $\sigma$ is $K$-unramified if and only if $\pi = \theta_{\psi}(\sigma)$ is $K'$-unramified. 
   Indeed, it was shown by Rallis \cite[\S 6]{R}  that there is an algebra morphism
  \[  c:  \mathcal{H}(G, K') \longrightarrow \mathcal{H}(\SL_2, K)  \]
  such that for any $f \in \mathcal{H}(G, K')$, one has
  \[      f \cdot \Phi_0   = c(f) \cdot \Phi_0. \]
  From this, one easily deduces the following ``fundamental lemma for spherical Hecke algebras":
  \vskip 5pt
  
  \begin{lem}  \label{L:fund2}
  For any $f \in \mathcal{H}(G, K')$,
  the element $f \cdot \phi_0 \in C^{\infty}_c(X_1)$  corresponds to the element $c(f) \cdot f_0 \in \mathcal{S}(N, \psi \backslash \SL_2)$.
    \end{lem}
     
     \vskip 10pt
     
     \subsection{\bf Relation with Adjoint L-factors.}  \label{SS:adjoint}
     
  We shall see that the space $\mathcal{S}(N, \psi \backslash \SL_2)$  is intimately related to the standard (degree $3$) L-factor of irreducible representations of $\SL_2$. 
   Let us recall a certain Rankin-Selberg  local zeta integral for this particular L-factor, due to Gelbart-Jacquet \cite{GJ}.  It requires the following 3 pieces of data:

  \vskip 5pt
  \begin{itemize}
  \item a $\psi$-generic  $\sigma \in {\rm Irr}(\SL_2)$ with $\psi$-Whittaker model $\mathcal{W}_{\sigma}$,  
  \item the Weil representation $\omega_{\psi}$ of $\Mp_2$ acting on the space $C^{\infty}_c(F)$ (regarding $F$ as a 1-dimensional quadratic space equipped with the quadratic form $x \mapsto x^2$);
  \item a principal series representation $I_{\psi}(\chi, s)$ of $\Mp_2$, consisting of genuine functions $\phi_s: N \backslash \Mp_2 \rightarrow \C$ such that $\phi_s(t(a)g) = \chi_{\psi}(a) \cdot \chi(a) \cdot  |a|^{1+s} \cdot \phi_s(g)$ (where $\chi_{\psi}$ is a genuine character of the diagonal torus of $\Mp_2$ defined using the Weil index).  
  \end{itemize}
  Then for $f \in \mathcal{W}_{\sigma}$, $\varphi \in C^{\infty}_c(F)$ and  a section $\phi_s \in I(s)$, one can consider the local zeta integral
  \[  Z(f, \varphi, \phi_s)
  =  \int_{N \backslash \SL_2} \overline{f(g)} \cdot   (g \cdot\varphi)(1) \cdot   \phi_s(g)  \, \, \frac{dg}{dn}.  \]
  This converges when $Re(s) \gg 0$, and when $\sigma$ is tempered, it converges for $Re(s) >0$. Moreover, the GCD of this family of local zeta integrals is used to define the local twisted adjoint L-factor $L(s+\frac{1}{2}, \sigma, Ad \times \chi) = L(s + \frac{1}{2}, \sigma, std \times \chi)$.
  \vskip 5pt
  
 Hence, the (twisted) adjoint L-value $L(s+\frac{1}{2}, \sigma, Ad \times \chi)$ is obtained by considering the integrals of $f \in \mathcal{W}_{\sigma}$ against a space of functions $\mathcal{S}_s(N,\psi \backslash \SL_2)$
of the form
\[  g \mapsto \phi_s(g)  \cdot (g \cdot \varphi)(1). \]  
Moreover, as an $\SL_2$-module, $\mathcal{S}_s(N,\psi \backslash \SL_2)$ is a quotient of $\omega_{\psi} \otimes I_{\psi}(\chi, s)$. 
  \vskip 5pt
  
  Now let us return to our space of test functions  $\mathcal{S}(N, \psi \backslash \SL_2)$.  Let us write
  \[  V = \langle v_1 \rangle \oplus U,  \quad   \quad \text{with $U = v_1^{\perp}$.} \]
  Then $C^{\infty}_c(V)  =  C^{\infty}_c(Fv_1) \otimes C^{\infty}_c(U)$. Here, $C^{\infty}_c(Fv_1)$ affords the Weil representation $\omega_{\psi}$ of $\Mp_2 \times \OO(v_1)$ whereas $C^{\infty}_c(U)$ afford a Weil representation of $\Mp_2 \times \OO(U)$.   If $\Phi \in C^{\infty}_c(V)$ is of the form $\Phi_1 \otimes \Phi'$, then 
  \[  p(\Phi)(g)  =  (g \cdot \Phi_1)(v_1)  \cdot (g \cdot \Phi')(0). \]
 The function
 \[  g \mapsto  \phi_{\Phi'}(g) = (g \cdot \Phi')(0)  \]
 belongs to the principal series $I_{\psi}(\chi_{{\rm disc}(U)}, \frac{1}{2}(\dim V -3))$. Indeed, by a result of Rallis \cite[Thm. II.1.1]{R2}, the map
 \[  \Phi' \mapsto \phi_{\Phi'} \]
 gives an $\OO(U)$-invariant, $\Mp_2$-equivariant injective map
 \[  0 \ne C^{\infty}_c(U)_{\OO(U)} \hookrightarrow I_{\psi}(\chi_{{\rm disc}(U)},  \frac{1}{2}(\dim V -3)). \]
This result of Rallis underlies the theory of the doubling seesaw and the Siegel-Weil formula. When $\dim V \geq 4$, this injective map is surjective as well.
Indeed, when $\dim V > 4$, the relevant principal series $I_{\psi}(\chi_{{\rm disc}(U)}, \frac{1}{2}(\dim V -3))$ is irreducible. If $\dim V = 4$, the principal series $I_{\psi}(\chi_{{\rm disc}(U)}, 1/2)$ has length $2$ with unique irreducible quotient the even Weil representation $\omega_{\psi, \chi_{{\rm disc}(U)}}^+$ and unique submodule a (twisted) Steinberg representation. The above map is nonetheless surjective, as the small theta lift of the trivial representation of $\OO(U)$ is equal to $\omega_{\psi, \chi_{{\rm disc}(U)}}^+$.  
  \vskip 5pt
  
  To summarise, we have more or less shown:

     \begin{prop}  \label{P:adjoint}
When $\dim V \geq 4$, the map $p$ factors as
\[ C^{\infty}_c(V) \twoheadrightarrow C^{\infty}_c(V)_{\OO(U)}   \cong \omega_{\psi} \otimes I_{\psi}(\chi_{{\rm disc}(U)}, \frac{1}{2}(\dim V -3)) \twoheadrightarrow   
\mathcal{S}(N,\psi \backslash \SL_2), \]
so that
\[   \mathcal{S}(N, \psi \backslash \SL_2)  = \mathcal{S}_{\frac{1}{2}(\dim V - 3)}(N,\psi \backslash \SL_2).  \]

    \end{prop}

\begin{proof}
One has an  $\Mp_2 \times \Mp_2$-equivariant isomorphism
\[  S(V)_{\OO(U)}  =  S(F v_1)  \widehat{\otimes} S(U)_{\OO(U)} \cong \omega_{\psi} \widehat{\otimes} I_{\psi}(\chi_{{\rm disc}(U)} , \frac{1}{2}(\dim V -3)). \]
The rest of the proposition follows from our preceding discussion. 
\end{proof}
     \vskip 10pt

     \noindent{\bf Question:}  Is the surjective map in Proposition \ref{P:adjoint} in fact an isomorphism
     \[   C^{\infty}_c(V)_{\OO(U)}   \cong  \mathcal{S}(N, \psi \backslash \SL_2)? \]
     \vskip 5pt
         \vskip 5pt
     
   \subsection{\bf Orbital Integrals.}
   Let us set
   \[  \mathcal{I}(N,\psi \backslash \SL_2) := \mathcal{S}(N, \psi \backslash \SL_2)_{N, \psi} \]
   so that by definition, a linear functional on $ \mathcal{I}(N,\psi \backslash \SL_2)$ is a $(N,\psi)$-equivariant linear form on $ \mathcal{S}(N, \psi \backslash \SL_2)$.
   One may think of  $ \mathcal{I}(N,\psi \backslash \SL_2)$ as the space of orbital integrals (with respect to the $(N,\psi)$-period) and write $\mathcal{I}(f)$ for the image of $f$ in
    $\mathcal{I}(N,\psi \backslash \SL_2)$.
   Likewise, we set
   \[  \mathcal{I}(X_1) = S(X_1)_{\OO(U)}. \]
   This may be regarded as the space of orbital integrals with respect to the $\OO(U)$-period and we write $\mathcal{I}(\phi)$ for the image of $\phi$ in  $\mathcal{I}(X_1)$.
   \vskip 5pt
   
   The following proposition summaries the properties of the transfer of test functions:
   \vskip 5pt
   
   \begin{prop}
   The composite map
   \[  \begin{CD}
   S(V) @>p>> \mathcal{S}(N, \psi \backslash \SL_2) @>>>  \mathcal{I}(N,\psi \backslash \SL_2) \end{CD} \]
   factors through $q$, i.e.. it induces a linear map
   \[  S(X_1) \longrightarrow \mathcal{I} (N, \psi \backslash \SL_2).  \]
   Hence the transfer correspondence descends to a linear map when one passes to the space of orbital integrals in the target. Indeed, it further descends to give a surjective  linear map
   \[ t_{\psi}:  \mathcal{I}(X_1) \longrightarrow  \mathcal{I} (N, \psi \backslash \SL_2).  \]
\end{prop}

   \vskip 5pt
  
  \begin{proof}
  The composite map in question is  $(N,\psi)$-invariant, and hence factors through $S(V)_{N, \psi} \cong S(X_1)$. But it is also $\OO(U)$-invariant and  so further factors through
 \[  (S(V)_{N, \psi})_{\OO(U)}  = \mathcal{I}(X_1) \] 
  as desired.
  \end{proof}
  
  \vskip 5pt 
   Likewise, one may consider the composite
   \[  \begin{CD}
   S(V) @>q>> S(X_1) @>>> \mathcal{I}(X_1)  \end{CD} \]
   which as above factors through $S(V)_{\OO(U)}$. But now we do not know if $S(V)_{\OO(U)} \cong \mathcal{S}(N,\psi\backslash \SL_2)$; see the Question at the end of the previous subsection. If the answer to that question is Yes, then we will likewise conclude that the above composite map  induces a linear map
   \[   \mathcal{S}(N, \psi \backslash \SL_2)  \longrightarrow  \mathcal{I}(X_1), \]
   which descends further to 
   \[   \mathcal{I}(N, \psi \backslash \SL_2)  \longrightarrow  \mathcal{I}(X_1), \]
  In that case,  this  linear map will be inverse to the one in the proposition, and hence we will have an isomorphism of vector spaces:
  \[   t_{\psi}:  \mathcal{I}(N,\psi \backslash \SL_2) \cong   \mathcal{I}(X_1). \]
  In other words, the transfer correspondence would give an isomorphism of the space of orbital integrals (for the relevant spaces of test functions).  As it stands, we only have the surjective transfer map 
  \[  t_{\psi}: \mathcal{I}(X_1) \twoheadrightarrow  \mathcal{I} (N, \psi \backslash \SL_2) \]  
given in the above proposition.

\vskip 15pt

\section{\bf Relative Character Identities}
Finally, we are ready to establish the following relative character identity, which is the main local result of this paper. 
\vskip 5pt

\begin{thm} \label{T:main}
Suppose that
\vskip 5pt

\begin{itemize}
\item $f \in \mathcal{S}(N,\psi \backslash \SL_2)$ and $\phi \in S(X_1)$ are in correspondence;
\vskip 5pt
\item $\sigma \in \widehat{\SL_2}_{temp,\psi}$ with (nonzero) theta lift $\theta_{\psi}(\sigma) \in \widehat{\OO(V)}$;
\vskip 5pt
\item $\ell_{\sigma} \in \Hom_N(\sigma,\psi)$ is the canonical element determined in (\ref{E:Whit2}) by the Whittaker-Plancherel theorem;
\vskip 5pt
\item $\ell_{\theta(\sigma)} = f_{\sigma}(\ell_{\sigma}) \in \Hom_{\OO(v_1^{\perp})}(\theta(\sigma), \C)$ is the canonical element determined by the spectral decomposition in Proposition \ref{P:L2p} (which is in turn determined by $\ell_{\sigma}$ and $\theta_{\sigma}$). 
\end{itemize}
Then one has the character identity:
\[  \mathcal{B}_{\sigma, \ell_{\sigma}}(f)  = \mathcal{B}_{\theta(\sigma), \ell_{\theta(\sigma)}}(\phi). \]
More succintly, one has the identity 
\[   \mathcal{B}_{\sigma, \ell_{\sigma}} \circ t_{\psi}  =  \mathcal{B}_{\theta(\sigma), \ell_{\theta(\sigma)}} \]
of linear forms on     $ \mathcal{I}(X_1)$ or equivalently the identity
\[    \mathcal{B}_{\sigma, \ell_{\sigma}} \circ p  =   \mathcal{B}_{\theta(\sigma), \ell_{\theta(\sigma)}} \circ q  \]
of linear forms on $S(V)$. 
\end{thm}
See \S \ref{S:RC}, especially \S \ref{SS:oi},  for the definition of $ \mathcal{B}_{\sigma, \ell_{\sigma}}$ and $ \mathcal{B}_{\theta(\sigma), \ell_{\theta(\sigma)}}$.

\vskip 15pt
\subsection{\bf Proof of Theorem \ref{T:main}}
This subsection is devoted to the proof of the theorem. 
 With $f$ and $\phi$ as given in the theorem, choose $\Phi \in C^{\infty}_c(V)$ such that $f = p(\Phi)$ and $\phi = q(\Phi)$.
 We shall now find two different expressions for $\Phi(v_1)$. 
 
 \vskip 5pt

On one hand, by the direct integral decomposition given in Proposition \ref{P:L2p}, (\ref{E:point}) gives
\[  \Phi(v_1)  = \phi(v_1)  =  \int_{\widehat{\SL_2}}  \beta_{\theta(\sigma)} \alpha_{\theta(\sigma)}(\phi)(v_1)  \, d\mu_{\SL_2, \psi}(\sigma). \]
On the other hand, by the Whittaker-Plancherel theorem for $(N,\psi) \backslash \SL_2$, (\ref{E:point}) gives
\[  \Phi(v_1) = f(1)  =  \int_{\widehat{\SL_2}} \beta_{\sigma} \alpha_{\sigma}( f)(1) \,  d\mu_{\SL_2, \psi}(\sigma). \]
Comparing the two expressions, we deduce that
\begin{equation}  \label{E:rel}
 \int_{\widehat{\SL_2}}  \mathcal{B}_{\theta(\sigma), \ell_{\theta(\sigma)}} (\phi) \, \,  d\mu_{\SL_2, \psi}(\sigma)  =  \int_{\widehat{\SL_2}}  \mathcal{B}_{\sigma, \ell_{\sigma}}(f)  \,  \,  
 d\mu_{\SL_2, \psi}(\sigma). \end{equation} 
We would like to remove the integral sign in the above identity. For this, we will apply a Bernstein center argument.
\vskip 5pt

Given an arbitrary element $z$ in the Bernstein center (or the ring of Arthur multipliers in the archimedean case) of $\SL_2 \times \OO(V)$, the element $z$ acts on the irreducible representation $\sigma \boxtimes \theta_{\psi}(\sigma)$ by a scalar $z(\sigma \boxtimes \theta_{\psi}(\sigma))$.  This implies that one has a commutative diagram
\begin{equation} \label{E:bern1}
  \begin{CD}
 \Omega_{\psi}  @>z>>  \Omega_{\psi}  \\
 @V\theta_{\sigma}VV  @VV\theta_{\sigma}V  \\
 \sigma \boxtimes \theta_{\psi}(\sigma)  @>z(\sigma \boxtimes \theta_{\psi}(\sigma))>>   \sigma \boxtimes \theta_{\psi}(\sigma)  \\
@V{\lambda}VV   @VV{\lambda}V   \\
\sigma  @>z(\sigma \boxtimes \theta_{\psi}(\sigma))>>  \sigma  \end{CD} \end{equation}
for any linear form $\lambda$ on $\theta_{\psi}(\sigma)$. One has an analogous commutative diagram where one takes $\lambda$ to be any linear form on $\sigma$ (so the last row of the commutative diagram has $\theta_{\psi}(\sigma)$ in place of $\sigma$). 
\vskip 5pt

Now what we would like to show is that there are commutative diagrams
\begin{equation}  \label{E:bern2}
 \begin{CD}
\Omega_{\psi} @>z>> \Omega_{\psi} \\
@VpVV   @VVpV  \\
\mathcal{S}(N,\psi \backslash \SL_2)    @.   \mathcal{S}(N,\psi \backslash \SL_2)  \\
@V\alpha_{\sigma}VV  @VV\alpha_{\sigma}V  \\
\sigma  @>z(\sigma \boxtimes \theta_{\psi}(\sigma))>> \sigma     \end{CD}  
\quad   \quad \text{and} \quad \quad 
 \begin{CD}
\Omega_{\psi} @>z>> \Omega_{\psi} \\
@VqVV   @VVqV  \\
S(X_1)    @.   S(X_1) \\
@V\alpha_{\sigma}VV  @VV\alpha_{\sigma}V  \\
\theta_{\psi}(\sigma)  @>z(\sigma \boxtimes \theta_{\psi}(\sigma))>> \theta_{\psi}(\sigma)     \end{CD} 
 \end{equation}
 We shall explain how the commutativity of the diagram on the left follows from the commutativity of the diagram in (\ref{E:bern1}); a similar argument works for the diagram on the right. 
 \vskip 5pt
 
Since  the map $\alpha_{\sigma} \circ p$ is $\SL_2$-equivariant,  it factors  through $\theta_{\sigma}:  \Omega_{\psi} \longrightarrow \sigma \boxtimes \theta_{\psi}(\sigma)$, i.e. there is a 
$\lambda :  \theta_{\psi}(\sigma) \rightarrow \C$ such that 
\[  \alpha_{\sigma} \circ p  = \lambda \circ \theta_{\sigma}. \]
 Using this, we see that the desired commutativity of the left diagram in (\ref{E:bern2})  is reduced to the commutativity of the diagram in (\ref{E:bern1}).
 \vskip 5pt
 
 Now we shall apply the identity (\ref{E:rel}) to the pair of test functions arising from $z \cdot \Phi$.  Note that
 \[  \mathcal{B}_{\sigma, \ell_{\sigma}}(p(z \cdot \Phi)) = \beta_{\sigma} \alpha_{\sigma} (p(z \Phi))(1)  = z(\sigma \boxtimes \theta_{\psi}(\sigma)) \cdot  \mathcal{B}_{\sigma, \ell_{\sigma}}(f) \]
 and
 \[  \mathcal{B}_{\theta(\sigma), \ell_{\theta(\sigma)}}  = \beta_{\theta(\sigma)} \alpha_{\theta(\sigma)} (q (z \cdot \Phi))(x_1) =  
  z(\sigma \boxtimes \theta_{\psi}(\sigma)) \cdot  \mathcal{B}_{\theta(\sigma), \ell_{\theta(\sigma)}}(\phi). \]
Hence the identity (\ref{E:rel}), when applied to $z \cdot \Phi$, reads:
\begin{equation} \label{E:rel2}
  \int_{\widehat{SL}_2}  z(\sigma \boxtimes \theta_{\psi}(\sigma))  \cdot \left(  \mathcal{B}_{\sigma, \ell_{\sigma}}(f)  -  \mathcal{B}_{\theta(\sigma), \ell_{\theta(\sigma)}}(\phi) \right) 
d\mu_{\SL_2, \psi}(\sigma) = 0.  \end{equation}
\vskip 5pt

Now note that there is a natural homomorphism from the Bernstein center of $\SL_2$ to the Bernstein center for $\SL_2 \times \OO(V)$. Hence we may take $z$ to be an element in the (tempered) Bernstein center of $\SL_2$.   Then $z(\sigma \boxtimes \theta_{\psi}(\sigma))  = z(\sigma)$.  When regarded as  $\C$-valued functions on $\widehat{\SL_2}_{temp, \psi}$, the  elements $z$ of the (tempered) Bernstein center of $\SL_2$,   are dense in the space of all Schwarz functions on $\widehat{\SL_2}_{temp,\psi}$.    Hence, (\ref{E:rel2}) implies that for $d\mu_{\SL_2,\psi}$-almost all $\sigma$, one has
\[   \mathcal{B}_{\sigma, \ell_{\sigma}}(f)  = \mathcal{B}_{\theta(\sigma), \ell_{\theta(\sigma)}}(\phi). \]
To obtain the equality for all $\sigma \in \widehat{\SL_2}_{temp,\psi}$, we note that both sides of the identity are continuous as functions of $\sigma \in \widehat{\SL_2}_{temp,\psi}$ by Lemma \ref{L:cont2} and Corollary \ref{C:com3}.  This completes the proof of Theorem \ref{T:main}.
\vskip 10pt

\subsection{\bf Some consequences}
We shall now give some consequences of the relative character identity shown in Theorem \ref{T:main}. Let us consider the following  diagram:

   \[
  \xymatrix{ &  S(V)
  \ar[dl]_{q}    \ar[d]^{\theta_{\sigma}}     \ar[dr]^{p}&\\ 
   S(X_1)
  \ar[d]_{\alpha_{\theta(\sigma)}} & {\sigma \boxtimes \theta_{\psi}(\sigma)}    \ar[dl]^{\ell_{\sigma}}    \ar[dr]_{\ell_{\theta(\sigma)}} 
    & { \mathcal{S}(N,\psi \backslash \SL_2)} \ar[d]^{\alpha_{\sigma}} \\
{\theta_{\psi}(\sigma)} \ar[dr] _{\ell_{\theta(\sigma)}} &   &     {\sigma}  \ar[dl]^{\ell_{\sigma}}  &\\
& {\C} &  \\
  }
\]
In this diagram, the rhombus at the bottom is clearly commutative. Now the parallelogram at the upper left side is precisely the commutative diagram in Proposition \ref{P:com}.  On the other hand, the parallelogram at the upper right side is   commutative up to a scalar since  
\[  \dim \Hom_{\OO(v_1^{\perp}) \times \SL_2}( \Omega_{\psi},  \C \boxtimes \sigma)  =1  \quad \text{for $\sigma \in \widehat{\SL_2}_{temp,\psi}$} \]
and both $\alpha_{\sigma} \circ p$ and $\ell_{\theta(\sigma)} \circ \theta_{\sigma}$ are nonzero elements of  this space.
We would like to show that it is in fact commutative.
\vskip 5pt

 To deduce this, we observe that the composite of the three maps along the left boundary of the hexagon is simply the relative character $\mathcal{B}_{\theta(\sigma)} \circ q$, whereas the composite of the three maps along the right boundary of the hexagon is the relative character $\mathcal{B}_{\sigma} \circ p$.  The relative character identity of Theorem \ref{T:main} says that the boundary of the diagram is commutative! From this, we deduce the following counterpart of Proposition \ref{P:com}:
 \vskip 5pt
 
 \begin{prop} \label{P:othercom}
  The following diagram is commutative:
    \[
 \xymatrix{ &  S(V)
 \ar[dl]_{\theta_{\sigma}} \ar[dr]^{p}& \\
 \sigma \otimes \theta_{\psi}(\sigma)
  \ar[dr]_{\ell_{\theta(\sigma)}} & &
 { \mathcal{S}(N,\psi \backslash \SL_2)}  
    \ar[dl]^{\alpha_{\sigma}} & \\
 & { \sigma} & 
  }
\]
  \end{prop}
\vskip 5pt

Pairing the above identity with an element $v \in \sigma$, we obtain the following counterpart of  Corollary \ref{C:com2}:
\vskip 5pt

\begin{cor}  \label{C:othercom2}
For any $\Phi  \in \Omega_{\psi}$ and $v \in \sigma$,  one has:
\[  \ell_{\theta(\sigma)} \left( A_{\sigma}(\Phi, v) \right) =  \langle p(\Phi),  \beta_{\sigma}(v) \rangle_{N \backslash \SL_2} \]
where $A_{\sigma}$ is the map defined in \S \ref{SS:AB}.
\end{cor}
 
\vskip 15pt

\section{\bf Local L-factor} \label{S:L}
In this section, we are going to examine the local L-factor $L_{X}(s, - )$ associated to a spherical variety $X$.  As we explained in the introduction, this local L-factor is associated to a $\frac{1}{2}\Z$-graded representation $V_X = \oplus_d V_X^d$ of the dual group $X^{\vee}$ and its value has been computed by Sakellaridis \cite{S1, S2} in great generality. However, we shall show that  for the particular case treated in this paper, the results developed thus far through theta correspondence can be used to compute $L_{X}(s, -)$ in terms of the analogous local L-factor for the Whitaker variety $(N,\psi) \backslash \SL_2$. 
\vskip 5pt

\subsection{\bf Unramified setting.} 
 We place ourselves in the unramified setting of \S \ref{SS:basicFL}, so that 
 \vskip 5pt
 \begin{itemize}
 \item $F$ is a nonarchimedean local field of residual characteristic not $2$;  
 \item $\psi$ has conductor $\mathcal{O}_F$, so that the associated measure $d_{\psi}x$ of $F$ gives $\mathcal{O}_F$ volume $1$;
 \item  $L \subset V$ is a self-dual lattice with stabilizer $K' = K'_L$, so that  the measure $d_{\psi}v$ on $V$ gives $L$ volume $1$. 
 \item  the lattice $L$  endows $X = X_1$ with the structure of a smooth scheme over $\mathcal{O}_F$; 
  \item the vector $v_1$ lies in the lattice $L$.
 \end{itemize}
 Hence,    $v_1 \in X(\mathcal{O}_F) = X(F) \cap L$ and  we have an orthogonal decomposition
 \[   L = \mathcal{O}_F v_1 \oplus  (v_1^{\perp} \cap L). \]
  The lattices $L$ and $L \cap v_1^{\perp}$ (which are both self-dual) endow $\OO(V)$ and $\OO(v_1^{\perp})$ with $\mathcal{O}_F$-structures so that they become smooth group schemes over $\mathcal{O}_F$; in particular, $K' = \OO(V)(\mathcal{O}_F)$. 
 Then the map $h \mapsto h^{-1} \cdot v_1$ defines an $\OO(V)$-equivariant isomorphism 
 \[ X \longrightarrow \OO(v_1^{\perp}) \backslash \OO(V) \]
 of smooth schemes over $\mathcal{O}_F$.  Moreover, as a consequence of Hensel's lemma, $K'$ acts transitively on $X(\mathcal{O}_F)$, so that 
 \[ X(\mathcal{O}_F) \cong \OO(v_1^{\perp})(\mathcal{O}_F) \backslash \OO(V)(\mathcal{O}_F). \]
  
  \vskip 5pt
  
  Moreover, the Haar measure $|\omega|$ on $X$ is associated to an $\OO(V)$-invariant differential of top degree which has nonzero reduction on the special fiber.  
  Further, if we equip the smooth group schemes $\OO(V)$ and $\OO(v_1^{\perp})$ over $\mathcal{O}_F$ with invariant differentials $\omega_{\OO(V)}$ and $\omega_{\OO(v_1^{\perp})}$ of top degree with nonzero reduction on the speical fibers, then 
 \[  |\omega|  =   \frac{|\omega_{\OO(V)}|}{|\omega_{\OO(v_1^{\perp})}|}. \]
  This means that
   \begin{equation} \label{E:volume}
    {\rm Vol}(X(\mathcal{O}_F); |\omega|) := \int_{X(\mathcal{O})}  |\omega|  =  q^{-\dim X} \cdot  \# X( \kappa_F) = q^{-\dim X} \cdot \frac{\#\OO(V)(\kappa_F)}{\#\OO(v_1^{\perp})(\kappa_F)} \end{equation}
   where $\kappa_F$ is the residue field of $F$ and $q = \# \kappa_F$.

 \vskip 5pt
 \subsection{\bf The L-factor $L^{\#}_X$}
 From the spectral decomposition of $L^2(X, |\omega|)$ obtained in Proposition \ref{P:L2p}, we have a family of $\OO(v_1^{\perp})$-invariant linear functionals
$\ell_{\theta(\sigma)}$ on $\theta_{\psi}(\sigma)$ for $\sigma \in \widehat{\SL_2}_{temp,\psi}$. 
We remind the reader that even in this unramifed setting that we have placed ourselves, the linear functional $\ell_{\theta(\sigma)}$ depends on the Haar measure $dg$ on $\SL_2$.   
We have also specified an inner product $\langle-, -\rangle_{\theta(\sigma)}$ in Proposition \ref{P:zeta} (using the doubling zeta integral) and this depends on the Haar measure $dg$ on $\SL_2$ as well.  
\vskip 5pt

Let us assume that $\sigma \in \widehat{\SL_2}_{temp,\psi}$ is $K$-unramified, where $K = \SL_2(\mathcal{O}_F)$.
Fix $v_0 \in \sigma$ such that
\[ \langle v_0, v_0 \rangle_{\sigma} = 1 \quad \text{and} \quad  K \cdot v_0 = v_0. \]
Then $\theta_{\psi}(\sigma)$ is $K'$-unramified and we fix  $w_0 \in \theta_{\psi}(\sigma)$ with
\[     \langle w_0, w_0 \rangle_{\theta(\sigma)} =1 \quad \text{and} \quad  K' \cdot w_0 = w_0. \]
We then set
\[    L^{\#}_X(\sigma) := |\ell_{\theta(\sigma)}(w_0)|^2 \in \mathbb{R}_{\geq 0}. \]
Our goal is to determine this non-negative valued function defined on the $K$-unramified part of $\widehat{\SL_2}_{temp,\psi}$, where $K = \SL_2(\mathcal{O}_F)$.
\vskip 5pt

According to the conjecture of Sakellaridis and Venkatesh, one should have
\[   L^{\#}_X(\sigma) =  \Delta(0)  \cdot  \frac{L_X(1/2, \sigma)}{ L(1, \sigma, Ad)} \]
where $\Delta(s)$ is a product of local zeta factors (which is independent of the representation $\sigma$), $L(s, \sigma, Ad)$ is the adjoint L-factor of $\sigma$ and 
\[  L_X(s, \sigma) =  \prod_d L(s+d, \sigma, V^d_X) \]
is the L-factor of $\sigma$ associated to a $1/2\cdot  \Z$-graded representation $V_X = \oplus_d V_X^d$ of $X^{\vee}$. This is the essential part of $L^{\#}_X(\sigma)$.  
The computation of $L_X(1/2,\sigma)$ is thus equivalent to the precise determination of $L^{\#}_X(\sigma)$.
\vskip 5pt

\subsection{\bf  Some constants}
We shall determine $L^{\#}_X(\sigma)$ in terms of the analogous quantity for the Whittaker variuety $(N, \psi) \backslash \SL_2$. 
To this end, let $\Phi_0 = 1_L \in S(V)$ be the characteristic function of $L$ which is a unit vector in $\Omega_{\psi}$ and is $K\times K'$-invariant.
 Then under the canonical map
 \[  \theta_{\sigma} : \Omega_{\psi} \longrightarrow \sigma \otimes \theta_{\psi}(\sigma), \]
 we have
  \begin{equation} \label{E:constant2} 
     \theta_{\sigma}(\Phi_0) = c_{\sigma} \cdot v_0 \otimes w_0 \in \sigma \otimes \theta_{\psi}(\sigma). \end{equation}
    for some nonzero constant $c_{\sigma}$. 
    \vskip 5pt
    
    On the other hand, recall that we have the basic functions
\[   p(\Phi_0) = f_0 \in \mathcal{S}(N, \psi \backslash \SL_2) \quad \text{and} \quad q(\Phi_0) = \phi_0 \in S(X). \]
We have observed that
\[   \phi_0 = 1_{X(\mathcal{O}_F)}. \]
Now we define a constant $\lambda_{\theta(\sigma)}$  by:
\begin{equation} \label{E:constant}
   \alpha_{\theta(\sigma)}(\phi_0) = \lambda_{\theta(\sigma)} \cdot w_0. \end{equation}
    \vskip 5pt

\subsection{\bf Key computations}
We now perform the following computations:

\vskip 5pt

\begin{itemize}
\item[(a)]  Taking inner product of both sides of (\ref{E:constant}) with $w_0$ gives:
\[    \lambda_{\theta(\sigma)} = \langle \alpha_{\theta(\sigma)}(\phi_0), w_0 \rangle_{\theta(\sigma)}. \]
Now the right hand side of this identity is equal to 
\[ \langle \phi_0, \beta_{\theta(\sigma)}(w_0)  \rangle_X = \langle 1_{X(\mathcal{O}_F)},  \beta_{\theta(\sigma)}(w_0) \rangle_X
 = \beta_{\theta(\sigma)}(w_0)(v_1) \cdot {\rm Vol}(X(\mathcal{O}_F), |\omega|). \]
 Hence we have
 \[   \lambda_{\theta(\sigma)}  = \ell_{\theta(\sigma)}(w_0) \cdot {\rm Vol}(X(\mathcal{O}_F), |\omega|). \]
\vskip 5pt

\item[(b)] On the other hand, applying the commutative diagram in Proposition \ref{P:com} to the left hand side of (\ref{E:constant})  and using (\ref{E:constant2})  gives:
\[  c_{\sigma} \cdot \ell_{\sigma} (v_0) \cdot w_0 =  \ell_{\sigma}(\theta_{\sigma}(\Phi_0))  =  \lambda_{\theta(\sigma)} \cdot  w_0, \]
so that
\[    \lambda_{\theta(\sigma)} =  c_{\sigma} \cdot \ell_{\sigma} (v_0)  \]
\vskip 5pt

\item[(c)] Finally, taking inner product of both sides of (\ref{E:constant2}) with $v_0$ gives:
\[ A_{\sigma}(\Phi_0, v_0) = \langle \theta_{\sigma}(\Phi_0), v_0 \rangle_{\sigma} = c_{\sigma} \cdot w_0 \]
Computing inner product of both sides gives:
\[   |c_{\sigma}|^2 = \langle A_{\sigma}(\Phi_0, v_0),  A_{\sigma}(\Phi_0, v_0)\rangle_{\theta(\sigma)}  = Z_{\sigma}(\Phi_0, \Phi_0, v_0, v_0), \]
where the last equality is (\ref{E:zeta}) and $Z_{\sigma}(-)$ is the doubling zeta integral. 
\end{itemize}
Combining the last identities resulting from (a), (b) and (c) above, we obtain:
\[   |\ell_{\theta(\sigma)}(w_0)|  = |\ell_{\sigma}(v_0)| \cdot |Z(\Phi_0, \Phi_0, v_0, v_0)|^{1/2}  \cdot {\rm Vol}(X(\mathcal{O}_F), |\omega|)^{-1}. \]
Hence, it remains to determine the 3 quantities on the right hand side. We have already determined the volume of $X(\mathcal{O}_F)$ in (\ref{E:volume}). 
On the other hand, we have:
\vskip 5pt

\begin{lem}. \label{L:unram}
(i) Suppose $\dim V= 2n \geq 4$ is even. Then
\[ |\ell_{\sigma}(v_0)|^2 = \frac{\zeta_F(2)}{L(1, \sigma, Ad)} \]  
and
\[  Z(\Phi_0, \Phi_0, v_0, v_0) = {\rm Vol}(K; dg) \cdot  \frac{L(n-1, \sigma \times \chi_{\disc V}, std)}{\zeta_F(2n-2) \cdot L(n, \chi_{\disc V})}  \]
where ${\rm Ad} = {\rm std}$ is the adjoint $=$ standard L-factor for $\SL_2$.
\vskip 5pt

(ii) Suppose  $\dim V = 2n+1 \geq 3$ is odd, so that one is working with $\Mp_2$ instead of $\SL_2$. Then
\[  |\ell_{\sigma}(v_0)|^2 = \frac{\zeta_F(2) \cdot L_{\psi}(1/2, \sigma, std)}{L_{\psi}(1, \sigma, Ad)} \] 
and
\[    Z(\Phi_0, \Phi_0, v_0, v_0) = {\rm Vol}(K; dg) \cdot  \frac{L_{\psi}(n-\frac{1}{2}, \sigma \times \chi_{\disc V}, std)}{\zeta_F(2n)}  \]
\end{lem}
\begin{proof}
The determination of $|\ell_{\sigma}(v_0)|$ was carried out in \cite[Prop. 2.14 and \S2.6]{LM} whereas that of the doubling zeta factor can be found in \cite[Prop. 3]{LR} and \cite[Prop. 6.1]{G2}. 
\end{proof}
\vskip 5pt

From this lemma, one sees that dependence of $\ell_{\theta(\sigma)}$ on the Haar measure $dg$. In the unramified setting, it is customary to take the Haar mesure $dg$ such that ${\rm Vol}(K; dg) =1$. However, in view of the global applications later on, we prefer to take the Haar measure associated to an invariant differential of top degree on $\SL_2$ over 
$\mathbb{Z}$.  In that case, we have
\[  {\rm Vol}(K, dg) =   q^{-3} \cdot \# \SL_2(\kappa_F) = \zeta_F(2)^{-1}. \]
\vskip 5pt

Putting everything together, we have shown:
\vskip 5pt

\begin{prop} \label{P:Lfactor}
(i) When $\dim V = 2n$ is even, one has
\[   L^{\#}_X(\sigma) = |\ell_{\theta(\sigma)}(w_0)| ^2 =  \frac{L(n-1, \sigma \times \chi_{\disc V}, std)}{L(1, \sigma , Ad)} \cdot  \frac{L(n, \chi_{\disc V})}{\zeta_F(2n-2)} , \]
taking note that ${\rm Ad} = {\rm std}$ for $\SL_2$.
\vskip 5pt

(ii) When $\dim V = 2n+1$ is odd, so that one is working with $\Mp_2$, 
\[  L^{\#}_X(\sigma) = |\ell_{\theta(\sigma)}(w_0)| ^2 =  \frac{L_{\psi}(n - \frac{1}{2}, \sigma \times \chi_{\disc V}, std) \cdot L_{\psi}(1/2, \sigma, std)}{L_{\psi}(1, \sigma, Ad)} \cdot  \frac{\zeta_F(2n)}{L(n, \chi_{\disc V})^2 } . \]
\end{prop}
\vskip 5pt

As an illustration, when $\dim V = 3$ (so $n=1$) and $\disc(V) =1$, one gets
\[  L^{\#}_X(\sigma) = \frac{L_{\psi}(1/2,\sigma, std)^2}{L_{\psi}(1,\sigma, Ad)} \cdot \frac{\zeta_F(2)}{\zeta_F(1)^2} \]
whereas when $\dim V = 4$ (so $n=2$) and $\disc (V) =1$, one has
\[ L^{\#}_X(\sigma)  = 1. \]
The reader should compare these values with those given the table (3) in the introduction of \cite{S4}.
\vskip 10pt

\subsection{\bf General case} So far, we have placed ourselves in the unramified setting. We now return to the general setting and define 
\[ L^{\#}_X : \widehat{\SL_2}  \longrightarrow \C \quad \text{or} \quad   \widehat{\Mp_2} \longrightarrow  \C \]
by using the formulae given in Proposition \ref{P:Lfactor}, depending on whether $\dim V$ is even or odd. 
\vskip 15pt

\section{\bf Transfer in Geometric Terms}
We have defined the transfer of test functions and established a relative character identity without making any geometric comparison. This is not so surprising, as the theta correspondence is a means of transferring spectral data from one group to another. Nonetheless, one can ask for an explicit formula for the transfer map 
\[  t_{\psi} : \mathcal{I}(X_1) \longrightarrow  \mathcal{I}(N,\psi\backslash \SL_2). \]
For example, we may wonder if one could describe $t_{\psi}$  as an integral transform.  We shall derive such a formula in this section, assuming that $F$ is nonarchimedean (with ring of integers $\mathcal{O}_F$ and uniformizer $\varpi$). 
We also assume for simplicity that the conductor of the additive character $\psi$ is $\mathcal{O}_F$ and the discriminant of $V$ is $1$. In particular, the measure $d_{\psi}x = dx$ on $F$ gives $\mathcal{O}_F$ volume $1$. 
\vskip 5pt

Recall that we have called  the domain and target of $t_{\psi}$ the spaces of orbital integrals. To describe $t_{\psi}$ geometrically, we shall appeal to incarnations of 
these spaces as concrete spaces of functions. Consider for example the case of $\mathcal{I}(N,\psi\backslash \SL_2) =\mathcal{S}(N,\psi \backslash \SL_2)_{N,\psi}$. 
Given a function $f \in \mathcal{S}(N,\psi \backslash \SL_2)$,  we may consider its literal $(N,\psi)$-orbital integral:
\[  \mathcal{I}(f)(a) = \int_F  f(wt(a)n(b)) \cdot \overline{\psi(n)} \, db. \]
Assuming this converges, it defines a smooth function on  the open Bruhat cell $NwB$ which is $(N,\psi)$-invariant on both sides. Hence it is determined by its value on $wT$ and we 
may regard it as a function on $F^{\times}$. The map $\mathcal{I}$ factors as:
\[ \mathcal{S}(N,\psi \backslash \SL_2) \longrightarrow  \mathcal{I}(N,\psi \backslash \SL_2)   \longrightarrow C^{\infty}(F^{\times}), \] 
and we view it as giving an incarnation of the elements of $ \mathcal{I}(N,\psi \backslash \SL_2)$ as functions on $F^{\times}$.
Likewise, we shall later see an incarnation of the elements of $\mathcal{I}(X_1)$, as functions on a set of generic $\OO(v_1^{\perp})$-orbits on $X_1$. 
\vskip 5pt

Given $\Phi \in S(V)$, we would thus like to compute the $(N,\psi)$-orbital integral of $p(\Phi) = f \in \mathcal{S}(N,\psi\backslash \SL_2)$: 
\[  \mathcal{I}(f)(a) = \int_F  f(wt(a)n(b)) \cdot \overline{\psi(n)} \, db  =  \int_F (wt(a) n(b) \cdot \Phi) (v_1)  \cdot \overline{\psi(b)} \, db. \]
 We should perhaps say a few words about the convergence of this integral.   Let us identify $N \backslash \SL_2$ with $W^* = F^2 \setminus O$  (where $O$ is the origin of $F^2$) via $g \mapsto (0, 1) \cdot g$. Then $|f|$ is a function on $W^*$ which vanishes on a neighbourhood of $O$.    Now the element $N w t(a) n(b) \in N \backslash \SL_2$   corresponds to the element $(-a, -ab) \in W^*$.   For fixed $a \in F^{\times}$, the function 
\[   b \mapsto f(w t(a) n(b))  \]
is thus not necessarily compactly supported on $F$.  However, if we had assumed that $f \in C^{\infty}_c(N , \psi\backslash \SL_2)$ 
(which is a dense subspace of $\mathcal{C}(N, \psi \backslash \SL_2)$), then $|f|$ would in addition vanish outside a compact set of $W$, so that the above function of $b$ is compactly supported on $F$ and the integral defining $\mathcal{I}(f)(a)$ would have been convergent. 
This  suggests that  if we let $U_n = \varpi^{-n} \mathcal{O}_F$ and set
\[   \mathcal{I}_n (f)(a)  =    \int_{U_n}  (wt(a) n(b) \cdot \Phi) (v_1)  \cdot \overline{\psi(b)} \, db, \]
 then the value $\mathcal{I}_n(f)(a)$ should stabilize for sufficiently large $n$ (and this does happen for $f \in C^{\infty}_c(N , \psi\backslash \SL_2)$).  With this motivation, 
 we shall define 
 \[   
 \mathcal{I}(f)(a)  := \lim_{n \to \infty}   \mathcal{I}_n (f)(a)  \]
and shall show below that the right hand side indeed stabilizes. 
\vskip 5pt

For this, we will perform an explicit computation:
\begin{align}
\mathcal{I}_n( f)(a)  &=  \int_{U_n} (wt(a) n(b) \cdot \Phi)(v_1)  \cdot \overline{\psi (b)} \, db   \notag \\
&= \int_{U_n} \mathcal{F}(t(a) n(b) \cdot \Phi)(v_1)  \cdot \overline{\psi (b)} \, db   \notag \\
&= \gamma_{\psi, q} \cdot \int_{U_n}  \int_V  (t(a) n(b) \cdot \Phi)(y)  \cdot \psi(\langle v_1, y \rangle) \cdot  \overline{\psi (b)} \, dy \,  db   \notag \\
&=  \gamma_{\psi, q} \cdot  \int_{U_n} \int_V  |a|^{\frac{1}{2} \dim(V)} \cdot  (n(b) \cdot \Phi)(a y)  \cdot  \psi(\langle v_1, y \rangle) \cdot  \overline{\psi (b)} \, dy \,  db   \notag \\
&=  \gamma_{\psi, q} \cdot  \int_{U_n} \int_{a^{-1} \omega}  |a|^{\frac{1}{2} \dim(V)} \cdot  \Phi(ay) \cdot \psi(a^2b q(y))  \cdot \psi(\langle v_1, y \rangle) \cdot  \overline{\psi (b)} \, dy \,  db   \notag \\
&=  \gamma_{\psi, q} \cdot  \int_{ \omega}  |a|^{-\frac{1}{2} \dim(V)}  \cdot  \Phi(x)  \cdot \psi(\langle v_1, a^{-1}x \rangle) \cdot  \left( \int_{U_n} \psi( b ( q(x) -1)) \, db \right) \, dx \notag
\end{align}
where $\omega = {\rm supp}(\Phi)$ is compact and we have made the substitution $x =ay$ in the last step. Recall also that $\mathcal{F}$ is a normalized Fourier transform giving the action of the standard Weyl group element $w$ on the Weil representation and $\gamma_{\psi, q}$ is a root of unity (a Weil index). 
\vskip 5pt

Now let us consider the inner integral
\[  \int_{U_n} \psi( b ( q(x) -1)) \, db   \]
If $q(x) -1 \notin \varpi^n \mathcal{O}_F$, then the integrand is a nontrivial character of $U_n$ and hence the integral is $0$. On the other hand, if $q(x) -1 \in \varpi^n \mathcal{O}_F$, the integral gives the volume of $U_n = \varpi^{-n} \mathcal{O}_F$. Since $\psi$ is assumed to have conductor $\mathcal{O}_F$, the volume of $U_n$ with respect to the measure $db$  is $q^n$ (where $q$ is the size of the residue field of $F$).
Hence
\[    \int_{U_n} \psi( b ( q(x) -1)) \, db  =  q^n \cdot 1(q(x)  \in 1+ \varpi^n \mathcal{O}_F) \]   
and so
\begin{align}
\gamma_{\psi,q}^{-1} \cdot \mathcal{I}_n(f)(a)  &=q^n \cdot   |a|^{-\frac{1}{2} \dim(V)} \cdot  \int_{V}  \Phi(x)  \cdot \psi(\langle v_1, a^{-1}x \rangle) \cdot  1(q(x) -1 \in \varpi^n \mathcal{O}_F) \, dx
\notag  \\
&=  q^n  \cdot   |a|^{-\frac{1}{2} \dim(V)} \cdot  \int_{ q^{-1}(1 + \varpi^n \mathcal{O}_F)}  \Phi(x)  \cdot \psi(\langle v_1, a^{-1}x \rangle)  \, dx.  \notag 
\end{align}
Now this last expression is a quantity which appears in the theory of local densities in the theory of quadratic forms over local fields.  
Indeed, consider the map
\[  q:  q^{-1}(1 + \varpi^n \mathcal{O}_F)  \longrightarrow  1 + \varpi^n \mathcal{O}_F \]
of $p$-adic manifolds. Since every point in the base is a regular value of the map $q$, or equivalently $q$ is submersive at every point of the domain, the integral  of the compactly supported and locally constant integrand over 
$ q^{-1}(1 + \varpi^n \mathcal{O}_F)$ can be performed by first integrating over the fibers of $q$ followed by integration over the base. 
Indeed, this was how we had defined the measures $|\omega_a|$ on each fiber $X_a$ (for $a \in F^{\times}$).
In other words for  any locally constant compactly supported $\varphi$, 
\[  
\int_{ q^{-1}(1 + \varpi^n \mathcal{O}_F)}  \varphi(x) \, dx = \int_{1 + \varpi^n \mathcal{O}_F }   q_*(\varphi)(z) \, \, dz  \]
where
\[  q_*(\varphi)(z)  = \int_{q ^{-1}(z)}  \varphi(x) \cdot |\omega_z(x)| \]
But $q_*(\varphi)$ is a locally constant function on the base. Hence for $n$ sufficiently large, the above integral is simply equal to
\[   {\rm Vol}(\varpi^n \mathcal{O}_F) \cdot q_*(\varphi)(1)= q^{-n} \cdot  \int_{X_1} \varphi(x) \cdot |\omega_1(x)|. \]
Applying this to the integral of interest, we thus deduce that the sequence $\mathcal{I}_n(f)(a)$ stabilizes for large $n$ and 
 \[    \mathcal{I}(f)(a) =   \gamma_{\psi,q} \cdot  |a|^{-\frac{1}{2} \dim(V)} \cdot  \int_{X_1}   \Phi(x)  \cdot  \psi( \langle v_1, a^{-1} x \rangle) \cdot |\omega_1(x)| \]

\vskip 10pt

Now observe that the map
\[ \gamma:  X_1  = \OO(U) \backslash \OO(V) \longrightarrow F  \]
given by
\[  x = h^{-1} v_1  \mapsto   \langle v_1, x \rangle = \langle v_1, h^{-1} v_1 \rangle \]
is $\OO(U)$-invariant (on the right). Moreover, for $\xi \in F$,  the preimage of $\xi$ is equal to
\[   \gamma^{-1}(\xi) = \{  x = \xi \cdot v_1  + v:   v \in U \quad \text{and} \quad q(v) = 1- \xi^2 \}. \]
Outside $\xi^2 =1$, the map $\gamma$ is submersive at all points  and 
it follows by Witt's theorem that  the fiber $\gamma^{-1}(\xi)$  is a homogeneous space under $\OO(U)$. For $x_{\xi}  \in \gamma^{-1}(\xi)$ (with $\xi^2 \ne 1$), its stabilizer in $\OO(U)$ is $\OO(U \cap x_{\xi}^{\perp})$. 
Thus, if $F^{\heartsuit} = F \smallsetminus \{ \pm 1 \}$ and $X_1^{\heartsuit} = \gamma^{-1}(F^{\heartsuit})$, then 
\[  X_1^{\heartsuit}/\OO(U)  \cong F^{\heartsuit}. \]
Moreover, the measures $|\omega_1|$ on $X_1^{\heartsuit}$ and $d\xi$ on $F^{\heartsuit}$ that we have fixed give rise to measures $|\nu_{\xi}|$ on the fibers $\gamma^{-1}(\xi)$. 
The $\OO(U)$-orbital integrals of functions on $S(X_1)$ are thus obtained via integration on the fibres of $\gamma$ and are smooth  functions on $F^{\heartsuit}$,  
These functions give  an incarnation of the elements of $\mathcal{I}(X_1) = S(X_1)_{\OO(U)}$, so that we have a map
\[  \mathcal{I}: S(X_1) \longrightarrow \mathcal{I}(X_1) \longrightarrow C^{\infty}(F^{\heartsuit}). \]
\vskip 5pt

Hence, continuing with our computation, we have:
\begin{align}
\mathcal{I}(f)(a)  &=   \gamma_{\psi,q} \cdot  |a|^{-\frac{1}{2} \dim(V)}   \cdot  \int_{X_1^{\heartsuit} }     \Phi(x)  \cdot  \psi( a^{-1} \gamma(x) ) \cdot |\omega_1(x)|  \notag \\
&=  \gamma_{\psi,q} \cdot    |a|^{-\frac{1}{2} \dim(V)}    \cdot  \int_{F^{\heartsuit}}  \int_{\gamma^{-1}(x_{\xi})}  \Phi( y)   \cdot \psi(a^{-1}\xi) \cdot  |\nu_{\xi}(y)|  \,\, d\xi \notag \\
&=  \gamma_{\psi,q} \cdot     |a|^{-\frac{1}{2} \dim(V)}  \cdot  \int_{F^{\heartsuit}}   \mathcal{I}(\phi)(\xi)  \cdot \psi(a^{-1}\xi) \,  d\xi \notag 
\end{align}
where $\phi = q(\Phi) \in S(X_1)$ (so that $\phi$ and $f$ are transfers of each other) and  $\mathcal{I}(\phi)$ is the orbital integral of $\phi$ defined by the inner integral over the fibers of $\gamma$ over $F^{\heartsuit}$. 
We have shown:
\vskip 5pt

\begin{prop}  \label{P:geom}
The transfer map $t_{\psi}:  \mathcal{I}(X_1) \longrightarrow \mathcal{I}(N ,\psi \backslash \SL_2)$
is given by the integral transform:
\[  t_{\psi}( \phi)(a) =\gamma_{\psi,q} \cdot   |a|^{-\frac{1}{2} \dim(V)}    \cdot  \int_{F^{\heartsuit}}   \phi(\xi)  \cdot \psi(a^{-1}\xi) \,  d\xi. \]
where we have regarded $ \mathcal{I}(X_1) $ and $\mathcal{I}(N ,\psi \backslash \SL_2)$ as spaces of functions on $F^{\heartsuit}$ and $F^{\times}$ respectively. 

\end{prop}

Comparing with the formula for the transfer defined in \cite{S4}, we see that our transfer map $t_{\psi}$ essentially agrees with that of \cite{S4}. In particular, our approach gives an alternative proof of the transfer theorem of \cite{S4} in the setting of hyperboloids.  
\vskip 5pt

We close this section with another remark. As mentioned in the introduction, the transfer map in \cite{S4} was first defined and studied on the level of the boundary degenerations of the rank $1$ spherical varieties and then one uses essentially the same formula in the setting of the spherical varieties themselves.  For the case treated in this paper, the boundary degeneration of $X_1$ is simply the  nullcone (with vertex removed) 
\[  X_0 = \{ 0 \ne x \in V:    q(x) = 0 \} \]
of nonzero isotropic vectors.  This is a homogeneous $\OO(V)$-variety and one can carry out essentially all the analysis of the earlier sections with $X_0$ in place
of $X_a$ with $a \ne 0$. One would then be describing the spectrum of $X_0$ in terms of the spectrum of the basic affine space $N \backslash \SL_2$. Indeed, since the map $q : V \setminus \{0 \} \longrightarrow F$ is submersive at all points, the derivation of the formula for the transfer map given in this section can be carried out essentially uniformly for $X_a$
with any $a \in F$. In other words, the Weil representation allows one to construct a coherent family of transfer maps relating $\mathcal{S}(N, \psi_a \backslash \SL_2)$ and $S(X_a)$ varying smoothly with $a \in F$ (though 
 $S(X_0) \ne C^{\infty}_c(X_0)$ in the nonarchimedean case), which explains in some sense why ``the same formula works". We leave the analysis of the transfer map for the boundary degeneration $X_0$ as an exercise for the interested reader. 
\vskip 15pt

\section{\bf Factorization of Global Periods}  \label{S:global}
In this final section, we turn to the global setting, where we examine the question of factorisation of global period integrals, in the context of the periods considered in the earlier sections. We first need to introduce the global analogs of various constructions encountered in the local setting.
\vskip 5pt

\subsection{\bf Tamagawa measures} \label{SS:Tamagawa}
Let $k$ be a number field with ring of ad\`eles $\A$. We fix a nontrivial unitary character 
\[ \psi: k \backslash \A \longrightarrow S^1. \]
This has a factorization $\psi = \prod_v \psi_v$ where $\psi_v$ is a nontrivial character of the local field $k_v$ for each place $v$ of $k$. Then $\psi_v$ determines  
a self-dual Haar measure $d_{\psi_v}x$ on $k_v$ such that for almost all $v$, the volume of the ring of integers $\mathcal{O}_{k_v}$ relative to $d_{\psi_v}x$ is $1$. The product measure $dx:= \prod_v d_{\psi_v}x$ then gives a measure on $\A$. This is the Tamagawa measure of $\A$: it is independent of $\psi$  (by the Artin-Whaples product formula) and satisfies
\[ \int_{F \backslash \A} dx  =1. \]

If $G$ is a (smooth) algebraic group over $k$, we may consider the adelic group $G(\A)$. It is a restricted direct product $\prod'_v G(k_v)$,  taken with respect to a sequence $\{ K_v = G(\mathcal{O}_{k_v})\}$ of open compact subgroups  determined by any $\mathcal{O}_k$-structure on $G$. For almost all $v$, $K_v$ is a hyperspecial maximal compact subgroup of $G(k_v)$.
\vskip 5pt

Now suppose  $\omega_G$ is a nonzero invariant differential of top degree on $G$ over $k$. Then for each place $v$ of $k$, 
the pair $(\omega_G, \psi_v)$ determines a Haar measure $|\omega_{G, \psi_v}|_v$ of $G(k_v)$.  We would like to consider the product measure
$\prod_v |\omega_{G, \psi_v}|_v$ on $G(\A)$. For this, one needs to assume that $\prod_v {\rm Vol}(K_v; |\omega_{G, \psi_v}|_v)$ is finite.
This is the case for unipotent groups or semisimple groups. 
If the infinite product is not convergent (e.g. if $G = \mathbb{G}_m$), one can still deal with this by introducing ``normalization factors"; we will not go into this well-documented story here. In any case, this product measure on $G(\A)$ is independent of $\psi$ and $\omega_G$. It is the so-called Tamagawa measure of $G$. 
When $G  = \mathbb{G}_a$ is the additive group, the Tamagawa measure on $G(\A) = \A$ is precisely the measure $dx = \prod_v d_{\psi_v} x$ defined above (so that the terminology is used consistently). 
\vskip 5pt

More generally, let $X = H \backslash G$ be a homogeneous $G$-variety over $k$ (with $G$ acting on the right). Assume for simplicity that $X(k_v) = H(k_v) \backslash G(k_v)$ for each place $v$ of $k$ and $X(\A) = H(\A) \backslash G(\A)$.  Suppose further that $\omega_X$ is a nonzero $G$-invariant differential form of top degree on $X$ over $k$.
Then for each $v$, one has a $G(k_v)$-invariant measure $|\omega_{X,\psi_v}|_v$ on $X(k_v)$. We shall call the product measure $\prod_v |\omega_{X,\psi_v}|_v$ (when it makes sense) the Tamagawa measure of $X(\A)$. It is independent of $\psi$ and $\omega_X$. Moreover, it is simply the quotient of the Tamagawa measures of $G(\A)$ and $H(\A)$. Indeed, one can construct an invariant differential $\omega_X$ of top degree as a quotient of (right-)invariant differentials $\omega_G$ and $\omega_H$ of top degree on $G$ and $H$.
 \vskip 5pt

\vskip 5pt

In short, when working with adelic groups or the adelic points of homogeneous $G$-varieties, we shall always use such Tamagawa measures. 

\vskip 5pt
\subsection{\bf Automorphic Forms.}
  For a reductive group $G$ defined over $k$, we shall write $[G]$ for the quotient $G(k) \backslash G(\A)$ and equip it with its Tamagawa measure $dg$ (divided by the counting measure on the discrete subgroup $G(k)$).
\vskip 5pt

Let $C^{\infty}_{mod}([G])$ denote the space of smooth functions on $[G]$ which are of (uniform) moderate growth. It is a representation of $G(\A)$ containing the $G(\A)$-submodule $\mathcal{A}(G)$ of (smooth) automorphic forms on $G$, which in turn contains the submodule $\mathcal{A}_{cusp}(G)$  of cusp forms:
\[  \mathcal{A}_{cusp}(G) \subset \mathcal{A}(G) \subset  C^{\infty}_{mod}([G]). \]
When the group $G$ has a nontrivial split torus in its center, we shall fix a unitary automorphic central character $\chi$  and consider the space $\mathcal{A}_{\chi}(G)$ of automorphic forms with central character $\chi$; we shall suppress this technical issue in the following discussion. 
\vskip 5pt

On $\mathcal{A}_{cusp}(G)$, we have the Petersson inner product 
$\langle-, -\rangle_G$ (defined using the Tamagawa measure $dg$).  Indeed, the Petersson inner product defines a pairing between $\mathcal{A}_{cusp}(G)$ and the larger space $C^{\infty}_{mod}([G])$. Hence, we have  a canonical projection map 
\[ C^{\infty}_{mod}([G])   \longrightarrow \mathcal{A}_{cusp}(G). \]
 In particular, for an irreducible cuspidal representation $\Sigma \subset \mathcal{A}_{cusp}(G)$, we have a projection
\[  {\rm pr}_{\Sigma}:  C^{\infty}_{mod}([G])   \longrightarrow  \Sigma.\]
We denote the restriction of the Petersson inner product on $\Sigma$ by $\langle - , - \rangle_{\Sigma}$. 
\vskip 10pt

\subsection{\bf Global periods}
Let $H \subset G$ be a subgroup so that $X= H \backslash G$ is quasi-affine. Fix a unitary Hecke character $\chi$ of $H$. Then we may consider the global $(H, \chi)$-period:
\[ P_{H,\chi}:  \mathcal{A}_{cusp}(G)  \longrightarrow  \C  \]
defined by
\[  P_{H, \chi}(\phi)  = \int_{[H]} \overline{\chi(h)} \cdot  \phi(h) \, dh \]
where $dh$ is the Tamagawa measure of $H(\A)$.
For a cuspidal representation $\Sigma \subset \mathcal{A}_{cusp}(G)$, we may thus consider the restriction of $P_{H, \chi}$ to $\Sigma$, denoting it by $P_{H, \chi, \Sigma}$. 

\vskip 5pt
\subsection{\bf The maps $\alpha_{\rm Aut}$ and $\beta_{\rm Aut}$.}. \label{S:aabb}
We shall now introduce the global analog of the maps $\alpha_{\sigma}$ and $\beta_{\sigma}$ introduced in \S \ref{SS:alpha} in the local setting. Set $X_{\A} =  H(\A) \backslash G(\A)$, equipped with its Tamagawa measure (which is the quotient of the Tamagawa measures of $G(\A)$ and $H(\A)$). 
 We have a $G(\A)$-equivariant map
\[  \theta:  C^{\infty}_c(X_{\A}, \chi) \longrightarrow C_{mod}^{\infty}([G])  \]
defined by
\[  \theta(f)(g)  = \sum_{x \in X_k}  f(  x \cdot g). \]
 The map $\theta$ is called the formation of theta series.  Hence, we may define a composite map
\[  \begin{CD}
\alpha_{\Sigma}^{\rm Aut}:   C^{\infty}_c(X_{\A}, \chi) @>\theta>>  C_{mod}^{\infty}([G])  @>{\rm pr}_{\Sigma}>>  \Sigma.  \end{CD} \]
Concretely, we have:
\[  \alpha_{\Sigma}^{\rm Aut}(f) = \sum_{\phi \in {\rm ONB}(\Sigma)}  \langle \theta(f) ,  \phi \rangle_G \cdot \phi. \]  
\vskip 5pt

On the other hand, we have the $G(\A)$-equivariant map
\[   \beta^{\rm Aut}_{\Sigma}: \Sigma \longrightarrow C^{\infty}(X_{\A}) \]
defined by
\[  \beta_{\Sigma}^{\rm Aut}(\phi)(g)   = P_{H, \chi} (g \cdot \phi). \]
One has the following adjunction formula,  which is the global analog of (\ref{E:alpha}):
\vskip 5pt

\begin{lem}
For $f \in C^{\infty}_c(X_{\A}, \chi)$ and   $\phi \in \Sigma$, one has
\[  \langle \alpha_{\Sigma}^{\rm Aut}(f), \phi \rangle_G  =  \langle f,  \beta_{\Sigma}^{\rm Aut}(\phi) \rangle_X  \]
\end{lem}
\begin{proof}
We have:
\begin{alignat}{2}
 \langle \alpha_{\Sigma}^{\rm Aut}(f), \phi \rangle_G   
&=  \int_{[G]} \theta(f)(g)  \cdot \overline{\phi(g)} \, dg   &&=  \int_{[G]}  \sum_{\gamma \in H(k) \backslash G(k)}  f(\gamma g) \cdot  \overline{\phi(g)} \, dg  \notag \\
&= \int_{ H(k) \backslash G(\A)}  f(g) \cdot  \overline{\phi(g)} \, dg   &&= \int_{X_{\A}}    \int_{[H]} f(hg)  \cdot  \overline{\phi(hg)} \, dh  \, \frac{dg}{dh} \notag \\
&= \int_{X_{\A}} f(x)  \cdot \overline{P_{H,\chi}(\phi)(x)} \, dx   &&= \langle f,  \beta^{\rm Aut}_{\Sigma}(\phi) \rangle_X,     \notag 
\end{alignat}
as desired.
\end{proof}

\subsection{\bf Global Relative Characters} \label{SS:GRC}
We may also introduce the global analog of the inner product $J_{\sigma}$:
\[  J^{\rm Aut}_{\Sigma}(\phi_1, \phi_2)   :=  \langle \alpha^{\rm Aut}_{\Sigma}(\phi_1) , \alpha^{\rm Aut}_{\Sigma} (\phi_2)\rangle_{\Sigma} = \langle \beta^{\rm Aut}_{\Sigma} \alpha^{\rm Aut}_{\Sigma}(\phi_1),  \phi_2 \rangle_X.  \]
Then
\begin{align}
  J^{\rm Aut}_{\Sigma}(\phi_1, \phi_2)    &= \sum_{f \in {\rm ONB(\Sigma)}}  \langle \alpha^{\rm Aut}_{\Sigma}(\phi_1) , f \rangle_{\Sigma} \cdot  \langle f, \alpha^{\rm Aut}_{\Sigma} (\phi_2)\rangle_{\Sigma} \notag \\
  &= \sum_{f \in {\rm ONB(\Sigma)}}  \langle \phi_1,  \beta^{\rm Aut}_{\Sigma}(f) \rangle_{X} \cdot  \langle  \beta^{\rm Aut}_{\Sigma}(f) ,  \phi_2\rangle_{X}. \notag
  \end{align} 
By analog with the local case, we may introduce the global relative character $\mathcal{B}^{\rm  Aut}_{\Sigma}$ as an equivariant distribution on $C^{\infty}_c(X_{\A},\chi)$, defined by
\[  \mathcal{B}^{\rm Aut}_{\Sigma}(f)  =  \beta^{\rm Aut}_{\Sigma} \left( \alpha^{\rm Aut}_{\Sigma}(f) \right)(1)  
=  \sum_{\phi \in {\rm ONB}(\Sigma)}   \langle f, \beta^{\rm Aut}_{\Sigma}(\phi) \rangle_X \cdot   P_{H, \chi}(\phi)\]
for $f \in C^{\infty}_c(X_{\A}, \chi)$. When pulled back to give a distribution on $G(\A)$, one has
\[    
 \mathcal{B}^{\rm Aut}_{\Sigma}(f)    =  \sum_{\phi \in {\rm ONB}(\Sigma)}  \overline{P_{H,\chi}( \Sigma(\overline{f})  \phi) } \cdot   P_{H,\chi}(\phi), \]
 for $f \in C^{\infty}_c(G(\A))$.
\vskip 5pt

\subsection{\bf Quadratic spaces and hyperboloids} \label{SS:hyper}
Suppose now that $(V,q)$ is a quadratic space over $k$. Then as an additive group scheme over $k$, $V \cong \mathbb{G}_a^k$ and so $V(\A)$ has its canonical Tamagawa measure.
We would like to compare this Tamagawa measure with the measures we considered in the local setting.
\vskip 5pt 
If $\langle- , -\rangle$ is the symmetric bilinear form associated to $q$ and $\psi = \prod_v \psi_v$
  is  our fixed additive character of $F \backslash \A$, then the pair $(\langle-, -\rangle, \psi_v)$    determines a Haar measure $d_{\psi_v} v$ on $V_v = V \otimes_k k_v$ (the self-dual measure with respect to the pairing $\psi(\langle -, -\rangle)$). This is the measure on $V_v$ that we have been using in the local setting. 
 If $L \subset V$ is any $\mathcal{O}_k$-lattice, which endows $V$ with an $\mathcal{O}_k$-structure, 
  then  for almost all places $v$, the volume of $L_v = L \otimes_{\mathcal{O}_k} \mathcal{O}_{k_v}$ with respect to $d_{\psi_v}v$ is $1$. 
  We may thus consider  the product measure 
  \[  dv_{\A} = \prod_v d_{\psi_v} v \quad \text{ on $V_{\A}$.} \]
   As the notation suggests, it is independent of the choice of $\psi$. Moreover, $dv_{\A}$ is equal to the Tamagawa measure on $V_{\A}$.
   \vskip 5pt

 Suppose that the quadratic space $(V,q)$ contains a vector $v_1 \in V_k$ with $q(v_1) =1$. By changing $L$ if necessary, we may assume that $v_1$ lies in the lattice $L$.  Let 
\[  X_1 = \{ x \in V: q(x) =1 \} \]
be a hyperboloid.  Then the map $h \mapsto h^{-1} \cdot v_1$ gives an isomorprhism
\[   \OO(v_1^{\perp}) \backslash \OO(V) \cong X_1 \]
of $\OO(V)$-varieties over $k$. 
Moreover, in this case, one has (by Witt's theorem):
\[   \OO(v_1^{\perp})(\A)  \backslash \OO(V)(\A) \cong X_1(\A). \]
Recall that we have equipped both sides with their Tamagawa measures which are respected by this isomorphism. 
Now we would like to relate the Tamagawa measure on $X_1(\A)$ with the measures we have been using in the local case. 
\vskip 5pt

As noted, the additive character $\psi = \prod_v \psi_v$ gives us decompositions of Tamagawa measures 
\[  dv_{\A}  = \prod_v d_{\psi_v} v     \quad  \text{and} \quad dx_{\A} = \prod_v d_{\psi_v} x   \]
on $V_{\A}$ and $\A$ respectively.  Using the submersive  map 
\[ q: V \setminus X_0  \longrightarrow F \setminus \{0 \}, \] 
the local measures $d_{\psi_v}v $ and $d_{\psi_v}x$ determine an $\OO(V_v)$-invariant measure $|\omega_{1,v}|$ on $X_1(k_v)$: this is the measure on $X_1(k_v)$ that we have been using in the local setting. We observe that the product measure 
\[ |\omega_{1,\A}| := \prod_v |\omega_{1,v}| \]
 is equal to the Tamagawa measure of $X_1(\A) = \prod'_v X_1(k_v)$ where the restricted direct product is taken with respect to the family $\{ X_1(\mathcal{O}_{k_v}) = X_1(k_v) \cap L_v \}$ for almost all $v$. 

\vskip 10pt

\subsection{\bf Global Weil representation} \label{SS:globalweil}
 We now consider the dual pair $\SL_2 \times \OO(V)$ and recall its global Weil representation.
We have fixed  an $\mathcal{O}_k$-lattice $L \subset V$.   For almost all $v$, $L_v = L \otimes_{\mathcal{O}_k} \mathcal{O}_{k_v}$ is a self dual lattice of volume $1$ with respect to $d_{\psi_v}v$.  Let
 \[  K'_v = \text{stabilizer of $L$ in $\OO(V_v)$} \quad \text{and}\quad \Phi_{0,v} = 1_{L_v} \in S(V_v).\]
 For each $v$, we have the (smooth) Weil representation $\Omega_{\psi_v}$ of $\SL_2(k_v) \times \OO(V_v)$ realized on the space $S(V_v)$ of Schwarz-Bruhat functions on $V_v = V \otimes_k k_v$. For almost all $v$, $\Phi_{0,v}$ is a unit vector  which is fixed by $K_v \times K'_v = \SL_2(\mathcal{O}_{k_v}) \times K'_v$. 
The restricted tensor product 
\[  \Omega_{\psi} = \otimes_v' \Omega_{\psi_v} \]
 with respect to the family of vectors $\Phi_{0,v}$ is the  Weil representation of the adelic dual pair $\SL_2(\A) \times \OO(V)(\A)$. It is realised on the space 
 \[ S(V_{\A}) = S(V_{\infty} )\otimes \left(\otimes_{v< \infty}' S(V_v)\right) \]
  of  Schwarz-Bruhat functions on $V_{\A}$ (where $V_{\infty} = V \otimes_{\Q} \R$).  
\vskip 5pt

The Weil representation $\Omega_{\psi}$ has  a canonical automorphic realization 
\[  \theta: \Omega_{\psi}  = S(V_{\A})  \longrightarrow C_{mod}^{\infty}([\SL_2 \times \OO(V)]) \]
defined by
\[ \theta(\Phi)(g,h) = \sum_{v \in V_k}  (\Omega_{\psi}(g,h) \Phi) (v). \]

 \vskip 5pt
\subsection{\bf Global theta lifting}  \label{SS:globaltheta}
 For an irreducible  cuspidal representation $\Sigma \subset \mathcal{A}_{cusp}(\SL_2)$ of $\SL_2$, we may consider its global theta lift to $\OO(V)$. More precisely, given $\Phi  \in \Omega_{\psi}$ and $f \in \Sigma$, one defines the $\SL_2(\A)$-invariant and $\OO(V)(\A)$-equivariant map
 \[  A_{\Sigma}^{\rm Aut}:   \Omega_{\psi} \otimes \overline{\Sigma} \longrightarrow \mathcal{A}(\OO(V)) \]
 by
 \[  A_{\Sigma}^{\rm Aut}(\Phi, f)(h)  = \int_{[\SL_2]}  \theta(\Phi)(gh) \cdot \overline{f(g)} \, dg.  \]
The image of $A_{\Sigma}^{\rm Aut}$  is the global theta lift of $\Sigma$, which we denote by 
\[ \Pi = \Theta^{\rm Aut}(\Sigma) \subset \mathcal{A}(\OO(V)). \]
 If $\Pi$ is cuspidal and nonzero, then  it follows by the Howe duality conjecture that $\Pi$ is an irreducible cuspidal representation. 

\vskip 5pt

Conversely, assume that $\Pi \subset \mathcal{A}_{cusp}(\OO(V))$ is irreducible cuspidal. Then  we may consider the global theta lift of $\Pi$  to $\SL_2$.  More precisely, given $\Phi \in \Omega_{\psi}$ and $\phi \in \Pi$, one defines the $\OO(V)(\A)$-invariant and $\SL_2(\A)$-equivariant map
\[  B_{\Pi}^{\rm Aut}:  \Omega_{\psi} \otimes \overline{\Pi} \longrightarrow\Sigma \subset  \mathcal{A}(\SL_2) \]
by
\[  B_{\Pi}^{\rm Aut}(\Phi, \phi) (g) = \int_{[\OO(V)]}\theta(\Phi)(gh) \cdot \overline{\phi(h)} \, dh. \]
By computing constant term, one can show that  the image of $B_{\Pi}^{\rm Aut}$ is necessarily cuspidal.  
\vskip 5pt

\subsection{\bf The maps $A_{\Sigma}^{\rm Aut}$ and $B_{\Theta(\Sigma)}^{\rm Aut}$}
We continue with the setting of the previous subsection. 
If $\Pi = \Theta^{\rm Aut}(\Sigma)$ is nonzero cuspidal, then the image of $B^{\rm Aut}_{\Theta(\Sigma)}$ is $\Sigma$ (because the cuspidal spectrum of $\SL_2$ is multiplicity-free \cite{Rama}). In this case, 
 the maps $A_{\Sigma}^{\rm Aut}$ and $B_{\Theta(\Sigma)}^{\rm Aut}$ are global analogs of the maps $A_{\sigma}$ and $B_{\Theta(\sigma)}$ introduced in \S \ref{SS:AB}. By an exchange of the order of integration, we have the following global analog of (\ref{E:AB}):
\begin{equation} \label{E:gAB}
  \langle  A^{\rm Aut}_{\Sigma}(\Phi, f), \phi \rangle_{\Theta(\Sigma)}  =  \langle   B_{\Theta(\Sigma)}^{\rm Aut}(\Phi, \phi), f \rangle_{\Sigma}  \end{equation}
 for $\Phi \in \Omega_{\psi}$, $f \in \Sigma$ and $\phi \in \Theta^{\rm Aut}(\Sigma)$. 
\vskip 10pt

\vskip 10pt
\subsection{\bf Global transfer of periods.}
For $\Phi \in \Omega_{\psi}$ and $\phi \in \Pi$, we may compute the $\psi$-Whittaker coefficient of $B^{\rm Aut}_{\Pi}(\Phi, \phi)$. One has the following global analog of Corollary \ref{C:com2}:
\vskip 5pt

\begin{prop}  \label{P:gperiod}
For $\Phi \in \Omega_{\psi}$ and $\phi \in \Pi$, 
\[  P_{N, \psi}(B^{\rm Aut}_{\Pi}(\Phi, \phi) ) =\int_{H(\A) \backslash G(\A)}    \Phi( g^{-1} v_1)  \cdot   P_{\OO(v_1^{\perp})} ( \phi) (g)  \, \frac{dg}{dh}  
= \langle \Phi|_X,  \beta^{\rm Aut}_{\Pi}(\phi) \rangle_{X_{\A}}.  \]
In particular, a cuspidal representation $\Pi$ of $\OO(V)$ has nonzero $\OO(v_1^{\perp})$-period if and only if its global theta lift to $\SL_2$ is globally $\psi$-generic.
\end{prop}
\vskip 5pt

We omit the proof as it is based on a standard computation.
 \vskip 10pt

\subsection{\bf Decompositions of unitary representations} \label{SS:dec}
Suppose now that $\Sigma \subset \mathcal{A}_{cusp}(\SL_2)$ is an irreducible  {\em tempered} $\psi$-generic cuspidal representation of $\SL_2$ and $\Theta^{\rm Aut}(\Sigma)\subset \mathcal{A}_{cusp}(\OO(V))$ is a nonzero (irreducible) cuspidal representation of $\OO(V)$.
In this case, $\Theta^{\rm Aut}(\Sigma)$ is globally $\OO(v_1^{\perp})$-distinguished.  We would like to factor the global $\OO(v_1^{\perp})$-period of $\Theta^{\rm Aut}(\Sigma)$  as a product of the local functionals constructed in the earlier part of the paper.  To carry this out, we need to set things up precisely and systematically.

\vskip 5pt
On the side of $\SL_2$, we begin by fixing a decomposition of the Tamagawa measures
\[   dg = \prod_v dg_v \]
To be concrete, we take an invariant differential $\omega_{\SL_2}$ of top degree on $\SL_2$ over $\Z$ (which is well-defined up to $\pm 1$), which together with the $d_{\psi_v}x$ on $k_v$ gives a measure  
\[ dg_v :=  |\omega_{\SL_2, \psi_v}|_v  \quad \text{on $\SL_2(k_v)$,} \]
for which
\[  \int_{\SL_2(\mathcal{O}_{k_v})} \, dg_v  = \zeta_{k_v}(2)^{-1}. \]
This is the measure we used on $\SL_2(k_v)$ in \S \ref{S:L} when we computed the local L-factor $L_X^{\#}$. 
The product of these measures is then equal to the Tamagawa measure $dg$. 
\vskip 5pt
We also fix an isomorphism
\begin{equation}  \label{E:isom}
 \Sigma \cong \otimes'_v \Sigma_v  \end{equation}
and a decomposition of the Petersson inner product
\begin{equation} \label{E:decomp} 
  \langle -, - \rangle_{\Sigma} = \prod_v  \langle -, - \rangle_{\Sigma_v}. \end{equation}
This equips $\Sigma_v$ with a unitary structure $\langle-, - \rangle_{\Sigma_v}$. 
Alternatively, we could work with the tautological measurable field of unitary representations provided by the local Whittaker-Plancherel theorem (so that each $\Sigma_v$ comes equipped with a unitary structure already), in which case one would require the isomorphism in (\ref{E:isom}) to be an isometry, so that one has (\ref{E:decomp}) as a consequence. 
These are two viewpoints with no difference in content.  The restricted tensor product in (\ref{E:isom}) is with respect to a family of unit vectors $v_{0,v}$ which is fixed by $K_v = \SL_2(\mathcal{O}_{k_v})$.  
\vskip 5pt

For the Weil representation, we have already seen in \S \ref{SS:globalweil} the decompositions:
\begin{equation} \label{E:decompweil}
    S(V_{\A}) = S(V_{\infty} )\otimes \left(\otimes_{v< \infty}' S(V_v)\right) \end{equation}
   and
   \[ dv_{\A} = \prod_v d_{\psi_v}v \quad \text{ of Haar measures on $V_{\A}$,} \]
 inducing compatible unitary structures on the global and local Weil representations. In particular,  
  the restricted direct product in (\ref{E:decompweil}) is with respect to the family of unit vectors $\Phi_{0,v} = 1_{L_v}$ for almost all $v$ and 
 \[  \langle-,- \rangle_{\Omega_{\psi}} = \prod_v \langle  -, - \rangle_{\Omega_{\psi_v}}. \]
\vskip 5pt

Once these decompositions are fixed as above, we see that for each place $v$,  we have the local big theta lift
\[  \Theta_{\psi_v}(\Sigma_v) := (\Omega_{\psi_v} \otimes \Sigma_v^{\vee})_{\SL_2(k_v)}  \]
and its unique irreducible quotient 
\[ \theta_{\psi_v}(\Sigma_v)  \in {\rm Irr}(\OO(V_v)) \]
Moreover, $\theta_{\psi_v}(\Sigma_v)$  inherits an inner product
$\langle-, - \rangle_{\theta(\Sigma_v)}$ defined via the local doubling zeta integral. We let $w_{0,v} \in \theta_{\psi_v}(\Sigma_v)$ be a unit vector which is fixed by $K'_v$ for almost all $v$. Then we may form the abstract global theta lift
\[  \Theta^{\A}(\Sigma) = \otimes'_v \theta_{\psi_v}(\Sigma_v) \] 
where the restricted tensor product is with respect to the family of unit vectors $w_{0,v}$. The abstract global theta lift  $\Theta^{\A}(\Sigma)$ inherits a unitary structure 
\[ \langle-, -\rangle_{\Theta^{\A}(\Sigma)} = \prod_v \langle-, -\rangle_{\theta(\Sigma_v)} \]
from that of its local components.  Hence we may fix an  isometric isomorphism
\[  \Theta^{\rm Aut}(\Sigma) \cong \Theta^{\A}(\Sigma)  \]
so that
\[  \langle -, - \rangle_{\Theta^{\rm Aut}(\Sigma)} =  \prod_v \langle-, -\rangle_{\theta(\Sigma_v)} \]
where the inner product on the left is that defined by the Petersson inner product.
\vskip 10pt

\subsection{\bf Adelic periods}
Having fixed the various decompositions in the previous subsection, we can now introduce the adelic versions of various period maps. 
We have introduced various global or automorphic quantities associated to $\Sigma$ and $\Theta^{\rm Aut}(\Sigma)$, namely
\vskip 5pt

\begin{itemize}
\item the maps $\alpha^{\rm Aut}_{\Sigma}$, $\beta^{\rm Aut}_{\Sigma}$, $P_{N, \psi, \Sigma}$ and $J^{\rm Aut}_{\Sigma}$ related to the global Whittaker period with respect to $(N, \psi)$;
\vskip 5pt

\item the maps $\alpha^{\rm Aut}_{\Theta(\Sigma)}$, $\beta^{\rm Aut}_{\Theta(\Sigma)}$, $P_{\OO(v_1^{\perp}), \Theta(\Sigma)}$ and $J^{\rm Aut}_{\Theta(\Sigma)}$ related to the
$\OO(v_1^{\perp})$-period; 
\vskip 5pt

\item the maps $A^{\rm Aut}_{\Sigma}$, $B^{\rm Aut}_{\Theta(\Sigma)}$ related to global theta lifting.
\end{itemize}
\vskip 5pt
 \vskip 5pt

All the above global objects have local counterparts, relative to the decompositions fixed in the previous subsection. Namely, 
for each place $v$ of $k$,  we have:
\vskip 5pt

\begin{itemize}
\item the maps $\alpha_{\Sigma_v}$, $\beta_{\sigma_v}$,  $\ell_{\Sigma_v}$ and $J_{\Sigma_v}$ given by the Whittaker-Plancherel theorem;
\vskip 5pt

\item the maps $\alpha_{\Theta(\Sigma_v)}$, $\beta_{\Theta(\sigma_v)}$,  $\ell_{\Theta(\Sigma_v)}$ and $J_{\Theta(\Sigma_v)}$ given by the spectral decomposition of $L^2(X_v, |\omega_v|)$;

\vskip 5pt

\item the maps $A_{\Sigma_v}$ and $B_{\Theta(\Sigma_v)}$ given by the spectral decomposition of the Weil representation $\Omega_{\psi_v}$. 
\end{itemize}
\vskip 5pt

We may take the Euler product of the above local quantities. As an example, we set:
\[   \beta_{\Sigma}^{\A}  := \prod_v^{\ast}  \beta_{\Sigma_v}  \quad \text{and} \quad \ell_{\Sigma}^{\A} := \prod_v^{\ast} \ell_{\Sigma_v}. \]
Here, the Euler product has to be understood  as a regularized product, via meromorphic continuation if necessary, as discussed in the introduction. Let us illustrate this in three instances, assuming that $\dim (V). =2n$ and $\disc (V) = 1 \in k^{\times}/ k^{\times 2}$ (for simplicity):
\vskip 5pt
\begin{itemize}
\item  Suppose we want to define $\ell_{\Sigma}^{\A}$. Given a decomposable vector $u = \otimes_v u_v \in \Sigma$, with $u_v = v_{0,v}$ for almost all $v$, we need to examine the convergence of the product
\[  \prod_v \ell_{\Sigma_v}( u_v). \]
This is determined by the value $\ell_{\Sigma_v}(v_{0,v})$ for almost all $v$. In Lemma \ref{L:unram}(i), we have noted that 
\[   |\ell_{\Sigma_v}(v_{0,v})|^2 = \frac{\zeta_{k_v}(2)}{L(1, \Sigma_v, Ad)}. \] 
Thus the Euler product may not converge because of the denominator.  However, we may normalize $\ell_{\Sigma_v}$ by setting
\[   \ell_{\Sigma_v}^{\flat}   =   \lambda_v^{-1}  \cdot \ell_{\Sigma_v}   \quad \text{for each place $v$} \]
with
\[ |\lambda_v|^2 = \frac{\zeta_{k_v}(2)}{L(1, \Sigma_v, Ad)} \]
and 
\[ \ell_{\Sigma_v}^{\flat}(v_{0,v}) = 1 \quad \text{for almost all $v$.} \]
Then the product $\prod_v \ell_{\Sigma_v}^{\flat}(u_v)$ certainly converges and we set
\[  \ell_{\Sigma}^{\A} := \prod_v^{\ast} \ell_{\Sigma_v} := \left( \frac{|\zeta_{k}(2)|}{|L(1, \Sigma, Ad)|}\right)^{1/2} \cdot \prod_v \ell_{\Sigma_v}^{\flat} \]
where the global L-value is defined by meromorphic continuation of the global L-function. 
\vskip 5pt

\item Suppose we want to define 
\[ A_{\Sigma}^{\A} : \Omega_{\psi} \otimes \overline{\Sigma} \longrightarrow \Theta^{\A}(\Sigma). \]
Hence we need to consider the infinite product
\[  \prod_{v \notin S} A_{\Sigma_v}(\Phi_{0,v}, v_{0,v})  \]
for some finite set $S$ of places of $k$. By Lemma \ref{L:unram}(i), we see that for almost all $v$, 
\[    A_{\Sigma_v}(\Phi_{0,v}, v_{0,v})  = \frac{L(n-1, \Sigma_v,  std)}{\zeta_{k_v}(2) \cdot \zeta_{k_v}(n) \cdot \zeta_{k_v}(2n-2) } \cdot w_{0,v} \in \theta_{\psi_v}(\Sigma_v)  \]
where $\dim V = 2n$ and $w_{0,v}$ is a unit vector.
Since $\Sigma$ is tempered, the relevant Euler product actually converges absolutely when $n > 2$, in which case we can simply define
\[  A^{\A}_{\Sigma} = \prod_v A_{\Sigma_v} . \]
In general, we would set
\[  A_{\Sigma_v}^{\flat} = \lambda_v^{-1} \cdot A_{\Sigma_v}  \]
with 
\[ | \lambda_v|^2 =  \frac{|L(n-1, \Sigma_v,  std)|}{|\zeta_{k_v}(2)| \cdot |\zeta_{k_v}(n)| \cdot |\zeta_{k_v}(2n-2)|} \quad \text{ for all $v$,} \]
and 
\[ A_{\Sigma_v}^{\flat}(\Phi_{0,v}, v_{0,v}) = w_{0,v} \quad \text{ for almost all $v$.} \]
Then we set
\[  A^{\A}_{\Sigma} = \left( \frac{|L(n-1, \Sigma,  std)|}{|\zeta_k(2)| \cdot |\zeta_k(n)|\cdot |\zeta_k(2n-2)| }   \right)^{1/2} \cdot \prod_v A_{\Sigma_v}^{\flat}. \]
\vskip 5pt

\item Suppose we want to define $\ell_{\Theta(\Sigma)}^{\A} : \Theta^{\A}(\Sigma) \longrightarrow \C$. Then we need to consider the Euler product
\[  \prod_{v \notin S} \ell_{\theta(\Sigma_v)} ( w_{0,v})     \quad \text{for some finite set $S$.} \]
In Proposition \ref{P:Lfactor}, we have seen that
\[  |\ell_{\theta(\Sigma_v)}(w_{0,v}) |^2 =  \frac{L(n-1, \Sigma_v, std)}{ L(1, \sigma, Ad)}  \cdot  \frac{\zeta_{k_v}(n)}{\zeta_{k_v}(2n-2) }\]
for almost all $v$ (with $\dim V = 2n$). If $n =2$, the right hand side is equal to $1$, so the relevant Euler product converges. For $n > 2$, we set
\[  \ell_{\theta(\Sigma_v)}^{\flat} = \lambda_v^{-1} \cdot \ell_{\theta(\Sigma_v)} \]
with 
\[  |\lambda_v|^2 =  \frac{L(n-1, \Sigma_v, std)}{ L(1, \Sigma_v, Ad)}  \cdot  \frac{\zeta_{k_v}(n)}{\zeta_{k_v}(2n-2)}  \quad \text{for all $v$} \]
 and
 \[  \ell_{\theta(\Sigma_v)}^{\flat} (w_{0,v}) =1 \quad \text{ for almost all $v$.} \]
 Then we set
 \[  \ell^{\A}_{\Theta^{\A}(\Sigma)}  =\left(  \frac{|L(n-1, \Sigma, std)|}{ |L(1, \Sigma, Ad)|}  \cdot  \frac{|\zeta_k(n)|}{|\zeta_k(2n-2)| } \right)^{1/2}   \cdot \prod_v \ell^{\flat}_{\theta(\Sigma_v)}. \]
 \end{itemize}

After these three illustrative examples, we leave the precise definition of other adelic period maps to the reader. 

\vskip 10pt

\subsection{\bf Comparison of automorphic and adelic periods}
We can now compare the various adelic period maps with their automoprhic counterparts, using the decompositions
\[ \Omega_{\psi} \cong \otimes_v' \Omega_{\psi} , \quad \Sigma \cong \otimes'_v \Sigma_v \quad \text{and} \quad \Theta^{\rm Aut}(\Sigma) \cong \Theta^{\A}(\Sigma):= \otimes'_v \theta_{\psi_v}(\Sigma_v) \]
 fixed in \S \ref{SS:dec}.
\vskip 5pt

Since both
$\ell_{\Sigma}^{\A}$ and $P_{\Sigma, N, \psi}$ are nonzero elements of  the 1-dimensional space $\Hom_{N(\A)}(\Sigma, \psi)$, 
 there is a constant $c(\Sigma) \in \C^{\times}$ such that
 \[   P_{\Sigma, N, \psi} =  c(\Sigma) \cdot \ell_{\Sigma}^{\A}, \]
 so that  
  \[   \beta^{\rm Aut}_{\Sigma} = c(\Sigma) \cdot  \beta_{\Sigma}^{\A}  \]
 Likewise, we have $c(\Theta(\Sigma)) \in \C^{\times}$ such that 
 \[  P_{\Theta(\Sigma), \OO(v_1^{\perp})} = c(\Theta(\Sigma)) \cdot \ell_{\Theta(\Sigma)}^{\A}   \]
 so that
 \[   \beta^{\rm Aut}_{\Theta(\Sigma)} = c(\Theta(\Sigma)) \cdot  \beta_{\Theta(\Sigma)}^{\A}, \]
   Similarly, we have $a(\Sigma)$ and $b(\Sigma)  \in \C^{\times}$ such that
 \[  A^{\rm Aut}_{\Sigma}  = a(\Sigma) \cdot A^{\A}_{\Sigma} \quad \text{and} \quad B^{\rm Aut}_{\Theta(\Sigma)} = b(\Sigma) \cdot B^{\A}_{\Theta(\Sigma)} \]
    
 \vskip 10pt

\subsection{\bf Global result}
 The main global problem is to determine the constant $|c(\Theta(\Sigma))|^2$. We shall resolve this by relating $c(\Theta(\Sigma))$ to the other constants $c(\Sigma)$, $a(\Sigma)$, and  $b(\Sigma)$.  
\vskip 5pt

\begin{prop}  \label{P:cb}
We have:
\[  \overline{c(\Theta(\Sigma))}  =  c(\Sigma) \cdot b(\Sigma).  \]
\end{prop}

\vskip 5pt

\begin{proof}
This follows by combining the global Proposition \ref{P:gperiod} and the local Corollary \ref{C:com2}.
\end{proof}
\vskip 5pt

It remains then to compute $c(\Sigma)$ and $b(\Sigma)$.   
\vskip 5pt

\begin{prop} \label{P:ab}
 We have:
\[  a(\Sigma)  = b(\Sigma)  \quad \text{and} \quad  |a(\Sigma)|  =1. \] 
Moreover, 
\[  c(\Sigma)  = \begin{cases}
1, \text{  if $\Sigma$ is non-endoscopic;} \\
1/2,  \text{  if $\Sigma$ is an endoscopic lift from $\OO_2 \times \OO_1$;} \\ 
1/4,  \text{  if $\Sigma$ is an endoscopic lift from $\OO_1 \times \OO_1 \times \OO_1$.}
\end{cases} \]
\end{prop}
\vskip 5pt

\begin{proof}
The equality of $a(\Sigma)$ and $b(\Sigma)$ follows by the global equation (\ref{E:gAB}) and the local equation (\ref{E:AB}). 
On the other hand, the Rallis inner product formula \cite{GQT} gives:
\[  \langle A^{\rm Aut}_{\Sigma}( \Phi, f) ,   A^{\rm Aut}_{\Sigma}( \Phi, f)  \rangle_{\Theta(\Sigma)}   =  Z_{\Sigma}(\Phi, \Phi, f,f) 
 = \langle A^{\A}_{\Sigma}(\Phi,f),  A^{\A}_{\Sigma}(\Phi,f) \rangle \]
where $Z_{\Sigma}(-)$  is the global doubling zeta integral (evaluated at the point $s = (\dim V -3)/2$, where it is holomorphic).     Combining this with the local equation (\ref{E:zeta}), we deduce that $|a(\Sigma)|  =1$.
\vskip 5pt

Finally, the value of $c(\Sigma)$ (for the group $\SL_2$)  was determined in \cite[Cor. 4.3 and \S 6.1]{LM}.
\end{proof} 
\vskip 5pt

As a consequence, we have:

\begin{thm}  \label{T:global}
Let $\Sigma$ be a globally $\psi$-generic cuspidal representation of $\SL_2$ such that
$\Pi = \Theta(\Sigma) \subset \mathcal{A}_{cusp}(\OO(V))$. Then  $|c(\Theta(\Sigma))|  = | c(\Sigma)|$, so that
\[  |P_{\Pi, \OO(v_1^{\perp})} (\phi)|^2  =  | c(\Sigma)|^2 \cdot   |\ell_{\Pi}^{\A}(\phi)|^2\]
for all $\phi \in \Pi$.  
\end{thm}

 \vskip 10pt

 \subsection{\bf Avoiding Rallis inner product} 
 In proving Theorem \ref{T:global}, we have pulled  the Rallis inner product formula out of the hat to deduce that $|a(\Sigma)|  =1$ in Proposition \ref{P:ab}. In fact, it is possible 
to avoid the Rallis inner product formula, as we briefly sketch in this subsection.  
\vskip 5pt

Just as Proposition \ref{P:gperiod} is the global analog of the local Corollary \ref{C:com2}, one can establish a global analog of Corollary \ref{C:othercom2}, namely:
\vskip 5pt

  \begin{prop} \label{P:gperiod2}
  For $\Phi \in \Omega_{\psi}$ and $f \in \Sigma$, one has
  \[  P_{\OO(v_1^{\perp})}(A^{\rm Aut}_{\Sigma}(\Phi,f))   = \langle p(\Phi),     \beta^{\rm Aut}_{\Sigma}(f) \rangle_{N \backslash \SL_2}. \]
 \end{prop}
  \vskip 5pt

\begin{proof}
The proof  relies on the see-saw diagram
\[
 \xymatrix{
  {\Mp_2 \times \Mp_2}  \ar@{-}[dr] \ar@{-}[d] & { \OO(V)} 
     \ar@{-}[d] \\
 { \SL_2}  \ar@{-}[ur] &  {\OO(v_1) \times \OO(v_1^{\perp})}}
\] 
which gives rise to a global see-saw identity. More precisely, if we take 
\[  \Phi  =    \Phi_1 \otimes\Phi'  \in \mathcal{S}( \A v_1) \otimes \mathcal{S}((v_1^{\perp})_{\A}), \]
then the see-saw identity reads:
 \[    P_{\Pi, \OO(v_1^{\perp})} ( A^{\rm Aut}_{\Sigma}(\Phi, f) )  =   \int_{[\SL_2]}  \overline{f(g)} \cdot   \theta(\Phi_1) (g)  \cdot I(\Phi')(g)  \, dg   \]
 where
 \vskip 5pt
 \begin{itemize}
 \item[-] $\theta(\Phi_1)$ is the theta function associated to $\Phi_1 \in \mathcal{S}(\A v_1)$ which affords the Weil representation (associated to $\psi$) for $\Mp_2 \times \OO(v_1)$; 
\item[-] $I(\Phi')$ is the theta integral
 \[  I(\Phi')(g)  :=  \int_{[\OO(v_1^{\perp})]}   \theta(\Phi')(gh)  \, dh  \]
 which belongs to the global theta lift of the trivial representation of $\OO(v_1^{\perp})$ to $\Mp_2$.  
 \end{itemize}
 The theta integral converges absolutely when $\dim V > 4$ (it is in the so-called Weil's convergence range) or when $\OO(v_1^{\perp})$ is anisotropic. 
When $\dim V  = 4$ and $\OO(v_1^{\perp})$ is split, one needs to regularise the theta integral following Kudla-Rallis  (see \cite[\S 3]{GQT}).  Since our intention here is to indicate an alternative approach to a result which we have shown, we will ignore this analytic complication in the following exposition.
 \vskip 5pt
 
 In \S \ref{SS:adjoint}, we have seen that the map $\phi' \mapsto  \phi_{\Phi'}$, where 
 \[  \phi_{\Phi'}(g)  = (g \cdot \Phi')(0),  \]
gives an isomorphism
\[  \mathcal{S}((v_1^{\perp})_{\A})_{\OO(v_1^{\perp})}  \longrightarrow  I_{\psi}( \chi_{{\rm disc}(v_1^{\perp})}, (\dim V-3)/2)   \]
 of the $\OO(v_1^{\perp})$-coinvariant of the Weil representation of $\OO(v_1^{\perp}) \times \Mp_2$  to a principal series representation of $\Mp_2$.   
  Now  the Siegel-Weil formula shows that 
 \[  I(\Phi')  = E( \phi_{\Phi'})  \]
 where $E(\phi_{\Phi'})$ is the Eisenstein series associated to $\phi_{\Phi'}$.  Again, when $\dim V > 4$, the sum defining the Eisenstein series is convergent, but when $\dim V = 4$, it is defined by meromorphic continuation. Further, if $\OO(v_1^{\perp})$ is also split, then the Eisenstein series does have a pole at the point of interest, and we need to invoke the second term identity of the Siegel-Weil formula \cite{GQT}. As mentioned before, we omit these extra (though interesting) details in this proof. 
 
 \vskip 5pt
 
 Hence, we have the following identity: 
 \[    P_{\Pi, \OO(v_1^{\perp})} ( A^{\rm Aut}_{\Sigma}(\Phi, f) )  =    \int_{[\SL_2]}    \overline{f(g)} \cdot   \theta(\Phi_1) (g)  \cdot  E( \phi_{\Phi'}) \, dg. \]
 Now the right hand side is the value at $s = (\dim V-3)/2$ of the global zeta integral
 \[  Z( f, \Phi_1, \phi_s)  =   \int_{[\SL_2]}    \overline{f(g)} \cdot   \theta(\Phi_1) (g)  \cdot  E(\phi_s) \, dg \]
 for $\phi_s \in I_{\psi}(\chi_{{\rm disc}(v_1^{\perp})}, s)$. 
 This is the global analog of the local zeta integrals we discussed in \S \ref{SS:adjoint} and represents the (twisted) adjoint L-function of $\Sigma$.  
 The unfolding of this global zeta integral, for ${\rm Re}(s)$ sufficiently large, gives:
 \[  Z(s, f , \Phi_1, \phi)  =   \int_{N(\A) \backslash \SL_2(\A)}  \overline{P_{N, \psi}(g \cdot f)}   \cdot  (g \cdot \Phi_1)(v_1)  \cdot \phi_s(g)   \, dg. \]
 Specializing to $s = (\dim V - 3)/2$ gives the Proposition.
  \end{proof}
  \vskip 5pt

  By combining Proposition \ref{P:gperiod2} and Corollary \ref{C:othercom2}, we deduce:
  
  \begin{cor}  \label{C:ca} 
  One has:
  \[   \overline{c(\Sigma)}    = c(\Theta(\Sigma)) \cdot a(\Sigma). \]
   \end{cor}
  \vskip 5pt
  
  Combining Corollary \ref{C:ca} with Propositions \ref{P:cb} and \ref{P:ab}, we see that
  \[   \overline{c(\Sigma)}  =  c(\Theta(\Sigma)) \cdot a(\Sigma)= \overline{c(\Sigma)} \cdot \overline{b(\Sigma)} \cdot a(\Sigma) =  \overline{c(\Sigma)} \cdot |a(\Sigma)|^2 \]
  from which we deduce that
  \[  |a(\Sigma)| =1. \]
  \vskip 5pt
  
 There is a good reason for avoiding the use of the Rallis inner product formula in the treatment of the global problem. Indeed, the viewpoint and techniques developed in this paper  should carry over to essentially all the low rank spherical varieties treated in \cite{GG}. Many of these (such as ${\rm Spin}_9 \backslash F_4$, $G_2 \backslash {\rm Spin}_8$ or $F_4 \backslash E_6$ to name a few) would involve the exceptional theta correspondence. Unfortunately, in the setting of the exceptional theta correspondence, an analog of the Rallis inner product formula is not known. The argument in this subsection, however, shows that this lack need not be an obstruction in the exceptional setting. 
 
 \vskip 10pt
  
 \subsection{\bf Global Relative Character Identity}  
 We can also establish the global analog of the relative character identity. One has the diagram
      \[
 \xymatrix{ &  S(V_{\A})
 \ar[dl]_p \ar[dr]^q & \\
 \mathcal{S}(N(\A), \psi \backslash \SL_2(\A))
  & &
 S(X_{\A})  
  }
\]
which is the adelic analog of the diagram (\ref{E:diagram}).
 The space $ \mathcal{S}(N(\A), \psi \backslash \SL_2(\A))$ (which is the image of $p$) is the restricted tensor product of the local spaces $\mathcal{S}(N(k_v),\psi_v \backslash \SL_2(k_v))$ of test functions defined in (\ref{E:test}), where the restricted tensor product is with respect to the family of basic functions $\{ f_{0,v} \}$ given in Definition \ref{D:basic}.  Likewise, the space $ S(X_{\A})$ is the restricted tensor product of $S(X_{k_v})$ (which is $C^{\infty}_c(X_{k_v})$ at finite places) with respect to the family of basic functions $\{ \phi_{0,v}\}$ in Definition \ref{D:basic}. 
 As in the local case, one says that $f \in \mathcal{S}(N(\A), \psi \backslash \SL_2(\A))$ and $\phi \in  S(X_{\A})$ are in correspondence or are transfers of each other if there exists $\Phi \in S(V_{\A})$ such that $f = p(\Phi)$ and $\phi = q(\Phi)$. The fundamental Lemma  \ref{L:fund} ensures that every $f$ has a transfer $\phi$ and vice versa.  
 \vskip 5pt
 
Before formulating the global relative character identity, 
 we need to address an additional  subtle point here. In our general discussion in \S \ref{S:aabb}, we have considered the maps 
 \[ \alpha^{\rm Aut}_{\Sigma}: C^{\infty}_c(N(\A) , \psi \backslash \SL_2(\A)) \longrightarrow \Sigma  \subset \mathcal{A}_{cusp}(\SL_2) \]
 and 
 \[ \alpha^{\rm Aut}_{\Theta^{\rm Aut}(\Sigma)}: C^{\infty}_c(X(\A)) \longrightarrow \Theta^{\rm Aut}(\Sigma) \subset \mathcal{A}_{cusp}(\OO(V)). \] 
 The point here is that their domains consist of smooth compactly supported functions on the adelic points of the relevant spherical varieties. Likewise, in \S \ref{SS:GRC}, 
 the global relative character is given as a distribution on the space of smooth compactly supported functions. Now in the case of the hyperboloid $X$, 
 this is fine  since the basic function $\phi_{0,v}$ belongs to $C^{\infty}_c(X(k_v))$. However, this is not sufficient for the case of the Whittaker variety $(N ,\psi) \backslash \SL_2$ since the basic function $f_{0,v}$  is not compactly supported. In particular, while it is true that
 \[ C^{\infty}_c(N(k_v), \psi_v \backslash \SL_2(k_v)) \subset \mathcal{S}(N(k_v), \psi_v \backslash \SL_2(k_v)) \quad \text{ for each place $v$,} \]
 one has in the adelic setting:
  \[ C^{\infty}_c(N(\A), \psi \backslash \SL_2(\A)) \nsubseteq \mathcal{S}(N(\A), \psi \backslash \SL_2(\A)).\]
 Indeed, these adelic spaces have nothing much to do with each other. 
 \vskip 5pt
 
 To define global relative characters for the Whittaker variety, we thus need to ensure 
 that the map $\alpha^{\rm Aut}_{\Sigma} = {\rm pr}_{\Sigma} \circ \theta$ can be defined on $\mathcal{S}(N(\A), \psi \backslash \SL_2(\A))$.
 \vskip 5pt
 
  The main issue is to ensure that the formation of theta series $\theta$ can be applied to $f  = \otimes'_v f_v\in \mathcal{S}(N(\A), \psi \backslash \SL_2)$. Recall that
  \begin{equation} \label{E:thetaf}
   \theta(f)(g) = \sum_{\gamma \in N(k) \backslash \SL_2(k)} f(\gamma g) \quad \text{for $g \in \SL_2(\A)$.}  \end{equation}
Hence, if $f = p(\Phi)$ with $\Phi \in S(V_{\A})$,  we would like to show the convergence of the sum
 \[   \sum_{\gamma \in B(k) \backslash \SL_2(k)}  \sum_{\lambda  \in k^{\times}}  | f( t(\lambda) \gamma g) | =  \sum_{\gamma \in B(k) \backslash \SL_2(k)}  \sum_{ \lambda \in k^{\times}} | ( \gamma g \cdot \Phi) (\lambda \cdot v_1)|. \]
 Let us denote the inner sum by
 \[   F_{\Phi} (g) :=  \sum_{\lambda \in k^{\times}}  |(g \cdot \Phi) ( \lambda \cdot v_1)|. \] 
which is certainly convergent since $\Phi$ is a Schwatz function on $V_{\A}$. Then we need to show the convergence of
\begin{equation} \label{E:Eis}
      \sum_{\gamma \in B(k) \backslash \SL_2(k)}  F_{\Phi}( \gamma g). \end{equation}
 This looks very much like the sum defining an Eisenstein series on $\SL_2$. Indeed, observe that $F_{\Phi}$ is a function on $T(k) \cdot N(\A) \backslash \SL_2(\A)$  and  for $a \in \A^{\times}$ and $k \in K = \prod_v K_v$, we have:
 \[   F_{\Phi}(t(a)) k) =  |a|_{\A}^{\frac{\dim V}{2}} \cdot  \sum_{\lambda \in k^{\times}} | (k \cdot \Phi)(a \lambda \cdot v_1)|.   \]
 To understand the asymptotic of $F_{\Phi}$ as $|a|_{\A}$ tends to $0$ or $\infty$, we  note:
   
   \begin{lem}
   Let $\phi \in S(\A)$ and define  a function $\Psi$ on $\A^{\times}$ by
  \[  \Psi(a) = \sum_{\lambda \in k^{\times}} \phi (a \lambda). \]
  Then $\Psi$ is rapidly decreasing as $|a|_{\A} \to \infty$. Moreover, as $|a|_{\A} \to 0$, 
  \[  |\Psi(a)| \leq C \cdot |a|_{\A}^{-1}  \quad \text{for some constant $C$.} \]
    \end{lem}
\begin{proof}
The rapid decrease of $\Psi(a)$ as $|a|_{\A} \to \infty$ follows from the fact that $\phi \in S(V_{\A})$. For the asymptotic as $|a|_{\A} \to 0$, we need to apply the Poisson summation formula to the sum defining $\Psi$. Writing $\widehat{\phi}$ for the Fourier transform of $\phi$ (with respect to $d_{\psi}(x)$), we have:
\begin{align}
  \Psi(a) &=   - \phi(0) +  \sum_{\lambda \in k} \phi (a \lambda) \notag \\
 &= - \phi(0) + |a|_{\A}^{-1} \cdot \sum_{\lambda \in k}   \widehat{\phi} (a^{-1} \cdot \lambda) \notag \\
 &= - \phi(0) +  |a|_{\A}^{-1} \cdot \widehat{\phi}(0) +  |a|_{\A}^{-1} \cdot  \sum_{\lambda \in k^{\times}}   \widehat{\phi} (a^{-1} \cdot \lambda) \notag
 \end{align}
 As $|a|^{-1}_{\A} \to \infty$, the last sum tends to $0$ rapidly, so the third  term is bounded. Hence we see that the asymptotic of $\Psi(a)$ as $|a|_{\A} \to 0$ is governed by the second term in the last expression.  
  \end{proof}
 
 \vskip 5pt
 
 Applying the lemma to $\phi(a) = (k \cdot \Phi)(a\cdot v_1)$, we deduce that $F_{\Phi}(t(a) k)$ is rapidly decreasing as $|a|_{\A} \to \infty$ and 
 \[   |F_{\Phi}(t(a) k)| \leq C \cdot |a|_{\A}^{\frac{\dim V}{2} -1} \quad \text{ as $|a|_{\A} \to 0$} \]
for some $C$ which can be taken to be independent of $k \in K$.
  In other words, the sum in (\ref{E:Eis}) is dominated by the sum defining a spherical Eisenstein series associated to the principal series representation $I(\frac{\dim V}{2} -2)$ of $\SL_2$.   Hence, when $\dim V > 6$, the above sum does converge to give a smooth function on $[\SL_2]$ of moderate growth (c.f. \cite[Thm. 11.2]{Bor}),  so that (\ref{E:thetaf}) defines a $\SL_2(\A)$-equivairant map 
  \[   \theta: \mathcal{S}(N(\A), \psi \backslash \SL_2) \longrightarrow C_{mod}^{\infty}([\SL_2]). \]
  One then has the map $\alpha^{\rm Aut}_{\Sigma} = {\rm pr}_{\Sigma} \circ \theta$ such that one still has the adjunction formula
  \[   \langle \alpha^{\rm Aut}_{\Sigma}(f) , \phi \rangle_{\Sigma} = \langle f, \beta^{\rm Aut}_{\Sigma}(\phi) \rangle_{N(\A) \backslash \SL_2(\A)}, \]
  for $f \in \mathcal{S}(N(\A), \psi \backslash \SL_2(\A))$ and $\phi \in \Sigma$.
    
 \vskip 5pt
 
 Presumably, one can show that $\theta$ and $\alpha^{\rm Aut}_{\Sigma}$ are also defined when $3 \leq \dim V \leq 6$ by a more careful analysis, involving the meromorphic continuation of pseudo-Eisenstein series,  but we have not pursued this further. 
 One now has the following global relative character identity.
\vskip 5pt

\begin{thm}  \label{T:global2}
Assume that $\dim V > 6$.
If $f \in \mathcal{S}(N(\A), \psi \backslash \SL_2(\A))$ and $\phi \in  S(X_{\A})$ are  transfer of each other, then for a cuspidal representation $\Sigma$ of $\SL_2$ with cuspidal theta lift $\Theta^{\rm Aut}(\Sigma)$ on $\OO(V)$, one has:
\[  \mathcal{B}^{\rm Aut}_{\Sigma}(f)  = \mathcal{B}^{\rm Aut}_{\Theta(\Sigma)}(\phi). \] 
\end{thm}
\vskip 5pt

\begin{proof}
We have defined $c(\Sigma)$ by
\[   \beta^{\rm Aut}_{\Sigma} = c(\Sigma) \cdot \beta^{\A}_{\Sigma}. \]
By the adjunction formulae for the pairs $( \alpha^{\rm Aut}_{\Sigma}, \beta^{\rm Aut}_{\Sigma})$ and $(\alpha^{\A}_{\Sigma}, , \beta^{\A}_{\Sigma})$, we deduce that
\[  \alpha^{\rm Aut}_{\Sigma} = \overline{c(\Sigma)} \cdot \alpha^{\A}_{\Sigma}. \]
Hence, one has
 \[  \mathcal{B}^{\rm Aut}_{\Sigma}(f)  =( \beta^{\rm Aut}_{\Sigma}  \alpha^{\rm Aut}_{\Sigma} (f))(1)  =
  |c(\Sigma)|^2  \cdot ( \beta^{\A}_{\Sigma}  \alpha^{\A}_{\Sigma} (f))(1) =
  |c(\Sigma)|^2 \cdot \mathcal{B}^{\A}_{\Sigma}(f).  \]
 Likewise, we have
 \[   \mathcal{B}^{\rm Aut}_{\Theta(\Sigma)}(\phi) = |c(\Theta(\Sigma))|^2 \cdot  \mathcal{B}^{\A}_{\Theta(\Sigma)}(\phi). \]
 Since $|c(\Sigma)|  = |c(\Theta(\Sigma))|$, the desired result follows from the local relative character identity of Theorem \ref{T:main}.
\end{proof}
\vskip 10pt

 \vskip 10pt
 
\subsection{\bf End remarks}
  We end this  paper with some comparisons with the relative trace formula approach.  The spectral side of a 
relative trace formula is essentially a sum of the relevant global relative characters over all cuspidal representations. One then hopes to separate the different spectral contributions by using the action of the spherical Hecke algebra at almost all places.   The main global output of a comparison of (the geometric side of) two such relative trace formulae  
is typically a global relative character identity as in Theorem \ref{T:global2}, as a consequence of which one deduces Proposition \ref{P:gperiod} and the local relative character identities in Theorem \ref{T:main}, which in turn implies Proposition \ref{P:smoothp}.  It is interesting to compare this with the approach  via theta correspondence which we have pursued in this paper.


\vskip 15pt



\begin{thebibliography}{99}
\bibitem{AGRS} A. Aizenbud, D. Gourevitch, S. Rallis and O. Schiffmann, {\em Multiplicity one theorems}, Ann. of Math. (2) 172 (2010), no. 2, 1407-1434. 

\bibitem{BLM} M. Baruch, E. Lapid and Z.Y. Mao, {\em A Bessel identity for the theta correspondence}, Israel J. Math. 194 (2013), no. 1, 225-257.

\bibitem{B}  J. N. Bernstein, {\em On the support of Plancherel measure}, J. Geom. Phys. 5 (1988), no. 4, 663-710 (1989). 

\bibitem{BP1} R. Beuzart-Plessis, {\em A local trace formula for the Gan-Gross-Prasad conjecture for unitary groups:
the archimedean case}, to appear in Asterisque.

\bibitem{BP2} R. Beuzart-Plessis, {\em Plancherel formula for $\GL_n(F)\backslash \GL_n(E)$  and applications to the Ichino-Ikeda and formal degree conjectures for unitary groups}, arXiv:1812.00047.  

\bibitem{Bor} A. Borel, {\em Introduction to automorphic forms}, in {\em Algebraic Groups and Discontinuous Subgroups},  Proc. Sympos. Pure Math., Boulder, Colo., (1965), 199–210 Amer. Math. Soc., Providence, R.I.

 \bibitem{D} P. Delorme, {\em Formule de Plancherel pour les fonctions de Whittaker sur un groupe r\'eductif p-adique}, Ann. Inst. Fourier (Grenoble) 63 (2013), no. 1, 155-217.
 
\bibitem{G} W.T. Gan, {\em Periods and theta correspondence}, Proceedings of Symposia in Pure Mathematics
Volume 101 (2019), 113-132.
 
 \bibitem{G2} W.T. Gan, {\em Doubling zeta integrals and local factors for metaplectic groups}, Nagoya Math. J. 208 (2012), 67-95.
 
\bibitem{GG} W.T. Gan and R. Gomez, {\em A conjecture of Sakellaridis-Venkatesh on the unitary spectrum of spherical varieties}, in {\em Symmetry: representation theory and its applications}, 185-226, Progr.  Math., 257, Birkhauser/Springer,  New York, 2014. 

\bibitem{GQT} W.T. Gan, Y.N. Qiu and S. Takeda, {\em The regularized Siegel-Weil formula (the second term identity) and the Rallis inner product formula}, Invent. Math. 198 (2014), no. 3, 739-831.

\bibitem{GS} W.T. Gan and G. Savin, {\em Representations of metaplectic groups I: epsilon dichotomy and local Langlands correspondence}, Compositio Math. 148 (2012), no. 6, 1655-1694.
 
\bibitem{GT}  W.T. Gan and S. Takeda, {\em A proof of the Howe duality conjecture.}, J. Amer. Math. Soc. 29 (2016), no. 2, 473-493. 

\bibitem{GJ} S. Gelbart and H. Jacquet, {\em A relation between automorphic representations of $\GL(2)$ and $\GL(3)$},  Ann. Sci. \'Ecole Norm. Sup. (4) 11 (1978), no. 4, 471-542.
 
 \bibitem{GZ} R. Gomez and C. B. Zhu, {\em Local theta lifting of generalized Whittaker models associated to
nilpotent orbits}, Geom. Funct. Ana. 24, (2014), 796-853.
 
\bibitem{He} H.Y. He, {\em Theta correspondence. I. Semistable range: construction and irreducibility}, Commun. Contemp. Math. 2 (2000), no. 2, 255-283.
  
\bibitem{H} R. Howe, {\em On some results of Strichartz and Rallis and Schiffman},  J. Funct. Anal. 32 (1979), no. 3, 297-303. 

\bibitem{JK} D. Johnstone and R. Krishna, {\em Beyond endoscopy for spherical varieties of rank one: the fundamental lemma}, in preparation.


\bibitem{LM}  E. Lapid and Z.Y. Mao, {\em A conjecture on Whittaker-Fourier coefficients of cusp forms}, J. Number Theory 146 (2015), 448-505.

\bibitem{LR} E. Lapid and S. Rallis, {\em On the local factors of representations of classical groups}, in {\em Automorphic representations, L-functions and applications: progress and prospects}, 309-359, Ohio State Univ. Math. Res. Inst. Publ., 11, de Gruyter, Berlin, 2005.

\bibitem{Li} J.-S. Li, {\em Singular unitary representations of classical groups},  Invent. Math. 97 (1989), no. 2, 237-255.

\bibitem{MR} Z.Y. Mao and S. Rallis, {\em  Jacquet modules of the Weil representations and families of relative trace identities}, Compos. Math. 140 (2004), no. 4, 855-886.

\bibitem{MVW}  C. Moeglin, M.-F. Vigneras and J.-L. Waldspurger, {\em Correspondances de Howe sur un corps p-adique}, 
Lecture Notes in Mathematics, 1291. Springer-Verlag, Berlin, 1987. viii+163 pp.

\bibitem{R} S. Rallis, {\em Langlands' functoriality and the Weil representartion}, American Journal of Math., Vol. 104 (1982), No. 3,  469-515.

\bibitem{R2} S. Rallis, {\em On the Howe duality conjecture},  Compositio Math. 51 (1984), no. 3, 333-399.

\bibitem{R3} S. Rallis, {\em L -functions and the oscillator representation.}, Lecture Notes in Mathematics, 1245. Springer-Verlag, Berlin, 1987, 239 pp.

\bibitem{Rama} D. Ramakrishnan, {\em Modularity of the Rankin-Selberg L-series, and multiplicity one for $\SL(2)$},  Ann. of Math. (2) 152 (2000), no. 1, 45-111.

\bibitem{St} R.S. Strichartz, {\em Harmonic analysis on hyperboloids},  J. Functional Analysis 12 (1973), 341-383. 

\bibitem{S1} Y. Sakellaridis, {\em  On the unramified spectrum of spherical varieties over p-adic fields},  Compos. Math. 144 (2008), no. 4, 978-1016. 

\bibitem{S2} Y. Sakellaridis, {\em Spherical functions on spherical varieties},  Amer. J. Math. 135 (2013), no. 5, 1291-1381. 

\bibitem{S3} Y. Sakellaridis, {\em Plancherel decomposition of Howe duality and Euler factorization of automorphic functionals}, in {\em Representation theory, number theory, and invariant theory}, 545-585,  Progr. Math., 323, Birkhauser/Springer, Cham, 2017. 

\bibitem{S4} Y. Sakellaridis, {\em Functorial transfer between relative trace formulas in rank one}, arXiv 1808.09358. 

\bibitem{S5} Y. Sakellaridis, {\em Transfer operators and Hankel transforms between relative trace formulas, I: character theory}, arXiv 1804.02383.

\bibitem{S6} Y. Sakellaridis, {\em Transfer operators and Hankel transforms between relative trace formulas, II: Rankin--Selberg theory}, arXiv 1805.04640.

\bibitem{SV} Y. Sakellaridis and A. Venkatesh, {\em Periods and harmonic analysis on spherical varieties},  Ast\'erisque No. 396 (2017), viii+360 pp.

\bibitem{W} J.-L. Waldspurger, {\em  La formule de Plancherel pour les groupes $p$-adiques (d'apr\`es Harish-Chandra)}, J. Inst. Math. Jussieu 2 (2003), no. 2, 235-333. 

\bibitem{Wa1} N. Wallach, {\em Real reductive groups. I}, Pure and Applied Mathematics, 132. Academic Press, Inc., Boston, MA, 1988. xx+412 pp.

\bibitem{Wa2} N. Wallach, {\em Real reductive groups. II},  Pure and Applied Mathematics, 132-II. Academic Press, Inc., Boston, MA, 1992. xiv+454 pp.

\bibitem{Wan} X.L. Wan, {\em The Sakellaridis-Venkatesg conjecture for ${\rm U}_2 \backslash {\rm SO}_5$}, PhD thesis, National University of Singapore.

\bibitem{X}  H. Xue,  {\em Arithmetic theta lifts and the arithmetic Gan-Gross-Prasad conjecture for unitary groups}, Duke Math. J. 168 (2019), no. 1, 127-185.

 

\bibitem{Z} C.B. Zhu, {\em Vanishing and non-vanishing in local theta correspondence}. Proceedings of Symposia in Pure Mathematics
Volume 101 (2019).
\end{thebibliography}
\end{document}